\documentclass[11pt,reqno]{amsart}  
\usepackage{amsmath,amssymb,amsthm}
\theoremstyle{plain}
\newtheorem{theorem}{Theorem}[section]
\newtheorem{lemma}[theorem]{Lemma}
\newtheorem{remark}[theorem]{Remark}
\newtheorem{proposition}[theorem]{Proposition}

\newtheorem{corollary}[theorem]{Corollary}

\numberwithin{equation}{section}
\allowdisplaybreaks[1]

\theoremstyle{definition}

\newtheorem{definition}[theorem]{Definition}

\theoremstyle{remark}

\newcommand{\bA}{{\mathbf A}}
\newcommand{\bB}{{\mathbf B}}
\newcommand{\bC}{{\mathbf C}}
\newcommand{\bD}{{\mathbf D}}
\newcommand{\bU}{{\mathbf U}}
\newcommand{\bomega}{{\boldsymbol \omega}}
\newcommand{\bTheta}{{\boldsymbol \Theta}}

\newcommand{\bcO}{\boldsymbol{\mathcal O}}
\newcommand{\bT}{{\mathbf T}}

\newcommand{\bbeta}{{\boldsymbol \beta}}

\newcommand{\cA}{{\mathcal A}}

\newcommand{\cD}{{\mathcal D}}
\newcommand{\cE}{{\mathcal E}}

\newcommand{\cG}{{\mathcal G}}
\newcommand{\cH}{{\mathcal H}}

\newcommand{\cL}{{\mathcal L}}
\newcommand{\cM}{{\mathcal M}}
\newcommand{\cN}{{\mathcal N}}
\newcommand{\cO}{{\mathcal O}}

\newcommand{\cU}{{\mathcal U}}

\newcommand{\cX}{{\mathcal X}}
\newcommand{\cY}{{\mathcal Y}}

\newcommand{\D}{{\mathbb D}}

\newcommand{\sbm}[1]{\left[\begin{smallmatrix} #1
		\end{smallmatrix}\right]}

\newcommand{\Ob}{\boldsymbol{{{\mathfrak O}}}}
\newcommand{\Gr}{{\boldsymbol{\mathfrak G}}}

\begin{document}

\title[Weighted Hardy spaces]
{Weighted Hardy spaces: shift invariant and coinvariant subspaces, linear systems and 
operator model theory}
\author[J. A. Ball]{Joseph A. Ball}
\address{Department of Mathematics,
Virginia Tech,
Blacksburg, VA 24061-0123, USA}
\email{joball@math.vt.edu}
\author[V. Bolotnikov]{Vladimir Bolotnikov}
\address{Department of Mathematics,
The College of William and Mary,
Williamsburg VA 23187-8795, USA}
\email{vladi@math.wm.edu}

\dedicatory{In memory of Bela Sz.-Nagy, a fine mathematician and 
leading operator theorist}
\thanks{The second author's research was supported by the Plumeri Award of the
College of William and Mary}

\begin{abstract}
         The Sz.-Nagy--Foias model theory for  $C_{\cdot 0}$ 
	 contraction operators combined with the Beurling-Lax theorem 
	 establishes a correspondence between any two of four kinds of 
	 objects:  shift-invariant subspaces, operator-valued inner 
	 functions, conservative discrete-time input/state/output 
	 linear systems, and $C_{\cdot 0}$ Hilbert-space contraction 
	 operators.  We discuss an analogue of all these ideas in the 
	 context of weighted Hardy spaces over the unit disk and an 
	 associated class of hypercontraction operators.
 \end{abstract}


\maketitle

\section{Introduction}  \label{S:Intro}
\setcounter{equation}{0}

A couple of seminal developments in nonselfadjoint operator theory in the middle 
part of the last century was the Sz.-Nagy dilation theorem and the 
Sz.-Nagy--Foias model theory (we refer to the second edition of the 
Sz.-Nagy--Foias monograph \cite{NF} (with additional authors H.~Bercovici and 
L. Kerchy) for a comprehensive treatment which includes a thorough 
discussion of later related developments).   The Sz.-Nagy dilation 
theorem asserts that any Hilbert-space contraction operator can be 
lifted to a coisometry (as well as dilated to a unitary operator), 
while the Sz.-Nagy--Foias model theory associates with any completely 
nonunitary contraction operator $T$ a characteristic function 
$\Theta_{T}$ which is a contractive analytic  operator-valued 
function between two coefficient Hilbert spaces $\cU$ and $\cY$. The 
characteristic function $\Theta = \Theta_{T}$ has
the additional property that it is {\em pure} in the sense that there are no nonzero 
subspaces $\cU_{0}$ and  $\cY_{0}$ of $\cU$ and $\cY$ respectively 
such that $\Theta(z)|_{\cU_{0}}$ reduces to a constant unitary operator from 
$\cU_{0}$ onto $\cY_{0}$.  Conversely, starting with any pure contractive 
analytic function $(\Theta(z), \cU, \cY)$, there is a functional 
model Hilbert space $\cH(\Theta)$ and a completely 
nonunitary canonical-model contraction operator $T = T(\Theta)$ 
acting on $\cH(\Theta)$ so that $T$ and $T(\Theta)$ are unitarily 
equivalent in case $\Theta = \Theta_{T}$.  An important motivating special 
case is the $C_{\cdot 0}$ case where the characteristic function 
 is closely entangled with the Beurling-Lax-Halmos (or 
simply Beurling-Lax for short) theorem 
associating an inner function with a shift-invariant subspace of 
$H^{2}$ (operator-valued inner function for the case of a 
higher-multiplicity shift acting on $H^{2}(\cY): = H^{2} \otimes \cY$ for a 
coefficient Hilbert space $\cY$); see \cite{Beurling, Lax, Halmos}. 
Indeed, the Sz.-Nagy--Foias model theory for the $C_{\cdot 0}$ case 
can be summed up as follows:   roughly, apart from manageable 
degeneracies, {\em there is a one-to-one correspondence between four 
kinds of objects:  (1) shift-invariant subspaces $\cM \subset H^{2}(\cY)$, (2) 
operator-valued inner functions $\Theta$, (3) unitary colligation 
matrices $\left[ \begin{smallmatrix} A & B \\ C & D \end{smallmatrix} 
\right] \colon \left[ \begin{smallmatrix}  \cX \\ \cU  
\end{smallmatrix} \right] \to \left[ \begin{smallmatrix} \cX \\ \cY 
\end{smallmatrix} \right]$,
and (4) $C_{\cdot 0}$-contraction operators $T$.}
Note that the Beurling-Lax theorem gives the correspondence between a 
shift-invariant subspace $\cM \subset H^{2}(\cY)$ and an inner 
function $\Theta$ via $\cM = \Theta \cdot  H^{2}(\cY)$. Given an 
inner function $\Theta$, the Sz.-Nagy-Foias theory tells us that 
$\Theta$ coincides with the characteristic function $\Theta_{T}$ for 
the $C_{\cdot 0}$ contraction operator $T: = P_{\cM^{\perp}} 
M_{z}|_{\cM^{\perp}}$.
The formula for a characteristic function then leads to a realization 
for $\Theta = \Theta_{T}$ of the form
\begin{equation}   \label{clas-real}
  \Theta(z) = D + z C (I - zA)^{-1} B
\end{equation}
with $U = \left[ \begin{smallmatrix} A & B \\ C & D 
\end{smallmatrix} \right]$ unitary.  Conversely, if $U = 
\left[ \begin{smallmatrix} A & B \\ C & D \end{smallmatrix} \right]$ 
is unitary with $A$ of class $C_{0 \cdot}$, one can verify that 
$\Theta(z)$ given by \eqref{clas-real} is inner.
Associated with a shift-invariant subspace $\cM \subset H^{2}(\cY)$ is the contraction 
operator $T(\cM^{\perp}) : = P_{\cM^{\perp}} M_{z}|_{\cM^{\perp}}$.  Given a $C_{\cdot 
0}$-contraction operator $T$, the Sz.-Nagy--Foias characteristic 
function $\Theta_{T}$ provides an inner function so that we recover 
$T$ up to unitary equivalence in the form $T = P_{\cM} M_{z}|_{\cM}$ 
with $\cM = \Theta_{T} \cdot \cU$.  In short, starting with an object 
of any one of the four types listed above, there is a recipe for 
passing to its equivalent representative in any one of the three 
remaining types.

\smallskip

It should be mentioned 
that related notions of characteristic function and associated 
operator model theory appeared in the work of Liv\v{s}ic and the 
Ukrainian operator-theory school (see \cite{Brodskii, Liv1, Liv2}) 
as well as in the work of de Branges and Rovnyak (see \cite{dBR1, dBR2}).  The 
characteristic function also appears in other guises, namely, as the 
scattering function in the setting of Lax-Phillips scatterings (see 
\cite{LP, AA, Helton-scat, NV1, NV2}) where there is also a close 
connection with the geometry of the Sz.-Nagy--Foias dilation space, 
as well as the transfer function of a conservative discrete-time 
linear system (or the scattering function for a lossless LCR circuit) 
(see \cite{Helton}).  

\smallskip

While the Sz.-Nagy-Foias model theory in its 
original form had a tremendous influence and applications for the 
theory of a single contraction operator on a Hilbert space, we focus 
here on extensions to more general settings.  Also we focus only on 
the aspects of dilation theory, characteristic function and 
associated operator model theory; this leaves out another key 
and influential component of the theory, namely the Commutant Lifting 
Theorem which has also seen lots of extensions to more general 
settings (see e.g.~\cite{FFGK, MS99}).   There were first obtained 
compelling extensions of the Sz.-Nagy dilation theory to 
classes other than contractions, including multivariable versions 
involving commutative operator-tuples rather than a single operator, 
in the work of Agler \cite{Ag1, Ag2}, M\"uller \cite{Muller}, 
M\"uller-Vasilescu \cite{MV}, Athavale \cite{Athavale}, Curto-Vasilescu \cite{CV1, CV2}, 
Pott \cite{Pott}, Ambrozie-Engli\v{s}-M\"uller \cite{AEM}, and 
Arazy-Engli\v{s} \cite{AE}.  An early identification of a general 
operator-algebra setting for dilation theory which indeed influenced 
some of the work mentioned above was achieved by 
Arveson \cite{Arv, ArvII}.

\smallskip

The general settings for which there has also been identified a 
characteristic function to this point are more limited compared to 
those where a dilation theory exists.  There was early work of 
Abrahamse-Douglas \cite{AD1, AD2} and Ball \cite{BallC00} and 
continuing with \cite{LKMV, BV1, BV2} where function theory on the 
disk is replaced by function theory on a finitely-connected planar 
domain (or more generally bordered Riemann surface of dividing type);
this work is also closely tied up with the appropriate notion of a 
Beurling-Lax theorem for this setting.  A more complete analogue of 
the whole Sz.-Nagy--Foias model theory is the extension to the 
setting where the single operator $T$ is replaced by a freely 
noncommutative row contraction. Here we say that the $d$-tuple 
$T = (T_{1}, \dots, T_{d})$ of operators on the Hilbert space $\cH$ is 
a {\em row contraction} if the row matrix 
$\begin{bmatrix} T_{1} & \cdots & T_{d} \end{bmatrix}$ is 
contractive as an operator from the direct sum space $\cH^{d} = 
\bigoplus_{j=1}^{d} \cH$ into $\cH$; we refer to the work of Bunce, Frazho, 
and Popescu \cite{Bunce, Frazho, Popescu1, Popescu3} for the dilation 
theory aspects and Popescu \cite{Popescu2} for the characteristic 
function aspects; the work of Ball-Bolotnikov-Fang \cite{BBF1} drew 
out the system-theory aspects while that of Ball-Vinnikov \cite{BVMemoir} extended 
these results from the the completely non-coisometric setting to the 
general completely nonunitary setting.  
There has also been work (see e.g.~the work of Muhly-Solel 
\cite{MS05}) extending the Sz.-Nagy--Foias model theory to 
more abstract operator-algebra settings.  There is a parallel 
dilation theory, characteristic function, model theory, and 
Beurling-Lax theorem for the 
case of a commutative row contraction (see the work of Drury 
\cite{Drury}, Arveson \cite{Arv1998}, Bhattacharyya-Eschmeier-Sarkar 
\cite{BES1, BES2}, Ball-Bolotnikov \cite{BBmodelball}, 
McCullough-Trent \cite{McCT2000}), as well as 
more flexible settings simultaneously containing the freely 
noncommutative case and the commutative case \cite{BhatBhat, Popescu4}.
There is also an operator model theory and a version of the 
characteristic function for the setting where the single contraction 
operator is replaced by a family of contraction operators $T_{n} 
\in \cL(\cH_{n+1}, \cH_{n})$ ($n= \dots -1,0,1, \dots$) and the 
characteristic function is the input/output map of a conservative 
time-varying linear system (see the 
papers of Constantinescu \cite{Constantinescu1, Constantinescu2} 
and Alpay-Ball-Peretz \cite{ABP}).  We also mention that 
the Hilbert-module setting for the 
Sz.-Nagy--Foias model theory was pursued in \cite {DP, MS99, DKKS}.

\smallskip

Generally speaking, a distinguishing feature of the cases where the characteristic 
function appears versus the cases where there is only a dilation 
theory and associated model theory without a characteristic-function 
invariant is that the associated positive kernel is not of the type 
now called a Pick kernel (see e.g.~\cite{AgMcC} for the terminology).
A first step away from this restriction was in the work of Olofsson 
\cite{oljfa, olaa} who introduced a characteristic function for the 
class of $n$-hypercontractions, i.e., operators $T$ for which 
$\sum_{k=0}^{m} (-1)^{k} \binom{m}{k} T^{*k}T^{k} \ge 0$ for $1 \le m 
\le n$.  The class of $n$-hypercontractions is closely tied to the 
function theory for the weighted Bergman space $\cA_{n}(\cY)$ over the 
unit disk (where $\cY$ is a coefficient Hilbert space) having 
operator-valued reproducing kernel equal to $k_{n}(z, \zeta) \cdot I_{\cY} : = 
\frac{1}{(1 - z \zeta)^{n}} I_{\cY}$.  A Beurling-Lax theorem for the 
Bergman space setting has been of interest to the function-theoretic 
operator-theory community since the 1970s but has turned out to be  
much more difficult to come by; the results obtained are necessarily 
of a more delicate form, with the most progress just since the 1990s.
There eventually evolved a notion of Bergman 
inner function to be a function $\Theta$ which maps a coefficient 
Hilbert space isometrically onto a {\em wandering subspace} $\cE 
\subset \cA_{n}(\cY)$ for the 
Bergman shift operator $S_{n}$ equal to multiplication by the 
coordinate function $M_{z} \colon f(z) \mapsto z f(z)$ on 
$\cA_{n}(\cY)$.
Here $\cE$ is a  {\em wandering subspace} means only 
that $\cE$ is orthogonal to $S_{n}^{k} \cE$ for all $k > 0$; a key 
distinction from the Hardy space case is that it does not follow that $S_{n}^{k} \cE$ is 
orthogonal to $S_{n}^{k'}\cE$ for distinct positive integers $k,k'$.
Then at least we get a Beurling-like representation $\cM$ as the closure of 
$\Theta \cU[z]$ (where $\cU[z]$ is the linear space of polynomials 
with coefficients in $\cU$ and $\Theta$ is a Bergman inner function 
with $\Theta \cdot \cU$ equal to the wandering subspace $\cE = \cM 
\ominus S_{n} \cM$). 
Bergman inner functions were first constructed by Hedenmalm as the solution of an 
extremal problem (see \cite{heden1, heden2}).  The biharmonic Green function 
was introduced shortly thereafter to prove the contractive divisor 
property in a conceptually better way extending the result to an 
$L^{p}$-setting by Duren, Khavinson, Shapiro and Sundberg (see 
e.g.~\cite{DKSS}).
When the shift-invariant subspace $\cM \subset \cA_{n}(\cY)$ is {\em 
pure} in the sense that $\cE: = \cM \ominus S_{n} \cM$ is {\em 
generating} for $\cM$ (meaning that $\cM = 
\overline{\operatorname{span}}_{k\ge 0} S_{n}^{k} \cE$), one gets at 
least a Beurling-type representation of the form $\cM = 
\overline{\operatorname{span}} S_{n}^{k} \cE$; the fact that this 
holds in general for the unweighted case $n=2$ was first proved by  
Aleman-Richter-Sundberg \cite{ars}.
Shimorin \cite{sh1, sh2} noticed that many of these ideas can be developed 
in a purely operator-theoretic setting where $S_{n}$ is replaced by a 
left-invertible Hilbert-space operator $T$. Additional analysis of the 
wandering-subspace property has been developed in \cite{HJSh, McCR, 
ol2005, Sutton}.  There are now available two monographs \cite{DS, HKZ} leading the 
reader through many of these developments.

\smallskip

The work of Olofsson offered a new direction for the computation and 
construction of Bergman  inner functions by introducing ideas from 
linear system theory whereby Bergman inner functions have a 
transfer-function-like realization 
\begin{equation}   \label{Olofsson-real}
 \Theta(z) = D + z C \left( \sum_{k=1}^{n} (I - zA)^{-k} \right) B
 \end{equation}
for a certain colligation matrix $U = \left[ \begin{smallmatrix} A & 
B \\ C & D \end{smallmatrix} \right]$ constructed explicitly 
from the invariant subspace $\cM$.  Alternatively, 
one could start with the $n$-hypercontrac\-tion $A = T^{*}$ (assuming 
that $A$ is a $C_{0 \cdot}$ $n$-hypercontraction), use the dilation 
theory of Agler \cite{Ag2} to model $A$ as the restriction of 
$S_{n}^{*}$ to an $S_{n}^{*}$-invariant subspace $\cM^{\perp} \subset 
\cA_{n}(\cY)$, and identify explicitly the Bergman-inner function 
$\Theta$ associated with the $S_{n}$-invariant subspace $\cM$ 
as the {\em characteristic function} of $T$.  The explicit formula 
of the type \eqref{Olofsson-real} for $\Theta_{T}$ is very much like 
the Sz.-Nagy--Foias formula for the characteristic function for a 
contraction operator $T$, but now one must work with certain 
$n$-level generalized defect operators $D_{n,T}$ and $D_{n,T^{*}}$ 
rather than the standard Sz.-Nagy--Foias defect operators $D_{T} = (I 
- T^{*}T)^{1/2}$ and $D_{T^{*}} = (I - T T^{*})^{1/2}$ in the 
Sz.-Nagy--Foias theory.

\smallskip

Our own paper \cite{BBberg} followed up on this work of 
Olofsson by drawing out further the system-theory aspects and 
introducing several alternate Beurling-Lax-type representations for an 
$S_{n}$-invariant subspace $\cM$ of $\cA_{n}(\cY)$.  The present 
paper extends the work of \cite{BBberg} in two respects:
(1) we replace the Bergman space $\cA_{n}(\cY)$  with a more general 
weighted Hardy space $H^{2}_{\bbeta}(\cY)$ described below, and (2) in 
addition to Beurling-Lax representation theorems, we here explicitly 
define a characteristic function (more precisely, characteristic function 
family) for a $\bbeta$-$C_{\cdot 0}$ $*$-$\bbeta$-hypercontraction 
operator $T$ on a Hilbert space $\cX$ and obtain a complete 
Sz.-Nagy--Foias dilation and model theory  for this class of operators.

\smallskip

In detail, the class of weighted Hardy spaces which we consider is as follows.
Given a sequence $\bbeta=\{\beta_j\}_{j\ge 0}$ of positive numbers, the {\em 
weighted Hardy space} 
$H^2_\bbeta$ is defined as the set of all functions analytic on the  open unit
disk $\D$ and with finite norm $\|f\|_{H^2_{\bbeta}}$ given by
$$
\|f\|_{H^2_{\bbeta}}^2=\sum_{j=0}^\infty \beta_j |f_j|^2\quad\mbox{if}\quad
f(z)=\sum_{j=0}^\infty  f_jz^j.
$$
Polynomials are dense in  $H^2_\bbeta$ and  the monomials $\{z^k\}_{k\ge 0}$  
form an orthogonal set
uniquely defining the weight sequence $\bbeta$  by $\beta_j=\|z^j\|^2$ for $j\ge 
0$.  A general reference for such spaces and the associated weighted shift operators 
is the article of
Shields \cite{Shields}.

\smallskip 

For a Hilbert space $\cY$, we denote by $H^2_{\bbeta}(\cY)$ the
tensor product Hilbert space $H^2_{\bbeta}\otimes \cY$ which can be identified as
\begin{equation}
H^2_{\bbeta}(\cY)=\left\{f(z)={\displaystyle\sum_{k\ge 0}f_kz^k}: \;
\|f\|^2_{H^2_{\bbeta}(\cY)}:={\displaystyle \sum_{k\ge 0}\beta_k \cdot
\|f_k\|_{\cY}^2}<\infty\right\}.
\label{1.1}
\end{equation}
If ${\displaystyle \liminf\beta_j^{\frac{1}{j}}}\ge 1$, then the power series
\begin{equation}
R_\bbeta(z)=\sum_{j=0}^\infty \beta_j^{-1}z^j
\label{1.2}
\end{equation}
converges on the open unit disk $\D$. The function 
$$
K_{\bbeta}(z,\zeta)=R_\bbeta(z\overline{\zeta})=\sum_{j=0}^\infty   
\beta_j^{-1}\cdot z^j\overline{\zeta}^j
$$
turns out to be the reproducing kernel for $H^2_{\bbeta}$ in the sense that 
$z\mapsto K_\bbeta(z,\zeta)$
belongs to  $H^2_{\bbeta}$ and the equality $\langle f, \, 
K_\bbeta(\cdot,\zeta)\rangle_{H^2_\bbeta}=f(\zeta)$
holds for every $\zeta\in\D$ and $f\in H^2_\bbeta$.

\smallskip

We will write ${\bf 1}$ for the sequence $\bbeta$ with  $\beta_j=1$ for all $j\ge 
0$. The space
$H^2_{\bf 1}(\cY)$ is the classical vector Hardy space $H^2(\cY)$
of the unit  disk. Another important example is given by the weight sequence 
$\bbeta_\alpha=\{\beta_{\alpha,k}\}_{k\ge 0}$ with 
\begin{equation}
\beta_{\alpha,k}=\frac{k!}{\alpha(\alpha+1)\cdots 
(\alpha+k-1)}=\frac{k!\Gamma(\alpha)}{\Gamma(\alpha+k)}
\label{1.4}
\end{equation}
for any fixed $\alpha>1$. The space $H^2_{\bbeta_\alpha}$ equals the Bergman 
space $A^2_{\alpha-2}$
of $\cY$-valued functions $f$ analytic on $\D$ and such that  
$$
\|f\|^2_{A^2_{\alpha-2}}=(\alpha-1)\int_{\D} 
\|f(z)\|^2_{\cY}(1-|z|^2)^{\alpha-2}dA(z)<\infty
$$
where $dA$ is the planar Lebesgue measure normalized so that $A(\D)=1$.

\smallskip

The shift operator $S_\bbeta$ on  $H^2_\bbeta$ is  defined by 
$S_\bbeta: \, f(z)\mapsto zf(z)$ and simple 
inner-product calculations show that its adjoint $S_\bbeta^*$ is given by
\begin{equation}
S^*_\bbeta f=\sum_{k=0}^\infty\frac{\beta_{k+1}}{\beta_k}\cdot 
f_{k+1}z^k\quad\mbox{if}\quad
f(z)=\sum_{k=0}^\infty f_kz^k.
\label{1.5}   
\end{equation}
In this paper we will be primarily interested in subspaces of $H^2_\bbeta$ which 
are invariant either under 
$S_\bbeta$ or under $S_\bbeta^*$. For the rest of the paper, we assume that the 
weight sequence 
$\bbeta=\{\beta_j\}_{j\ge 0}$ meets the following conditions. Firstly we assume 
that 
\begin{equation}
\liminf\beta_j^{\frac{1}{j}}=1,\quad \beta_0=1\quad\mbox{and}\quad 1\le 
\frac{\beta_j}{\beta_{j+1}}\le M \quad\mbox{for all 
$j\in{\mathbb Z}_+$}
\label{1.6}
\end{equation}
and some $M\ge 1$. The two first conditions are normalizing and thus, 
non-restrictive. It is seen from 
\eqref{1.5} that $\|S_\bbeta^*\|={\displaystyle \sup_{j\ge 
0}\frac{\beta_{j+1}}{\beta_j}}$
and thus, the third condition in \eqref{1.6} means that the shift operator 
$S_\bbeta: \, H^2_\bbeta\to 
H^2_\bbeta$ is contractive and left-invertible. Secondly, we assume that  the  
reciprocal power series
\begin{equation}
{\displaystyle\sum_{j=0}^\infty c_jz^j}:={\displaystyle\frac{1}{R_{\bbeta}(z)}}
={\displaystyle \left(\sum_{j=0}^\infty \beta_j^{-1}z^j\right)^{-1}}
\label{1.7}
\end{equation}
belongs to the Wiener class $W^+$, that is, the coefficients  
$\{c_{j}\}_{j \ge 0}$ appearing in \eqref{1.7} are absolutely 
summable:
\begin{equation}
\text{if } c_0=1 \text{ and recursively } c_n=-\sum_{j=0}^{n-1}c_j\beta_{n-j}^{-1}, 
\text{ then } \sum_{j=0}^\infty |c_j|<\infty.
\label{1.8}
\end{equation}
We remark that the sequence $\bbeta_\alpha$ defined in \eqref{1.4} meets all the 
above assumptions.  In particular, when $\alpha = n$ is a positive 
integer, then $\frac{1}{R_{\bbeta}(z)}$ is even a polynomial.
A more general example is given by the function 
$$
R_\bbeta(z)=\sum_{j=0}^\infty \beta_j^{-1}z^j=\frac{1}{(1-z)^\alpha (1-zg(z))}
$$
where $\alpha>1$ and $g$ is a function in $W^+$ with non-negative Taylor 
coefficients at the origin having no zeros  in the closed unit disk.
An open question is whether this Wiener-algebra assumption \eqref{1.8}  can be 
weakened in such a way that the results of this paper continue to 
hold.

\smallskip

The operator-model theory developed in this paper is as follows. 
We establish a correspondence between three types of objects:  (1) 
{\em $S_{\bbeta}$-invariant subspaces} $\cM \subset H_{\bbeta}(\cY)$, (2)
$\bbeta$-inner function families $\{\Theta_{k}\}_{k \ge 0}$, (3) a 
family of colligation matrices $U_{k}$ of the type considered in 
Section 2 below which satisfy additional metric constraints discussed 
in Section 5 below ({\em $\bbeta$-unitary colligation family}), and (4) 
$\bbeta$-$C_{\cdot 0}$ $*$-$\bbeta$-{\em hypercontraction operators} $T$.  
Here the term $\bbeta$-$C_{\cdot 0}$ applied to a Hilbert space 
operator $T$ is a strengthening of the 
standard notion of $C_{\cdot 0}$ as used in the book \cite{NF}
(i.e., the property that $T^{*n} \to 0$ strongly as $n \to \infty$)
tailored  to the sequence $\bbeta$ (see Definition \ref{D:betaCdot0} 
below). Associated with any  colligation family is a certain function 
family $\{\Theta_{k}\}_{k \ge 0}$ which has the interpretation as the 
transfer function for a certain time-varying input/state/output linear system
described in Section 2; when the colligation 
family is also $\bbeta$-unitary, the associated family of 
functions is an inner function family and it is natural to say that 
the associated time-varying linear system is {\em 
$\bbeta$-unitary}: this is worked out in Sections 
5 and 6.2.  This gives the correspondence between $\bbeta$-inner 
function families and $\bbeta$-unitary colligation families or 
equivalently $\bbeta$-unitary linear systems. 
The Beurling-Lax piece of the correspondence referred above is the representation of an 
$S_{\bbeta}$-invariant subspace $\cM$ in terms of a $\bbeta$-inner function 
family $\{\Theta_{k}\}_{k \ge 0}$ as $\cM = 
\bigoplus_{k=0}^{\infty} S_{\bbeta}^{k} \Theta_{k} \cdot \cU_{k}$: this is worked out in Section 
6.2 below.  Given an $S_{\bbeta}$-invariant subspace $\cM \subset 
H_{\bbeta}(\cY)$, we show in Section 4 that $P_{\cM} 
S_{\bbeta}|_{\cM}$ is a $\bbeta$-$C_{\cdot 0}$ $*$-$\bbeta$-hypercontraction.
Given any $\bbeta$-$C_{\cdot 0}$ $*$-$\bbeta$-hypercontraction $T$, in Section 
7 we associate its  $\bbeta$-inner characteristic function family
$\{\Theta_{T,k}\}_{k \ge 0}$ so that  we recover $T$ 
up to unitary equivalence as $T = P_{\cM} S_{\bbeta}|_{\cM}$ with 
$\cM = \bigoplus_{k=0}^{\infty} S_{\bbeta}^{k} \Theta_{T,k} \cU_{k}$.
In Section 7 we also show how to go directly from a given 
$\bbeta$-inner function family $\{\Theta_{k}\}_{k \ge 0}$ to a 
$\bbeta$-unitary colligation family which realizes $\{\Theta_{k}\}_{k 
\ge 0}$, i.e., we obtain explicit formulas for the colligation operator-matrices 
$U_{k}  = \left[ \begin{smallmatrix} A & B_{k} \\ C & D_{k} \end{smallmatrix} \right]$ 
constructed from the function family   $\{\Theta_{k}\}_{k \ge 0}$.

\smallskip
    
A natural follow-up project to this paper is to develop this model 
theory for $*$-$\bbeta$-hypercontractions with the $\bbeta$-$C_{\cdot 0}$ 
hypothesis removed.  For a given $*$-$\bbeta$-hypercontraction $T$, 
it does hold that $T$ being $\bbeta$-$C_{\cdot 0}$ implies that $T$ is in 
fact $C_{\cdot 0}$ but we have not been able to resolve the converse 
except for the special case where $\beta_{k} = \frac{k!(n-1)!}{(k+n-1)!}$ 
(i.e., $\bbeta$ of the form $\bbeta_{\alpha}$ as in \eqref{1.4}
with $\alpha = n$ a positive integer as studied in \cite{oljfa, olaa, 
BBberg}) (see Remark \ref{R:stable} below);  another open question for 
future work is to resolve this issue.

\smallskip

After the present Introduction, in Section 2 we give a 
time-domain system-theoretic interpretation for the class of inner 
function families coming up in Section 6.2. Section 3 presents the 
$\bbeta$-analogues of standard notions from systems theory concerning 
observability operators, observability gramians, and associated Stein 
equations and inequalities which will be needed in the sequel.  Section 4 applies 
these constructions to the model setting where the system operator is 
the restriction of the backward $\bbeta$-shift $S_{\bbeta}^{*}$ to an 
invariant subspace $\cM^{\perp} \subset H^{2}_{\bbeta}(\cY)$.  Section 5 
separates out the consequences of the metric properties associated 
with a $\bbeta$-unitary colligation family; these are used for 
the construction of the  $\bbeta$-inner function family 
Beurling-Lax representer for an  $S_{\bbeta}$-invariant subspace $\cM 
\subset H^{2}_{\bbeta}(\cY)$ in Section 6.2.  
Section 6.1 and 6.3 discuss other types of Beurling theorems in 
parallel to those developed in our earlier work 
\cite{BBberg} for the special case $\beta_{j} = \frac{j! 
(n-1)!}{(j+n-1)!}$ for a positive integer $n$; in particular, the 
Beurling representation theorem in Section 6.3 corresponds to that of 
Olofsson in \cite{olaa}.   The final Section 7 defines the 
characteristic function family and develops the 
operator-model theory for the class of $\bbeta$-$C_{\cdot 
0}$, $*$-$\bbeta$-hypercontraction operators.  There also is given the 
functional-model form for the colligation matrix associated with a 
given $\bbeta$-inner function family $\{\Theta_{k}\}_{k \ge 0}$. 

\smallskip

Finally it is a pleasure to thank the anonymous referee for a 
thorough reading and review of the  manuscript  which led to a number of 
improvements in the final version.

\section{System-theoretic motivation}
\label{STM}   

For $\cU$ and $\cY$ any pair of Hilbert spaces, we use the notation
$\cL(\cU, \cY)$ to denote the space of bounded, linear operators
from $\cU$ to $\cY$.  For $\cX$ a single Hilbert space, we shorten
the notation $\cL(\cX, \cX)$ to $\cL(\cX)$. Let $\bbeta$ be a given weight 
sequence, 
let $\cX$, $\cY$ and $\cU_k$ ($k\ge 0$) be Hilbert spaces, let
$$
A \in \cL(\cX), \; \; C \in \cL(\cX, \cY), \; \;
B_k \in \cL(\cU_k, \cX), \; \; D_k \in \cL(\cU_k, \cY)
$$ be bounded linear operators, and let us consider
the associated discrete-time time-variant linear system
 \begin{equation}  \label{2.1}
 \Sigma_{\bbeta}: \quad    \left\{ \begin{array}{rcl}
 x(j+1) & = & {\displaystyle\frac{\beta_j}{\beta_{j+1}}}\cdot 
Ax(j)+{\displaystyle\frac{1}{\beta_{j+1}}}\cdot B_ju(j), \\ 
[3mm]
 y(j) & = & C x(j)+{\displaystyle\frac{1}{\beta_j}}\cdot D_ju(j)
 \end{array} \right.
 \end{equation}
with $x(j)$ taking values in the {\em state space} $\cX$, $u(j)$ taking
values in the {\em input-space} $\cU_j$ and $y(j)$ taking values in the
{\em output-space} $\cY$. If we let the system evolve on the
nonnegative  integers $j \in {\mathbb Z}_{+}$, then the whole
trajectory $\{u(j), x(j), y(j)\}_{j \in {\mathbb Z}_{+}}$ is
determined from the input signal $\{u(j)\}_{j\ge 0}$
and the initial state $x(0)$ according to the formulas
\begin{align}
x(j) &= \beta_j^{-1}\cdot \left(A^{j} x(0)+\sum_{\ell=0}^{j-1}A^{j-\ell-1} B_\ell
u(\ell)\right),\label{2.2}\\  
y(j) &= \beta_j^{-1}\cdot \left(CA^{j} x(0)+\sum_{\ell=0}^{j-1}CA^{j-\ell-1} 
B_\ell u(\ell)
+D_ju(j)\right).\label{2.3}
  \end{align}
Formula \eqref{2.2} is established by simple induction arguments, while 
\eqref{2.3} is obtained by
straightforward substituting of \eqref{2.2} into the second equation in \eqref{2.1}.
The integral form of the system equations \eqref{2.1} is the map
from initial-state/input signal to state trajectory/output signal
$$
T(\Sigma_{\bbeta}) = \begin{bmatrix}  T(\Sigma_{\bbeta})_{11} &  
T(\Sigma_{\bbeta})_{12}  \\ T(\Sigma_{\bbeta})_{21} & 
T(\Sigma_{\bbeta})_{22} \end{bmatrix} \colon \begin{bmatrix} x(0) \\ 
\{u(j)\}_{j \ge 0} \end{bmatrix} \mapsto \begin{bmatrix} 
\{x(j)\}_{j>0} \\ \{y(j)\}_{j \ge 0} \end{bmatrix}
$$
determined by solving the system equations \eqref{2.1} recursively.
Our main interest here will be only in the two pieces
\begin{equation}   \label{observe}
\bcO_{\bbeta} : = T(\Sigma_{\bbeta})_{21} 
\colon x(0) \mapsto 
\{y(j)\}_{j \ge 0},
\end{equation}
the map from initial state to output signal induced by setting the 
input signal equal to 0 and usually called the {\em observation map}, 
and the map 
\begin{equation}  \label{IOmap}
\bT_{\bbeta} \colon = T(\Sigma_{\bbeta})_{22} \colon \{u(j)\}_{j \ge 0} 
\mapsto \{ y(j)\}_{j \ge 0}
\end{equation}
from input signal to output signal determined by setting the initial 
state equal to 0, usually called the {\em input-output map} for the system 
$\Sigma_{\bbeta}$.  From \eqref{2.3} we see that $\cO_{\bbeta}$ and 
$T_{\bbeta}$  have the explicit matrix representations
$$ \bcO_{\bbeta} = [ \beta_{i}^{-1} C A^{i}]_{i \ge 0}, \quad
[\bT_{\bbeta}]_{i,j} = \begin{cases} 0 & \text{for } i<j, \\
\beta_{i}^{-1} D_{i}& \text{for } i=j, \\
\beta_{i}^{-1} C A^{i-1-j}B_{j} & \text{for } i>j, \quad  0 \le i,j. \end{cases} 
$$

Let us introduce the $Z$-transformed input, state and output signals
 $$ 
 \widehat u(z)  = \sum_{k=0}^{\infty} u(k) z^{k}, \quad
 \widehat x(z) = \sum_{k=0}^{\infty} x(k) z^{k}, \quad \widehat y(z)
 = \sum_{k=0}^{\infty} y(k) z^{k}.
 $$
 (Note that $\widehat u(z)$ is merely formal since $u(j) \in \cU_{j}$ 
 and in general the input spaces are distinct linear spaces for 
 distinct indices $j$ and hence there is no ambient linear space in 
 which to take the sum.) 
To write the $Z$-transformed version of the system-trajectory formulas 
\eqref{2.2}--\eqref{2.3},
we introduce the standard backward shift operator $S_{\bf 1}^{*}$ acting on  
formal power series 
according  to
$$
        S_{\bf 1}^{*} \colon \sum_{n=0}^{\infty} a_{n} z^{n} \mapsto
        \sum_{n=0}^{\infty} a_{n+1} z^{n}.
$$
We next introduce backward shifts of the function \eqref{1.2} by letting
 \begin{equation}  \label{2.5}
R_{\bbeta,k}(z):= 
\left(S_{1}^{*k}R_\bbeta\right)(z)=\sum_{j=0}^\infty\beta_{k+j}^{-1}z^j
\end{equation}
so that $R_{\bbeta,0}(z)=R_{\bbeta}(z)$. For every Hilbert space operator 
$A\in\cL(\cX)$ 
having spectral radius at most one, we can define operator-valued functions 
\begin{equation}
 \label{2.6}
R_{\bbeta,k}(zA)=\sum_{j=0}^\infty\beta_{k+j}^{-1}A^jz^j\quad \mbox{for all}\quad 
k\ge 0
\end{equation}
defined for $z\in\D$. Multiplying both sides of \eqref{2.2} by $z^{j}$ and 
summing 
up over $j\ge 0$, we  get, on account of \eqref{2.6},
\begin{align}
\widehat x(z)&=\left(\sum_{j=0}^\infty 
\beta_j^{-1}A^jz^j\right)x(0)+\sum_{k=1}^{\infty}
\left(\sum_{j=k}^\infty\beta_j^{-1} A^{j-k}z^j\right)B_{k-1}u(k-1)\notag\\
&=R_\bbeta(zA)x(0)+\sum_{k=1}^{\infty}z^k\left(\sum_{j=0}^\infty\beta_{j+k}^{-1}
A^jz^j\right)B_{k-1}u(k-1)\notag\\
&=R_\bbeta(zA)x(0)+\sum_{k=1}^{\infty}z^kR_{\bbeta,k}(zA)B_{k-1}u(k-1)\notag\\
&=R_\bbeta(zA)x(0)+\sum_{k=0}^{\infty}z^{k+1}R_{\bbeta,k+1}(zA)B_ku(k).\notag
\end{align}
The same procedure applied to \eqref{2.3} gives
\begin{align}
\widehat y(z)&=CR_\bbeta(zA)x(0)+\sum_{k=0}^{\infty}z^k \left(\beta_k^{-1}D_k+
zCR_{\bbeta,k+1}(zA)B_k\right)u(k)\notag\\
&=\cO_{\bbeta,C,A}x(0)+\sum_{k=0}^{\infty}z^k \Theta_k(z)u(k),\label{2.8k}
\end{align}
where
 \begin{equation}
\cO_{\bbeta,C,A} \colon \; x \mapsto CR_\bbeta(zA)x=\sum_{j=0}^\infty 
(\beta_j^{-1}CA^jx) \,
z^j \label{0.6}
\end{equation}
is the {\em $\bbeta$-observability operator} (the $Z$-transform of 
the time-domain observation operator $\bcO_{\bbeta}$ \eqref{observe}) and where
 \begin{equation}
\Theta_k(z)=\beta_k^{-1}D_k+zCR_{\bbeta,k+1}(zA)B_k \qquad (k=0,1,\ldots)
\label{2.10}
\end{equation}
is a family of transfer functions encoding the $Z$-transform of 
the time-domain input-output operator $\bT_{\bbeta}$ \eqref{IOmap}:
\begin{equation}   \label{ZtransIO}
\bT_{\bbeta} \colon \{u(j)\}_{j \ge 0} \mapsto \{y(j)\}_{j \ge 0} 
\Leftrightarrow 
\sum_{j = 0}^{\infty} y(j) z^{j} = \sum_{j=0}^{\infty} \Theta_{j}(z) 
 u(j) z^{j}.
\end{equation}
Note that we recover the classical time-invariant case by setting 
$\beta_{j} = 1$ for all $j$ and by taking $B_{j} = B$ and $D_{j} = D$ 
independent of $j$;  in this  case $\Theta_{j}(z) = D + z C (I 
- zA)^{-1} B$ is independent of $j$ and the formula on the right in 
\eqref{ZtransIO} reduces to
$$
    \widehat y(z) = \Theta_0(z) \cdot \widehat u(z).
$$

These observations suggest that the following terminology will be 
useful.

\begin{definition}  \label{D:bbeta-real}
When the function family 
$\{\Theta_k\}_{k \ge 0}$ is connected with the colligation family 
$\left\{U_{k} = \left[ \begin{smallmatrix} A & B_{k} \\ C & D_{k} 
\end{smallmatrix} \right] \right\}_{k \ge 0}$ as in \eqref{2.10}, we say 
that the colligation family $\{U_{k}\}_{k \ge 0}$ provides a {\em  $\bbeta$-realization} for
the function family $\{\Theta_k\}_{k \ge 0}$, and that the function 
family $\{\Theta_{k}(z)\}_{j\ge 0}$ is the {\em $\bbeta$-transfer function family} 
for the colligation family $\{U_{k}\}_{k \ge 0}$ and the associated 
system $\Sigma_{\bbeta}$ \eqref{2.1}.
\end{definition}

\section{Observability operators and 
gramians, Stein equalities and inequalities}  
\label{S:obs}

Formula \eqref{0.6} associates with any {\em output pair} $(C, A)$
(i.e., $C\in\cL(\cX,\cY)$ and $A\in\cL(\cX)$) the $\bbeta$-observability operator
$\cO_{\bbeta,C,A}$.  In case ${\mathcal O}_{\bbeta,C,A}$ is bounded as an 
operator from $\cX$ into 
$H^2_{\bbeta}(\cY)$, we say that the pair $(C,A)$ is {\em 
$\bbeta$-output-stable}.
If $(C,A)$ is $\bbeta$-output stable, then the
{\em $\bbeta$-observability gramian}
$$
{\mathcal G}_{\bbeta,C, A}:=(\cO_{\bbeta,C, A})^{*}\cO_{\bbeta,C,A}
$$
is bounded on $\cX$ and can be represented via the series
\begin{equation}
{\mathcal G}_{\bbeta,C, A}= \sum_{k=0}^\infty \beta_k^{-1}\cdot A^{*k}C^*CA^k
\label{3.1}
\end{equation}
converging in the strong operator topology; see e.g.~\cite[Problem 120]{halmosbook}. We will also make 
use of 
the shifted versions of $\cO_{\bbeta,C,A}$ and ${\mathcal G}_{\bbeta,C, A}$ given 
by
\begin{equation}
{\Ob}^{(k)}_{\bbeta,C,A}x=\sum_{j=0}^\infty 
\beta_{j+k}^{-1}(CA^jx)z^j\quad\mbox{and}\quad
{\Gr}^{(k)}_{\bbeta,C,A}=\sum_{j=0}^\infty \beta_{j+k}^{-1}A^{*j}C^*CA^j
\label{defR}
\end{equation}
for $k\ge 0$. It follows from \eqref{defR} and formula \eqref{1.1} for the norm 
in 
$H^2_\bbeta(\cY)$ that 
\begin{equation}
{\Gr}^{(k)}_{\bbeta,C,A}=\left(S_\bbeta^k{\Ob}^{(k)}_{\bbeta,C,A}\right)^*
S_\bbeta^k{\Ob}^{(k)}_{\bbeta,C,A}.
\label{defRa}
\end{equation}
On the other hand, comparing formulas \eqref{defR} for $k=0$ with \eqref{0.6} and 
\eqref{3.1} 
gives
\begin{equation}
{\Ob}^{(0)}_{\bbeta,C,A}={\cO}_{\bbeta,C,A}\quad\mbox{and}\quad 
{\Gr}^{(0)}_{\bbeta,C,A}=\cG_{\bbeta,C,A}.
\label{st2}     
\end{equation}
\begin{proposition}
If $\beta_j/\beta_{j+1}\le M$ for all $j\ge 0$ and $(C,A)$ is a 
$\bbeta$-output-stable pair, then
\begin{equation}
\|{\Ob}^{(k)}_{\bbeta,C,A}\|\le M^k\cdot \|{\cO}_{\bbeta,C,A}\|,\qquad
{\Gr}^{(k)}_{\bbeta,C,A}\le M^k\cdot {\mathcal G}_{\bbeta,C, A}.
\label{3.4a}
\end{equation} 
Furthermore, the weighted Stein identity
\begin{equation}
A^*{\Gr}^{(k+1)}_{\bbeta,C,A}A+\beta_k^{-1}C^*C={\Gr}^{(k)}_{\bbeta,C,A}
\label{7.3}
\end{equation}
holds for all integers $k\ge 0$.
\label{P:3.0}
\end{proposition}
\begin{proof}
Indeed, since  $\beta_j/\beta_{j+1}\le M$ for all $j\ge 0$, we have 
$$
\frac{\beta_{j}}{\beta_{j+k}}\le M^{k}\quad\mbox{for all}\quad k,j\ge 0
$$
and then it follows from \eqref{0.6}, \eqref{3.1} and \eqref{defR} that
\begin{align*}
{\Gr}^{(k)}_{\bbeta,C,A}&=\sum_{j=0}^\infty 
\frac{\beta_{j}}{\beta_{j+k}}\beta_{j}^{-1}A^{j*}C^*CA^j\\
&\le M^k \sum_{j=0}^\infty \beta_{j}^{-1}A^{j*}C^*CA^j=M^k \cdot 
\cG_{\bbeta,C,A},\\
\|{\Ob}^{(k)}_{\bbeta,C,A}x\|^2&=\left\langle 
\sum_{j=0}^\infty \frac{\beta_j^2}{\beta_{j+k}^2} \beta_j^{-1}A^{j*}C^*CA^jx, \, 
x\right\rangle\\
&\le M^{2k}\cdot\left\langle  {\mathcal G}_{\bbeta,C, A}x, \, 
x\right\rangle = M^{2k}\cdot\|{\cO}_{\bbeta,C,A}x\|^2
\end{align*}
proving inequalities \eqref{3.4a}.
Equality \eqref{7.3} follows immediately from power series representation 
\eqref{defR} for 
${\Gr}^{(k)}_{\bbeta,C,A}$ and the similar one for ${\Gr}^{(k+1)}_{\bbeta,C,A}$.
\end{proof} 

\smallskip

As suggested by the
Agler hereditary functional calculus as formulated by
Ambrozie-Engli\v{s}-M\"uller \cite{AEM}, we introduce the operator
$$
B_{A} \colon X \mapsto A^{*} X A
$$
 mapping  $\cL(\cX)$ into itself, and then view $\cG_{\bbeta,C,A}$ (at
 least formally) as being given by
\begin{equation} \label{defGGamma}
\cG_{\bbeta,C,A} = R_{\bbeta}(B_A)[C^{*}C].
\end{equation}
If $\rho(A)<1$ (where $\rho(A)$ denotes the spectral radius of $A$), 
\eqref{defGGamma} is precise; in general 
one
can make this precise by interpreting \eqref{defGGamma} in the form
$$
    \cG_{\bbeta,C,A} = \lim_{r \uparrow 1} R_{\bbeta}(r B_{A})[C^{*}C].
$$

\begin{remark}\label{R:3.1}
{\rm Given an operator $A\in\cL(\cX)$ and a function $f(z)={\displaystyle 
\sum_{j=0}^\infty f_j z^j}$ in the 
Wiener algebra $W^+$, i.e., the coefficients $f_{j}$ satisfy the 
summability condition 
\begin{equation}   \label{Wiener}
    \sum_{j=0}^{\infty} |f_{j}| < \infty,
\end{equation}
the operatorial map 
$$
f(B_A):=\sum_{j=0}^\infty f_jB_A^j: \;  X \mapsto \sum_{j=0}^\infty f_j\, 
A^{*j}XA^j
$$
is well defined  for any $X\in\cL(\cX)$ subject to inequalities
\begin{equation}   \label{2.3a}
X\ge A^*XA\ge 0.
\end{equation}
Indeed, by spectral theory, assumption \eqref{2.3a}  yields that $\|A^{*j} X A^j \| \le \| X \|$ for 
$j \ge 0$.  For $f$ in the Wiener algebra $W^+$,  we see that $\sum_{j=0}^\infty |f_j| \| A^{*j} X A^j \|$
converges.
We conclude that the sum defining $f(B_A)(X)$ is absolutely convergent in operator norm if $f \in W^+$ and \eqref{2.3a} 
holds.}
\end{remark}

\begin{proposition}
If $R_\bbeta$ and $R_{\bbeta,k}$ are defined as in \eqref{1.2} and \eqref{2.5}, 
and if 
${\displaystyle\frac{1}{R_\bbeta}}$
belongs to $W^+$, then ${\displaystyle\frac{R_{\bbeta,k}}{R_\bbeta}}$
belongs to $W^+$ for all $k\ge 1$.
\label{P:3.2}
\end{proposition}
\begin{proof} Due to condition \eqref{1.8}, the function 
\begin{equation}
g_k(z)=\sum_{\ell=1}^{k}\frac{1}{\beta_{k-\ell}}\cdot\left(\sum_{j=0}^\infty 
c_{\ell+j}z^j\right)
\label{3.0}
\end{equation}
belongs to $W_+$ and moreover, $\|g_k\|_{W^+}\le 
{\displaystyle\left(\sum_{\ell=1}^k 
\beta_{k-\ell}^{-1}\right)\cdot\left\|\frac{1}{R_\bbeta}\right\|_{W^+}}$. 
Therefore,
the order of summation in \eqref{3.0} can be switched and we get, again making 
use of \eqref{1.8}, 
\begin{align}
g_k(z)&=\sum_{j=0}^\infty\left(\sum_{\ell=1}^{k}\frac{c_{j+\ell}}{\beta_{k-\ell}}\right)z^j
=-\sum_{j=0}^\infty\left(\sum_{\ell=0}^j\frac{c_{j-\ell}}{\beta_{k+\ell}}\right)z^j\notag\\
&=-\left(\sum_{j=0}^\infty
\frac{z^j}{\beta_{k+j}}\right)\cdot\left(\sum_{j=0}^\infty c_jz^j\right)
=-\frac{R_{\bbeta,k}(z)}{R_\bbeta(z)}
\label{3.10}
\end{align}
where the last equality holds due to \eqref{2.5} and \eqref{1.7}. This completes 
the proof, 
since $g_k$ is in $W_+$.
\end{proof}
Making use of the reciprocal power series \eqref{1.7} 
let us  introduce the operator map 
\begin{equation}   \label{3.5}
\Gamma_{\bbeta,A}=\frac{1}{R_{\bbeta}}(B_A) \colon X \mapsto \sum_{j=0}^\infty 
c_j\, A^{*j}XA^j,
\end{equation}
which, according to Remark \ref{R:3.1}, is well defined for any 
operator $X\in\cL(\cX)$ subject to inequalities \eqref{2.3a}. We next use  the 
power  series \eqref{3.10} to define a family of operator maps
\begin{equation}   \label{3.5a}
\Gamma^{(k)}_{\bbeta,A}=\frac{R_{\bbeta,k}}{R_{\bbeta}}(B_A) \colon X \mapsto 
-\sum_{j=0}^\infty\left(\sum_{\ell=1}^{k}\frac{c_{j+\ell}}{\beta_{k-\ell}}\right)
\, A^{*j}XA^j
\end{equation}
for all $k\ge 0$, which are again well defined for any operator $X\in\cL(\cX)$ 
subject to inequalities \eqref{2.3a}, by  Proposition \ref{P:3.2} 
and Remark \ref{R:3.1}. Observe that $\Gamma^{(0)}_{\bbeta,A}$ amounts to the identity 
mapping.

\begin{proposition} 
Let us assume that conditions \eqref{1.6}, \eqref{1.8} are in force and let 
$(C,A)$ be a $\bbeta$-output-stable pair. Then the gramian $\cG_{\bbeta,C,A}$ is subject 
to relations
\begin{equation}   \label{form5}
\cG_{\bbeta,C,A}\ge A^*\cG_{\bbeta,C,A}A\ge 0,\qquad 
\Gamma_{\bbeta,A}[\cG_{\bbeta,C,A}]=C^*C
\end{equation}
and
\begin{equation}   \label{form5a}
\Gamma^{(k)}_{\bbeta,A}[\cG_{\bbeta,C,A}]=\Gr^{(k)}_{\bbeta,C,A}
\ge 0 \quad\mbox{for all}\quad k\ge 1.
\end{equation}
\label{P:2.2}
\end{proposition}
\begin{proof} 
Since $\bbeta$ is non-increasing (by the third condition in \eqref{1.6}), it 
follows from 
the power series representation \eqref{3.1} that
$$
\cG_{\bbeta,C,A}- A^*\cG_{\bbeta,C,A}A=C^*C+\sum_{j=1}^\infty 
\left(\beta_j^{-1}-\beta_{j-1}^{-1}\right) A^{*j}C^*CA^j\ge C^*C\ge 0,
$$
which proves the first statement in \eqref{form5}. Therefore, 
$\Gamma_{\bbeta,A}[\cG_{\bbeta,C,A}]$
and $\Gamma^{(k)}_{\bbeta,A}[\cG_{\bbeta,C,A}]$ are well-defined and the series 
\eqref{3.5}, \eqref{3.5a} (with $X=\cG_{\bbeta,C,A}$) converge absolutely 
against any $x\in\cX$.
Therefore, we can substitute the power series representation \eqref{3.1} for 
$\cG_{\bbeta,C,A}$
in \eqref{3.5} and then change the order of summation. This leads us to the 
second equality in 
\eqref{form5}:
\begin{align*}
\Gamma_{\bbeta,A}[\cG_{\bbeta,C,A}]=&
\sum_{j=0}^\infty c_jA^{*j}\left(\sum_{r=0}^\infty 
\beta_r^{-1}A^{*r}C^*CA^r\right)A^j\\
=&\sum_{\ell=0}^\infty \left(\sum_{j=0}^\ell 
c_{j}\beta_{\ell-j}^{-1}\right)A^{*\ell}C^*CA^\ell=C^*C
\end{align*}
 where the last step follows from the recursion in \eqref{1.8}. 
The same recursion implies
\begin{align*}
\sum_{r=0}^j\frac{1}{\beta_{j-r}}\cdot 
\sum_{\ell=1}^k\frac{c_{\ell+r}}{\beta_{k-\ell}}&=
-\sum_{r=0}^j\frac{1}{\beta_{j-r}}\cdot \sum_{i=0}^r\frac{c_i}{\beta_{k+r-i}}\\
&=-\sum_{\ell=0}^j\frac{1}{\beta_{k+\ell}}\cdot\sum_{r=0}^{j-\ell}\frac{c_r}{\beta_{j-\ell-r}}
=-\frac{1}{\beta_{k+j}}
\end{align*}
which together with \eqref{3.1}, \eqref{3.5a} and \eqref{3.5} leads us to
\begin{align*}
\Gamma^{(k)}_{\bbeta,A}[\cG_{\bbeta,C,A}]&=-
\sum_{j=0}^\infty\left(\sum_{\ell=1}^{k}\frac{c_{j+\ell}}{\beta_{k-\ell}}\right)
\, A^{*j}\left(\sum_{r=0}^\infty \beta_r^{-1}A^{*r}C^*CA^r\right)A^j\\
&=-\sum_{j=0}^\infty\left(\sum_{r=0}^j\sum_{\ell=1}^k\frac{c_{\ell+r}}{\beta_{k-\ell}\beta_{j-r}}
\right)A^{*j}C^*CA^j\\
&=\sum_{j=0}^\infty\beta_{j+k}^{-1}\cdot A^{*j}C^*CA^j=\Gr^{(k)}_{\bbeta,C,A}
\end{align*}
which verifies \eqref{form5a} and thus completes the proof.\end{proof} 

\begin{proposition}
Let us assume that the weight sequence $\bbeta$ meets conditions \eqref{1.6},
\eqref{1.8} and let the operators $A, \, H\in\cL(\cX)$ meet the conditions
\begin{equation}
H\ge A^*HA\ge 0,\qquad \Gamma^{(k)}_{\bbeta,A}[H]\ge 0\quad\mbox{for all}\quad k\ge 1
\label{3.15}
\end{equation}
(the existence of operators $\Gamma^{(k)}_{\bbeta,A}[H]$ is justified in Remark \ref{R:3.1}).
Then
\begin{enumerate}
\item There exists the strong limit
\begin{equation}
\Delta_{A,H}=\lim_{k\to \infty}A^{*k}\Gamma^{(k)}_{\bbeta,A}[H]A^k\ge 0.
\label{ref7}
\end{equation}
\item If in addition, $\Gamma_{\bbeta,A}[H]\ge 0$, then the series below converges
in the strong operator topology and satisfies
\begin{equation}
\sum_{j=0}^\infty \beta_j^{-1}A^{*j}\Gamma_{\bbeta,A}[H]A^j= H-\Delta_{A,H}.
\label{ref5}
\end{equation}
\end{enumerate}
\label{P:3.5}
\end{proposition}
\begin{proof}
From definitions \eqref{1.2} and \eqref{2.6}, it follows that for every $j\ge 0$,
$$
R_{\bbeta,j}(z)=zR_{\bbeta,j+1}(z)+\beta_j^{-1}
$$
Dividing both parts by $R_\bbeta(z)$ and applying Agler hereditary functional calculus to the
resulting identity and to the chosen operator $H$ (this can be done thanks to  Remark
\ref{R:3.1} and  Proposition \ref{P:3.2}) gives the operator equality
\begin{align}
\Gamma^{(j)}_{\bbeta,A}[H]&=B_A\Gamma^{(j+1)}_{\bbeta,A}[H]+\beta_j^{-1} 
\Gamma_{\bbeta,A}[H]
\notag\\
&=A^*\Gamma^{(j+1)}_{\bbeta,A}[H]A+\beta_j^{-1}\Gamma_{\bbeta,A}[H] \notag
\end{align} 
where $\Gamma^{(j)}_{\bbeta,A}[H]$ is simply $H$ for the case $j=0$.
This in turn implies that
\begin{equation}
A^{*j}\Gamma^{(j)}_{\bbeta,A}[H]A^j-A^{*j+1}\Gamma^{(j+1)}_{\bbeta,A}[H]A^{j+1}
=\beta_j^{-1}A^{*j}\Gamma_{\bbeta,A}[H] A^j\ge 0.
\label{ref1}
\end{equation}
Therefore, the sequence of operators $A^{*k}\Gamma^{(k)}_{\bbeta,A}[H]A^k$ is decreasing and
bounded below and therefore has a strong limit \eqref{ref7}. 

\smallskip 

Summing up equalities in \eqref{ref1} for $j=0,\ldots,k$ and taking into account that
$\Gamma^{(0)}_{\bbeta,A}[H]=H$ we get
\begin{equation}
\sum_{j=0}^k\beta_j^{-1}\cdot
A^{*j} \Gamma_{\bbeta,A}[H] A^j=H-(A^*)^{k+1}\Gamma^{(k+1)}_{\bbeta,A}[H]A^{k+1}.
\label{3.18}
\end{equation}
Letting $k\to\infty$ in the latter equality and making use of \eqref{ref7} we 
arrive at \eqref{ref5} thus completing the proof.
\end{proof}
The following result gives connections between $\bbeta$-output stability,
observability gramians and solutions of associated Stein equalities
and inequalities. 
\begin{theorem}
\label{T:1.1}
Let us assume that the weight sequence $\bbeta$ meets conditions \eqref{1.6}, 
\eqref{1.8} and 
let $(C,A)\in\cL(\cX,\cY)\times\cL(\cX)$ be an output pair. Then:
\begin{enumerate}
\item The pair $(C,A)$ is $\bbeta$-output-stable if and only if there is
an $H\in \cL(\cX)$ subject to inequalities \eqref{3.15} and the Stein inequality
\begin{equation}
\Gamma_{\bbeta,A}[H]\ge C^*C.
\label{3.16}
\end{equation}
\item If $(C,A)$ is a $\bbeta$-output-stable pair, then the gramian
$H={\mathcal G}_{\bbeta, C, A}$ satisfies inequalities \eqref{3.15} and the Stein 
equality 
\begin{equation}
\Gamma_{\bbeta,A}[H]= C^*C
\label{3.17}   
\end{equation}
and is the minimal positive semidefinite solution of the system \eqref{3.15}, 
\eqref{3.16}.
           \end{enumerate}
           \end{theorem}

\begin{proof} Suppose first that $(C,A)$ is $\bbeta$-output-stable.
Then the infinite series in \eqref{3.1} converges strongly to the operator
$H={\mathcal G}_{\bbeta,C,A}$. By Proposition \ref{P:2.2}, $H$ satisfies 
relations \eqref{3.15}, \eqref{3.17} and hence, also \eqref{3.16}.

\smallskip

Conversely, suppose that \eqref{3.15}, \eqref{3.16} hold for some $H\in\cL(\cX)$.  
By Proposition \ref{P:3.5}, inequality \eqref{ref1} holds for all $k\ge 1$ 
which being combined with \eqref{3.16} gives
\begin{align}
\sum_{j=0}^k\beta_j^{-1}\cdot
A^{*j}C^*CA^j&\le \sum_{j=0}^k\beta_j^{-1}\cdot
A^{*j}\Gamma_{\bbeta,A}[H]A^j\\
&= H-(A^*)^{k+1}\Gamma^{(k+1)}_{\bbeta,A}[H]A^{k+1}
\label{3.19}
\end{align}
for all $k\ge 1$. By letting $k\to \infty$ in \eqref{3.19} we conclude that the 
left hand side  sum converges to a bounded positive semidefinite operator, which is 
$\cG_{\bbeta,C,A}$ by 
\eqref{3.1}. Thus, passing to the limit in \eqref{3.19} as $k\to \infty$ gives 
$\cG_{\bbeta,C,A}\le H-\Delta_{A,H}$ where $\Delta_{A,H}\ge 0$ is the limit defined in 
\eqref{ref7}. Therefore, the pair $(C,A)$ is $\bbeta$-output-stable 
and  $\cG_{\bbeta,C,A}$ is indeed the minimal positive semidefinite solution to the 
system \eqref{3.15}, \eqref{3.16}.  \end{proof}

\begin{definition}  \label{D:3.5}
Let us say that the operator $A\in\cL(\cX)$ is {\em $\bbeta$-contractive} if $A$ 
is a contraction and
$$
\Gamma_{\bbeta,A}[I_{_\cX}]:=\sum_{j=0}^\infty c_jA^{*j}A^j \ge 0.
$$
Let us say that $A$ is a {\em $\bbeta$-hypercontraction} if in 
addition
$$
\Gamma^{(k)}_{\bbeta,A}[I_{\cX}]:=
-\sum_{j=0}^\infty\left(\sum_{\ell=1}^{k}\frac{c_{j+\ell}}{\beta_{k-\ell}}\right)
\, A^{*j}A^j\ge 0 \quad\mbox{for all}\quad k\ge 1.
$$ 
\end{definition}

\begin{definition} \label{D:2}
A pair $(C,A)\in\cL(\cX,\cY)\times\cL(\cX)$ will be called
$\bbeta$-{\em contractive output pair} if the inequalities \eqref{3.15}, 
\eqref{3.16} hold with 
$H = I_{{\mathcal 
X}}$, that is, $A$ is $\bbeta$-hypercontractive and 
$$
\Gamma_{\bbeta,A}[I_{_\cX}]:=\sum_{j=0}^\infty c_jA^{*j}A^j \ge C^*C.
$$
The pair $(C,A)$ will be called $\bbeta$-{\em isometric} if  $A$ is 
$\bbeta$-hypercontractive and
\begin{equation}
\Gamma_{\bbeta,A}[I_{\cX}]:=\sum_{j=0}^\infty c_jA^{*j}A^j =C^*C.
\label{betaisom}
\end{equation}
\end{definition}

\begin{remark}  \label{R:betan} {\em
    Following the terminology from \cite{olieot, BBberg}, for $n$ a 
    positive integer and $A$ an operator on a Hilbert space $\cX$, we 
    say that {\em $A$ is an $n$-hypercontraction} if 
    $\Gamma_{k,A}[I_{\cX}] \ge 0$ for $k=1$ and $k=n$, where in 
    general, for $k$ an integer between $1$ and $n$ we set
    \begin{equation}   \label{defGammakA}
    \Gamma_{k,A} \colon H \mapsto (I - B_{A})^{k}[H] = \sum_{j=0}^{k} 
    (-1)^{j} \binom{k}{j} A^{*j} H A^{j}.  
    \end{equation}

    If $A$ is an $n$-hypercontraction, it then follows automatically 
    that $\Gamma_{k,A}[I_{\cX}] \ge 0$ for $1 < k < n$ as well (see 
    Lemma 4.3 in \cite{BBberg} or Lemma 1.1 in \cite{olieot}).  It is 
    of interest to compare this notion to that of 
    $\bbeta$-hypercontraction as given in Definition \ref{D:3.5} for the special choice of $\bbeta$ 
    with $\beta_{k} = \frac{k! (n-1)!}{(j+n-1)!}$ (or $\bbeta = 
    \bbeta_{\alpha}$ as in \eqref{1.4} with $\alpha = n$):   $A$ 
    is a $\bbeta_{n}$-hypercontraction if $A$ is a contraction (so 
    $\Gamma_{1,A}[I_{\cX}] \ge 0$) and $\Gamma_{\bbeta_{n},A}[I_{\cX}] 
    = \Gamma_{n,A}[I_{\cX}] \ge 0$ (so in fact $A$ is an 
    $n$-hypercontraction) and also
    \begin{equation}   \label{Gammabeta}
    \Gamma_{\bbeta_{n},A}^{(k)}[I_{\cX}] \ge 0\quad \text{for all}\quad k = 
    1,2, \dots.
    \end{equation}

    We here check that this last condition \eqref{Gammabeta} is 
    automatic for an $n$-hypercontraction and hence the class of 
    $n$-hypercontractions and the class of 
    $\bbeta_{n}$-hypercontractions are identical. To see this, we 
    recall the notation for the resolvent and shifted resolvent:
$$
R_{\bbeta_{n}}(z) =: R_{n}(z) = (1 - z)^{-n},\qquad R_{\bbeta_{n},k}(z)=: R_{n,k}(z)
$$
and the formula for the shifted resolvent in terms of unshifted 
resolvents (see \cite[Section 2]{BBberg}):
$$
R_{n,k}(z) = \sum_{\ell=1}^{n} \binom{\ell + k - 2}{\ell - 1} 
R_{n-\ell +1}(z).
$$
Hence specializing the formula \eqref{3.5a} to the case $\bbeta = 
\bbeta_{n}$ gives us
\begin{align*}
 \Gamma_{\bbeta_{n},A}^{(k)}(B_{A}) &  = \frac{R_{n,k}}{R_{n}}(B_{A})  \\
    & = \sum_{\ell = 1}^{n} \binom{\ell + k - 2}{\ell - 1} 
    \frac{R_{n-\ell +1}}{R_{n}}(B_{A}) \\
    & = \sum_{\ell = 1}^{n} \binom{\ell + k - 2}{\ell - 1} 
    \Gamma_{\ell - 1, A}(B_{A})  \\
    & = \sum_{\ell = 0}^{n-1} \binom{\ell + k - 1}{\ell } 
    \Gamma_{\ell, A}(B_{A}).
\end{align*}
Hence, if $A$ is an $n$-hypercontraction, then $\Gamma_{\ell, 
A}[I_{\cX}] \ge 0$ for $\ell =0, \dots, n-1$ and it follows that
$$
\Gamma_{\bbeta_{n},A}^{(k)}(B_{A})[I_{\cX}] = \sum_{\ell=0}^{n-1} 
\binom{\ell + k -1}{\ell} \Gamma_{\ell, A}(B_{A})[I_{\cX}] \ge 0
$$
for all $k=1,2,\ldots$. We conclude that indeed any $n$-hypercontraction is also a 
$\bbeta_{n}$-hypercontraction.
}\end{remark}

In addition we shall use the following standard terminology from systems theory.
An output pair $(C,A)$ is called {\em observable} if the operator
${\mathcal G}_{{\bf 1}, C, \bA}$ is injective.  A pair $(C,A)$ is called {\em exactly 
observable} if ${\mathcal G}_{{\bf 1}, C, A}$ is bounded and
is strictly positive definite.  We note that the assumptions 
\eqref{1.6} on the sequence $\bbeta$ imply that each term of the 
sequence is positive:  $\beta_{j} > 0$ for $j=0,1,2, \dots$; this 
combined with the formula 
\eqref{3.1} enables one to see that $\cG_{\bbeta,C,A}$ is injective if and only if 
$\cG_{{\mathbf 1}, C,A}$ is injective.  In Proposition 
\ref{P:exactbetaobs} below, we observe that the strict positivity of 
$\cG_{{\mathbf 1}, C, A}$ implies the strict positivity of 
$\cG_{\bbeta, C,A}$.  However the converse implication can fail in 
general (see \cite[Proposition 5.7]{BBberg}.  It therefore makes 
sense to introduce the following notion.

\begin{definition} \label{D:3.11}
    We will say that  the pair $(C,A)$ is {\em exactly $\bbeta$-observable} if the 
$\bbeta$-gramian 
$\cG_{\bbeta,C,A}$ is bounded and is strictly positive definite.

\end{definition}

Let us recall that an operator $A\in\cL(\cX)$ is called {\em strongly stable}
if $A^k$ tends to zero strongly: $\|A^kx\|\to 0$ as $k\to \infty$ for all $x\in\cX$.
We introduce the weighted analog of this notion for $\bbeta$-hypercontraction. 
\begin{definition} \label{D:4}
A $\bbeta$-hypercontractive operator $A\in\cL(\cX)$ will be called $\bbeta$-strongly stable 
if the limit \eqref{ref7} with $H=I_{\cX}$ equals zero, i.e., 
$$
\Delta_{A,I}=\lim_{k\to \infty}A^{*k}\Gamma^{(k)}_{\bbeta,A}[I]A^k=0,
$$
or, equivalently,
\begin{equation}
-\lim_{k\to \infty} 
\sum_{j=0}^\infty\left(\sum_{\ell=1}^{k}\frac{c_{j+\ell}}{\beta_{k-\ell}}\right)
\|A^{j+k}x\|^2=0\quad\mbox{for all}\quad x\in\cX.
\label{betastrong}
\end{equation}
\end{definition}
\begin{remark}
{\rm Observe that ${\bf 1}$-strong stability amounts to the usual strong stability.
Indeed, if $\bbeta={\bf 1}$, then $c_0=1$, $c_1=-1$
and $c_j=0$ for $j\ge 2$ (see \eqref{1.7}). Therefore, the sum   
${\displaystyle
\sum_{\ell=1}^{k}\frac{c_{j+\ell}}{\beta_{k-\ell}}}$ equals to $-1$ if $j=0$ and 
it is equal to zero otherwise. Hence the double sum on the left of \eqref{betastrong}
amounts to the single term $\|A^kx\|^2$.}
\label{R:3.10}
\end{remark}
\begin{lemma} (1) Suppose $(C,A)$ is a $\bbeta$-contractive pair. Then $(C,A)$ is 
$\bbeta$-output stable and $\cG_{\bbeta,C,A}\le I$ so that $\cO_{\bbeta,C,A}: \, \cX\to 
H^2_\bbeta(\cY)$ is a contraction.

\smallskip

(2) Suppose $(C,A)$ is a $\bbeta$-isometric pair and $A$ is $\bbeta$-strongly stable.
Then $\cG_{\bbeta,C,A}=I$, the operator $\cO_{\bbeta,C,A}: \, \cX\to
H^2_\bbeta(\cY)$ is isometric and hence also the pair $(C,A)$ is exactly $\bbeta$-observable.
\label{L:3.12}
\end{lemma}
\begin{proof}
The first statement follows from Theorem \ref{T:1.1}. To prove the second statement, we first observe that 
part (2) in Proposition \ref{P:3.5} applies to $H=I$. Combining \eqref{betaisom} and \eqref{ref5} (with 
$H=I$) now gives
$$
\sum_{j=0}^\infty \beta_j^{-1}A^{*j}C^*CA^j= I-\Delta_{A,I}.
$$
Since $\Delta_{A,I}=0$ and since the series on the left equals $\cG_{\bbeta,C,A}$, we conclude 
$\cG_{\bbeta,C,A}=I$ which completes the proof.\end{proof}

\section{Observability-operator range spaces and reproducing
        kernel Hilbert spaces}\label{S:NC-Obs}

Let $S_\bbeta$ denote the shift operator on $H^2_\bbeta(\cY)$ defined as 
$S_\bbeta: \, f(z)\to zf(z)$. Iterating the formula \eqref{1.5} for its adjoint 
gives 
\begin{equation}
S^{*k}_\bbeta f=\sum_{j=0}^\infty\frac{\beta_{j+k}}{\beta_j}\cdot f_{j+k}z^j.
\label{3.1a}
\end{equation}
\begin{lemma}
Let us assume that the weight sequence $\bbeta$ meets conditions \eqref{1.6},
\eqref{1.8}. Let $E: \, H^2_\bbeta(\cY)\to \cY$ be defined by $Ef=f(0)$. Then
\begin{enumerate}
\item $S_\bbeta^*$ is strongly stable, i.e.,
${\displaystyle\lim_{k\to \infty}\|S^{*k}_{\bbeta}f\|_{H^2_\bbeta(\cY)}=0\quad\mbox{for 
all}\quad f\in  H^2_\bbeta(\cY)}$.
\item The pair $(E, S_\bbeta^*)$ is $\bbeta$-output stable and the 
observability 
operator
$\cO_{\bbeta,E, S_\bbeta^*}$ equals the identity operator on $H^2_\bbeta(\cY)$. 
\item $S^*_\bbeta$ is $\bbeta$-hypercontractive and $\bbeta$-strongly stable. Moreover,
\begin{equation}
\Gamma_{\bbeta,S_\bbeta^*}[I_{H^2_\bbeta(\cY)}]=E^*E\quad\mbox{and}\quad\
\Gamma^{(k)}_{\bbeta,S_\bbeta^*}[I_{H^2_\bbeta(\cY)}]=\Gr^{(k)}_{\bbeta,E,S_\bbeta^*}
\label{3.1e}  
\end{equation}
for all $k\ge 1$, or equivalently,
\begin{align}
\sum_{j=1}^\infty c_j\|S_\bbeta^{*j}f\|_{H^2_\bbeta(\cY)}^2&=\|f(0)\|^2_{\cY},
\label{3.1g}\\
-\sum_{j=0}^\infty\left(\sum_{\ell=1}^{k}\frac{c_{j+\ell}}{\beta_{k-\ell}}\right)
\, \|S_\bbeta^{*j}f\|_{H^2_\bbeta(\cY)}^2&=\sum_{j=0}^\infty 
\frac{\beta_{j}^{2}}{\beta_{k+j}}\|f_j\|^2_{\cY}
\qquad (k\ge 1)\label{3.1h}
\end{align}
for all $f\in H^2_\bbeta(\cY)$, where $c_j$'s are given in \eqref{1.8}.
\end{enumerate}
\label{L:3.1}
\end{lemma}

\begin{proof}
If $f(z)={\displaystyle \sum_{k=0}^\infty f_kz^k}$ belongs to $H^2_\bbeta(\cY)$, 
then  by \eqref{1.1},
\begin{equation}
\sum_{j=k}^\infty \beta_j\|f_j\|^2_{\cY}\to 0\quad\mbox{as}\quad k\to\infty,
\label{ref8}
\end{equation}
and then it follows from \eqref{3.1a} and  \eqref{1.1} that 
\begin{align*}
\|S^{*k}_{\bbeta}f\|^2=\sum_{j=0}^\infty 
\frac{\beta_{j+k}^2}{\beta_{j}}\|f_{j+k}\|^2
&\le \sup_{j\ge 0}\frac{\beta_{j+k}}{\beta_j}\cdot \sum_{j=0}^\infty 
\beta_{j+k}\|f_{j+k}\|^2\notag\\
&\le \sum_{j=k}^\infty \beta_{j}\|f_{j}\|^2\to 0\quad\mbox{as}\quad k\to\infty
\end{align*}
which proves the first statement.  It follows from \eqref{3.1a} that 
$ES^{*j}_\bbeta f= \beta_j f_{j}$ and therefore,
$$
{\mathcal O}_{\bbeta,E,S_\bbeta^*}f=
ER_\bbeta(zS_\bbeta^*)f=\sum_{j=0}^\infty \beta_j^{-1} \left(E 
S_{\bbeta}^{*j}f\right) z^{j}= \sum_{j=0}^{\infty} f_{j} z^{j} = f
$$
which proves the second statement. Therefore $\cG_{\bbeta,E,S_\bbeta^*}=I$ and 
hence, equalities \eqref{3.1e} follow by Proposition \ref{P:2.2}.
The useful more general identity
\begin{equation}   \label{Obk-model}
    \Ob^{(k)}_{\bbeta,E,S_{\bbeta}^{*}} f = \sum_{j=0}^{\infty} 
    \frac{\beta_{j}}{\beta_{k+j}} f_{j}z^j
\end{equation}
follows from the computation
$$
\Ob^{(k)}_{\bbeta, E, S_{\bbeta}^{*}} f 
= \sum_{j=0}^{\infty} \frac{1}{\beta_{k+j}} \left( E 
S_{\bbeta}^{*j}f \right) z^{j} = \sum_{j=0}^{\infty} 
\frac{\beta_{j}}{\beta_{k+j}} f_{j} z^{j}.
$$

Next we note that the first operator equality in \eqref{3.1e} is 
equivalent to the quadratic-form identity
$$
\langle \Gamma_{\bbeta, S_{\bbeta}^{*}}[I_{H^{2}_{\bbeta}(\cY)}]f, f 
\rangle = \langle E^{*} E f, f \rangle_{H^{2}_{\bbeta}(\cY)}
$$
holding for all $f \in H^{2}_{\bbeta}(\cY)$ which in turn is 
equivalent to \eqref{3.1g}.    To verify the equivalence of the 
second operator identity in \eqref{3.1e} with the quadratic-form 
identity \eqref{3.1h}, note first that
the equality
$$ \langle\Gamma^{(k)}_{\bbeta,S_\bbeta^*}[I_{H^2_\bbeta(\cY)}]f,\, f 
\rangle_{H^{2}_{\bbeta}(\cY)} = 
-\sum_{j=0}^\infty\left(\sum_{\ell=1}^{k}\frac{c_{j+\ell}}{\beta_{k-\ell}}\right)
\, \|S_\bbeta^{*j}f\|_{H^2_\bbeta(\cY)}^2
$$
is immediate from the definitions, while the identity
\begin{equation}  \label{Grk-quadform}
\langle \Gr^{(k)}_{\bbeta, E, S_{\bbeta}^{*}} f, \, f 
\rangle_{H^{2}_{\bbeta}(\cY)} = 
\sum_{j=0}^{\infty}\frac{\beta_{j}^{2}}{\beta_{k+j}} 
\|f_{j}\|^{2}_{\cY}
\end{equation}
follows from the computation
\begin{align*}
    \langle \Gr^{(k)}_{\bbeta,E, S_{\bbeta}^{*}} f, \, f 
    \rangle_{H^{2}_{\bbeta}(\cY)} &
    = \left\| S_{\bbeta}^{k} \Ob^{(k)}_{\bbeta, E, S_{\bbeta}^{*}} f 
    \right\|^{2}_{H^{2}_{\bbeta}(\cY)} \text{ (using \eqref{defRa})} \\
    & = \left\| \sum_{j=0}^{\infty} \frac{\beta_{j}}{\beta_{k+j}} 
    f_{j} z^{j+k} \right\|^{2} \text{ (using \eqref{Obk-model})} \\
    & = \sum_{j=0}^{\infty} \beta_{j+k} \cdot \left( \frac{ 
    \beta_{j}}{\beta_{k+j}} \right)^{2} \|f_{j}\|^{2}_{\cY} \\
    & = \sum_{j=0}^{\infty} \frac{ \beta_{j}^{2}}{\beta_{j+k}} \| 
    f_{j} \|^{2}_{\cY}. 
\end{align*}
Combining \eqref{Grk-quadform} and \eqref{3.1a} gives
\begin{align*}
\langle \Gr^{(k)}_{\bbeta, E, S_{\bbeta}^{*}} S_{\bbeta}^{*k}f, \, S_{\bbeta}^{*k}f
\rangle_{H^{2}_{\bbeta}(\cY)}
&=\sum_{j=0}^{\infty}\frac{\beta_{j}^{2}}{\beta_{k+j}}
\|(S_{\bbeta}^{*k}f)_{j}\|^{2}_{\cY}\\
&=\sum_{j=0}^{\infty}\frac{\beta_{j}^{2}}{\beta_{k+j}}
\left\|\frac{\beta_{j+k}}{\beta_j}f_{j+k}\right\|^{2}_{\cY}\\
&=\sum_{j=0}^{\infty}\beta_{j+k}\cdot\|f_{j+k}\|^2_{\cY}=
\sum_{j=k}^{\infty}\beta_{j}\cdot\|f_{j}\|^2_{\cY}.
\end{align*}
This together with \eqref{ref8} and the second equality in 
\eqref{3.1e} implies that
\begin{align*}
\langle S_{\bbeta}^{k}\Gamma^{(k)}_{\bbeta,S_\bbeta^*}[I]S_{\bbeta}^{*k}f, \, 
f\rangle_{H^{2}_{\bbeta}(\cY)}&=\langle \Gr^{(k)}_{\bbeta, E, S_{\bbeta}^{*}} S_{\bbeta}^{*k}f, \, 
S_{\bbeta}^{*k}f\rangle_{H^{2}_{\bbeta}(\cY)}\\
&=\sum_{j=k}^{\infty}\beta_{j}\cdot\|f_{j}\|^2_{\cY}\to 
0\quad\mbox{as}\quad k\to\infty.
\end{align*}
This finally verifies $\bbeta$-strong stability of $ S_{\bbeta}$ and completes the proof.
\end{proof}

\smallskip

Associated with a $\bbeta$-output-stable pair $(C,A)$ is the range of the
observability operator
$$
\operatorname{Ran}{\mathcal O}_{\bbeta,C,A} =
\{ CR_\bbeta(zA)x \colon \; x \in {\mathcal X}\}.
$$

\begin{theorem} \label{T:1.2}
Suppose that $(C,A)$ is a $\bbeta$-output-stable pair. Then 
\begin{enumerate}
\item The intertwining relation
\begin{equation}
S_\bbeta^*{\mathcal O}_{\bbeta,C, A}={\mathcal O}_{\bbeta,C, A}A
\label{3.3}   
\end{equation}
holds and hence the linear manifold ${\mathcal M}=\operatorname{Ran}
{\mathcal O}_{\bbeta,C,A}$  is $S_\bbeta^*$-invariant.
\item If in addition, $(C,A)$ is exactly $\bbeta$-observable, then the linear 
manifold ${\mathcal 
M}:=\operatorname{Ran}{\mathcal O}_{\bbeta,C,A}$ given the lifted norm 
$$
\|{\mathcal O}_{\bbeta,C,A} x \|_{{\mathcal M}}^{2} =
\langle \cG_{\bbeta,C,A} x, x \rangle_{{\mathcal X}},
$$
is isometrically included in $H^2_\bbeta(\cY)$ and is isometrically equal to the 
reproducing  kernel Hilbert space with reproducing kernel 
$$
K_{\bbeta,C,A}(z,\zeta) = 
CR_\bbeta(zA)\cG_{\bbeta,C,A}^{-1}R_\bbeta(\zeta A)^* C^{*}.
$$
\item Conversely, if ${\mathcal M}$ is a Hilbert space isometrically
included in $H^2_\bbeta(\cY)$ which is invariant under $S_\bbeta^{*}$, 
then there is a $\bbeta$-output stable exactly $\bbeta$-observable pair $(C,A)$
such that ${\mathcal M} = {\mathcal H}(K_{\bbeta,C,A}) = 
\operatorname{Ran}{\mathcal
O}_{\bbeta,C,A}$ isometrically.
\end{enumerate}
\end{theorem}

\begin{proof}  
Making use of power series expansion \eqref{0.6} and of
\eqref{1.5} we get \eqref{3.3}:
$$
S_\bbeta^*{\mathcal O}_{\bbeta,C, A}x=S_\bbeta^*\sum_{k=0}^\infty \beta_k^{-1} 
(CA^{k}x)z^k
=\sum_{k=0}^\infty \beta_{k}^{-1}(CA^{k+1}x)z^k={\mathcal O}_{\bbeta,C,A}Ax.
$$
The second statement follows from the definition \eqref{3.1} and general 
principles 
laid out in \cite{NFRKHS} (see also \cite{BC, oljfa, BBberg} for 
applications very close to the context here).  To prove the last statement, 
observe that 
for ${\mathcal M}$ a Hilbert space isometrically included in $H^2_\bbeta(\cY)$, 
we may 
let $A=S_\bbeta^*\vert_{\cM}$ and let 
$C$ be defined by $Cf=f(0)$ for all $f\in\cM$. In other words $C=E\vert_{\cM}$. 
Then 
the pair $(C,A)$ is $\bbeta$-output stable and exactly $\bbeta$-observable. It 
follows from part (2) 
in Lemma \ref{L:3.1} that $\cM=\operatorname{Ran}{\mathcal
O}_{\bbeta,C,A}$. \end{proof}

Theorem \ref{T:1.2} leads to the following operator-model theorem.

\begin{theorem} \label{T:beta-stablemodel}  Suppose that the Hilbert 
    space operator $A \in \cL(\cX)$ is a $\bbeta$-strongly stable 
    $\bbeta$-hypercontraction.  Let $\cY$ be a coefficient Hilbert 
    space with $\operatorname{dim} \cY = \operatorname{rank} 
    \Gamma_{\bbeta,A}[I_{\cX}]$.  Then there is a subspace $\cN 
    \subset H^{2}_{\bbeta}(\cY)$ invariant under $S_{\bbeta}^{*}$ 
    so that $T$ is unitarily equivalent to $S_{\bbeta}^{*}|_{\cN}$.
\end{theorem}

\begin{proof} We choose $C \in \cL(\cX, \cY)$ so that $C^{*}C = 
    \Gamma_{\bbeta, A}[I_{\cX}]$.  Then $(C,Aâ$ is a 
    $\bbeta$-isometric pair.  By Lemma \ref{L:3.12}, the assumption 
    that $A$ is $\bbeta$-strongly stable guarantees us that 
    $\cO_{\bbeta,C,A} \colon \cX \to H^{2}_{\bbeta}(\cY)$ is 
    isometric.  We set $\cN = \operatorname{Ran} \cO_{\bbeta, C,A} 
    \subset H^{2}_{\bbeta}(\cY)$.  Then the intertwining property 
    \eqref{3.3} in Theorem \ref{T:1.2} leads to the conclusion that 
    $A$ is unitarily equivalent to $S_{\bbeta}^{*}|_{\cN}$ via the 
    unitary similarity transformation  $\cO_{\bbeta,C,A} \colon \cX 
    \to \cN$.
\end{proof}

\begin{remark}   \label{R:stable}  {\em 
    As a consequence of Theorem \ref{T:beta-stablemodel} combined 
    with part (1) of Lemma \ref{L:3.1}, we see that $\bbeta$-strong 
    stability of a $\bbeta$-hypercontraction $A$ implies its strong 
    stability in the usual sense ($\| A^{n} x \| \to 0$ as $n \to 
    \infty$ for each $x \in \cX$).  For the special case where 
    $\bbeta = \bbeta_{n}$ for a positive integer $n$ as in Remark 
    \ref{R:betan},  it is known that strong stability of the 
    $\bbeta_{n}$-hypercontraction $A$ implies that $A$ is unitarily 
    equivalent to the restriction of $S_{\bbeta_{n}}^{*}$ to an 
    invariant subspace $\cN \subset H^{2}_{\bbeta_{n}}(\cY) = 
    \cA_{n}(\cY)$ for a suitable coefficient Hilbert space $\cY$ as 
    in Theorem \ref{T:beta-stablemodel}, by results from \cite{BBberg}
    (see Theorem 5.3 part (2) there).  When this is combined with 
    part (3) of Lemma \ref{L:3.1}, we see that strong stability 
    implies $\bbeta_{n}$-strong stability for a 
    $\bbeta_{n}$-hypercontraction, and hence strong stability and 
    $\bbeta_{n}$-strong stability are equivalent for 
    $\bbeta_{n}$-hypercontractions. We have not been able to 
    determine at this time whether this equivalence holds for a 
    general weight satisfying the standing hypothesis \eqref{1.6}.
    } \end{remark}

We record here the following fact which will be useful in the sequel.  

\begin{proposition}  \label{P:exactbetaobs}  
    Suppose that the pair $(C,A)$ is $\bbeta$-output-stable and  exactly 
    $\bbeta$-observable (so $\Gr^{(0)}_{\bbeta,C,A} = \cG_{\bbeta,C,A}$ 
    is strictly positive definite).  Then it follows that 
    $\Gr^{(k)}_{\bbeta,C,A}$ is strictly positive definite for $k=1,2,3, 
    \dots$.
    \end{proposition}
    
    \begin{proof} 
	The standing assumption \eqref{1.6} gives us that 
	$\beta_{k+1} \le \beta_{k}$, i.e., $\frac{1}{\beta_{k}} \le 
	\frac{1}{\beta_{k+1}}$, for all $k=0,1,2,\dots$.  We then 
	read off from the formula \eqref{defR} that 
	$\Gr^{(k)}_{\bbeta,C,A} \le \Gr^{(k+1)}_{\bbeta,C,A}$, and in 
	particular $\cG_{\bbeta,C,A} = \Gr^{(0)}_{\bbeta,C,A} \le 
	\Gr^{(k)}_{\bbeta, C,A}$ for all $k=1,2,3,\dots$.  It now 
	follows that strict positivity of $\cG_{\bbeta,C,A}$ implies 
	strict positivity of $\Gr^{(k)}_{\bbeta,C,A}$ as asserted.
	\end{proof}

\section{Functions $\Theta_{k}$ and metric constraints }
\label{S:metric}

In this section we will study the transfer functions $\Theta_{k}$ introduced by 
the
realization formula \eqref{2.10} when the system operators $A$, $B_k$, $C$, $D_k$ 
satisfy certain 
additional metric constraints. By Proposition \ref{P:3.0}, for a $\bbeta$-output 
stable pair $(C,A)$,
the associated backward-shifted observability operators $\Ob^{(k)}_{\bbeta,C,A}$ 
are bounded for all $k\ge 
0$ as operators from $\cX$ into $H^2_\bbeta(\cY)$. In this case, the 
multiplication operator 
$M_{\Theta_{k}}$ given (according to \eqref{2.10}) by
$$
M_{\Theta_{k}}=\beta_k^{-1}\cdot D_k+ S_\bbeta \Ob^{(k+1)}_{\bbeta,C,A}B_k: \;
\cU_k\to H^2_\bbeta(\cY)
$$
is also bounded. Therefore, the output function $\widehat{y}$ in \eqref{2.8k},
\begin{equation}
\widehat y(z)=\cO_{\bbeta, C,A}x+\sum_{k=0}^{N}z^k \Theta_{k}(z)u_k
\label{jul12}
\end{equation}
belongs to $H^2_\bbeta(\cY)$ for every choice of $x\in\cX$ and
$u_k\in\cU_k$ for each $N\ge 1$. We next impose some
additional metric relations on $\left[ \begin{smallmatrix} A & B_{k}
\\ C & D_{k} \end{smallmatrix} \right]$, specifically one or more of
the relations
\begin{align}
A^*\Gr^{(k+1)}_{\bbeta,C,A}B_k+\beta_k^{-1}\cdot C^*D_k&=0,\label{jul13}\\
B_k^*\Gr^{(k+1)}_{\bbeta,C,A}
B_k+\beta_k^{-1}\cdot D_k^*D_k&\le I_{\cU_k},\label{jul14}\\
B_k^*\Gr^{(k+1)}_{\bbeta,C,A}B_k+\beta_k^{-1}\cdot D_k^*D_k&= 
I_{\cU_k},\label{jul14a}
\end{align}  
and show how these lead to boundedness and orthogonality properties
for the associated multiplication operator $M_{\Theta_{k}}$.
Due to equality \eqref{7.3}, it turns out that relations \eqref{jul13} and 
\eqref{jul14} are equivalent
to the matrix inequality
\begin{equation}   \label{contr}
  \begin{bmatrix} A^{*} & C^{*} \\ B_{k}^{*} & D_{k}^{*}
  \end{bmatrix} \begin{bmatrix} \Gr^{(k+1)}_{\bbeta,C,A} & 0 \\ 0 & 
\beta_k^{-1}\cdot  I_{\cY}
\end{bmatrix}   \begin{bmatrix} A & B_{k} \\ C & D_{k} \end{bmatrix}
 \le \begin{bmatrix} \Gr^{(k)}_{\bbeta,C,A} & 0 \\ 0 & I_{\cU_{k}} \end{bmatrix},
    \end{equation}
while the equalities \eqref{jul13} and \eqref{jul14a} are equivalent to the 
matrix equality
\begin{equation}   \label{isom}
  \begin{bmatrix} A^{*} & C^{*} \\ B_{k}^{*} & D_{k}^{*}
  \end{bmatrix} \begin{bmatrix} \Gr^{(k+1)}_{\bbeta,C,A} & 0 \\ 0 & 
\beta_k^{-1}\cdot  I_{\cY}
\end{bmatrix}   \begin{bmatrix} A & B_{k} \\ C & D_{k} \end{bmatrix}
 = \begin{bmatrix} \Gr^{(k)}_{\bbeta,C,A} & 0 \\ 0 & I_{\cU_{k}} \end{bmatrix}.
    \end{equation}
The two latter conditions are of metric nature; they express the
contractivity or isometric property of the  colligation operator $U_k =
\left[ \begin{smallmatrix} A & B_k \\ C & D_k \end{smallmatrix} \right]$
with respect to certain weights. 

\begin{lemma}   \label{L:5.1}
Let $(C,A)$ be a $\bbeta$-output stable pair and let $\Theta_{k}$ be defined
as in \eqref{2.10}  for some integer $k\ge 0$ and operators $B_k\in \cL(\cU_k, 
\cX)$ and $D_k\in
\cL(\cU_k, \cY)$.

\smallskip
\noindent
$(1)$  If equality \eqref{jul13} holds, then
\begin{itemize}
\item[(a)] $\cO_{\bbeta,C,A}x$ is orthogonal to $S_\bbeta^k\Theta_{k}u$ for all 
$x\in\cX$ and
$u\in\cU_k$.
\item[(b)] $S_\bbeta^k\Theta_{k}u$ is orthogonal to  
$S_\bbeta^m\Theta_{k}u^\prime$ for all $m>k$ and
$u,u^\prime\in\cU_k$.
\end{itemize}

\noindent
$(2)$ If \eqref{jul14} holds, then 
$S_{\bbeta}^{k} M_{\Theta_{k}}$ is a contraction from
    $\cU_{k}$ into $H^2_\bbeta(\cY)$.

\smallskip
\noindent
$(3)$ If \eqref{jul14a} holds, then    
$S_{\bbeta}^{k} M_{\Theta_{k}}$ is an isometry from $\cU_{k}$ into 
$H^2_\bbeta(\cY)$.

\smallskip
\noindent
$(4)$ If both \eqref{jul13} and \eqref{jul14} hold, i.e., if
    \eqref{contr} holds, then $S_{\bbeta}^{k}M_{\Theta_{k}}$ is a contraction 
from $H^{2}(\cU_{k})$ into  $H^2_\bbeta(\cY)$.

\smallskip
\noindent
$(5)$ If \eqref{jul13} and \eqref{jul14a} hold, i.e., if \eqref{isom} holds, then
    \begin{equation}   \label{SThetank-isom}
        \| S_{\bbeta}^{k} \Theta_{k}f \|^{2}_{H^2_\bbeta(\cY)} =
    \| f \|^{2}_{H^{2}(\cU_{k})} - \sum_{j=1}^{\infty } \| (I -
    S_{\bbeta}^{*}S_{\bbeta})^{1/2} S_{\bbeta}^{k} \Theta_{k} S_{{\bf 1}}^{*j} f 
\|^{2}
    \end{equation}
for every $f\in H^2(\cU_k)$, and 
\begin{align}
&  \beta_k^{-1}I_{\cU_{k}} - \Theta_{k}(z)^*
   \Theta_{k}(\zeta) =   \beta_{k}B_k^*R_{\bbeta,k}(z 
A)^*\Gr^{(k+1)}_{\bbeta,C,A}R_{\bbeta,k}(\zeta A)B_k
\notag \\
&  \quad  - \overline{z}\zeta \beta_{k}  B_k^*R_{\bbeta,k+1}(z 
A)^*\Gr^{(k)}_{\bbeta,C,A}R_{\bbeta,k+1}(\zeta
A)B_k \label{jul18}
\end{align}
for all $z,\zeta\in\D$.
\end{lemma}
\noindent
{\em Proof of (1):} We first observe the power series expansion
\begin{equation}
\Theta_{k}(z)=\beta_{k}^{-1}D_k+\sum_{j=0}^\infty\beta_{j+k+1}^{-1} CA^jB_kz^{j+1}
\label{jul15}
\end{equation}
which is an immediate consequence of formulas \eqref{2.10} and \eqref{2.6}. We 
then make use of
expansions \eqref{0.6}, \eqref{jul15} and the definition of the inner product in 
$H^2_\bbeta(\cY)$
to prove part (a):
\begin{align*} 
&\left\langle S_\bbeta^{k}\Theta_{k} u, \, 
\cO_{\bbeta,C,A}x\right\rangle_{H^2_\bbeta(\cY)}=
\beta_{k}\cdot\left\langle \beta_{k}^{-1}D_ku, \, 
\beta_{k}^{-1}CA^{k}x\right\rangle_{\cY}  \\
&\qquad+\sum_{j=0}^\infty\beta_{j+k+1}\cdot \left\langle
\beta_{j+k+1}^{-1}C A^{j} B_k u, \, \beta_{j+k+1}^{-1}CA^{j+k+1} 
x\right\rangle_{\cY}  \\
&=\left\langle \left(\beta_k^{-1} C^*D_k+ A^*\left(
\sum_{j=0}^\infty \beta_{j+k+1}^{-1} A^{*j}C^*CA^{j}\right)B_k\right)u,\,
A^kx\right\rangle_{\cX}  \\
&=\left\langle \left(\beta_k^{-1}C^*D_k+
A^*\Gr^{(k+1)}_{\bbeta,C,A}B_k\right)u,\, A^kx\right\rangle_{\cX}=0  
\end{align*}
where the two last equalities are justified by \eqref{defR} and \eqref{jul13}, 
respectively.
Verification of part $(b)$ is quite similar: for $m>k$ we have
\begin{align*}  
&\left\langle S_\bbeta^{m}\Theta_{k} u^\prime,\, 
S_\bbeta^{k}\Theta_{k}u\right\rangle_{H^2_\bbeta(\cY)}=
\beta_{m}\cdot\left\langle  \beta_{k}^{-1}D_ku^\prime, \;
\beta_{m}^{-1}CA^{m-k-1}B_k u\right\rangle_{\cY}  \\
&\qquad+\sum_{j=0}^\infty\beta_{j+m+1}\cdot \left\langle
\beta_{j+k+1}^{-1}C A^{j} B_k u^\prime, \,
\beta_{j+m+1}^{-1}CA^{j+m-k}B_k u\right\rangle_{\cY}  \\
&=\left\langle \left(\beta_k^{-1} C^*D_k+A^*\left(
\sum_{j=0}^\infty \beta_{j+k+1}^{-1}A^{*j}C^*CA^{j}\right)B_k
\right)u^\prime,\, A^{m-k-1}B_ku\right\rangle_{\cX}  \\
&=\left\langle \left(\beta_k^{-1} C^*D_k+A^*\Gr^{(k+1)}_{\bbeta,C,A}B_k\right)
u^\prime, \,  A^{m-k-1}B_ku\right\rangle_{\cX} =0.
\end{align*}

\smallskip
\noindent
{\em Proof of (2) and (3):} According to \eqref{jul14},
\begin{equation}
\beta_k^{-1}\cdot \| D_ku\|^2_{\cY}+\left\langle\Gr^{(k+1)}_{\bbeta,C,A}B_ku, \; 
B_ku\right\rangle_{\cX}\le
\|u\|^2_{\cU_k}
\label{jul20}
\end{equation}
for all $u\in\cU_k$.
We now may make use of \eqref{defRa} and \eqref{jul20} to get
\begin{align}
\| S_\bbeta^{k}\Theta_{k} u\|^2_{H^2_\bbeta(\cY)}&=
\|(\beta_k^{-1}S_\bbeta^k D_k+ S^{k+1}_\bbeta 
\Ob^{(k+1)}_{\bbeta,C,A}B_k)u\|^2_{H^2_\bbeta(\cY)}\notag\\
&=\|\beta_k^{-1}S_\bbeta^k D_ku\|^2_{H^2_\bbeta(\cY)}+\|S^{k+1}_\bbeta 
\Ob^{(k+1)}_{\bbeta,C,A}B_ku\|^2_{H^2_\bbeta(\cY)}\notag\\
&=\beta_k^{-1}\| D_ku\|^2_{\cY}+\left\langle\Gr^{(k+1)}_{\bbeta,C,A}B_ku, \; 
B_ku\right\rangle_{\cX}
\le \|u\|^2_{\cU_k}.\notag
\end{align}
Thus, $S_\bbeta^k M_{\Theta_{k}}$  is a contraction from $\cU_k$ to
$H^2_\bbeta(\cY)$. In case \eqref{jul14a} holds,  then
\eqref{jul20} holds with equality and part (3) follows.

\smallskip
\noindent
{\em Proof of (4):} Under the
assumption that both \eqref{jul13} and \eqref{jul14} hold, we shall show that for 
any
$\cU_k$-valued polynomial  $f(z)={\displaystyle \sum_{j=0}^m f_jz^j}$,
\begin{equation}
\|S_\bbeta^{k}\Theta_{k} f\|^2_{H^2_\bbeta(\cY)}\le
\|f\|_{H^2(\cU_k)}^2=\sum_{j=0}^m\|f_j\|_{\cU_k}^2.
\label{8.3}
\end{equation}
Let $S_{\bf 1}^*$ be the operator of backward shift on $H^2(\cU_k)$ so that for 
the
polynomial $f$ as above, $(S_1^*f)(z)={\displaystyle \sum_{j=0}^{m-1} 
f_{j+1}z^j}$. By   
statements (1b) and (2) of the lemma, we have
\begin{align}
\left\|S_\bbeta^{k}\Theta_{k} f 
\right\|^2=&\left\|\sum_{j=0}^mS_\bbeta^{k+j}\Theta_{k}
f_j\right\|^2\notag\\
=&\left\|S_\bbeta^{k}\Theta_{k}f_0 \right\|^2+
\left\|\sum_{j=1}^mS_\bbeta^{k+j}\Theta_{k}f_j \right\|^2   \text{ (by (1b))} 
\notag\\
\le &\left\|f_0 \right\|^2+
\left\|S_\bbeta^{k+1}\sum_{j=0}^{m-1}S_\bbeta^{j}\Theta_{k}f_{j+1}\right\|^2\notag\\
=&\|f_0\|^2+\|S_\bbeta^{k+1}\Theta_{k} S_{\bf 1}^*f\|^2\notag\\
=&\left\|f_0 \right\|^2+
\left\|S_\bbeta^{k}\Theta_{k} S_{\bf 1}^*f \right\|^2- 
\left\|(I-S_\bbeta^*S_\bbeta)^{\frac{1}{2}}S_\bbeta^{k}
\Theta_{k}S_1^*f\right\|^2.\label{8.4}
\end{align}
Replacing $f$ by $S_{\bf 1}^{*j}f$ in \eqref{8.4} gives
\begin{align}
\|S_\bbeta^{k}\Theta_{k} S_{\bf 1}^{*j}f\|^2\le 
\|f_j\|^2&+\|S_{\bbeta}^k\Theta_{k}
(S_{\bf 1}^{*})^{j+1}f\|^2\notag\\
&-\|(I-S_\bbeta^*S_\bbeta)^{\frac{1}{2}}S_\bbeta^{k}\Theta_{k} (S_1^*)^{j+1}f\|^2
 \label{8.3'}
\end{align}
for $j=1,\ldots,m$.  Iteration of the inequality \eqref{8.4} using
\eqref{8.3'} then gives
\begin{align}
\|S_\bbeta^{k}\Theta_{k} f\|^2_{H^2_\bbeta(\cY)}
&\le 
\sum_{j=0}^m\|f_j\|_{\cU_k}^2-\sum_{j=1}^m\|(I-S_\bbeta^*S_\bbeta)^{\frac{1}{2}}S_\bbeta^{k}
\Theta_{k}S_{\bf 1}^{*j}f\|^2 \label{8.3iterate}\\
&\le \sum_{j=0}^{m} \| f_{j}\|^{2}_{\cU_{k}}.\notag
\end{align}
Letting $m \to \infty$ in \eqref{8.3iterate} now implies the validity
of \eqref{8.3} for every $f\in H^2(\cU_k)$ and the proof of (4) is now complete.

\smallskip
\noindent
{\em Proof of (5):}  In case \eqref{jul13} and \eqref{jul14a} hold, 
    then \eqref{8.4} holds with equality as well as in \eqref{8.3}, \eqref{8.3'}, 
and \eqref{8.3iterate}. Equality \eqref{SThetank-isom} now follows by letting $m
    \to \infty$ in \eqref{8.3iterate}.	

\smallskip

    It remains to verify the formula \eqref{jul18} under assumption
    \eqref{isom}.  The identity \eqref{isom} is equivalent to 
    \eqref{7.3}, \eqref{jul13} and \eqref{jul14a}.
We use these relations to compute
\begin{align*}
   & \beta_k^{-1}I_{\cU_{k}} - \Theta_{k}(z)^{*}\Theta_{k}(\zeta) \\
   &=  \beta_k^{-1}I_{\cU_{k}}-
    \left[ \beta_k^{-1}D_{k}^{*} + \overline{z} B_{k}^{*}
    R_{\bbeta,k+1}(zA)^{*} C^{*} \right]
    \left[  \beta_k^{-1}D_{k} + \zeta C R_{\bbeta,k+1}(\zeta A) B_{k}
    \right]  \\
    &   = \beta_k^{-1}I_{\cU_{k}} -    \beta_k^{-2}D_{k}^{*}D_{k}- 
\overline{z}\beta_k^{-1} 
B_{k}^{*} R_{\bbeta,k+1}(zA)^{*}C^{*}D_{k} \\
&\qquad -\zeta \beta_k^{-1}D_{k}^{*}C R_{\bbeta,k+1}(\zeta A) B_{k} 
- \overline{z} \zeta B_{k}^{*} R_{\bbeta,k+1}(zA)^{*}
    C^{*}C R_{\bbeta,k+1}(\zeta A) B_{k}  \\
    &  = \beta_k^{-1} B_{k}^{*} \Gr^{(k+1)}_{\bbeta,C,A} B_{k} +
    \overline{z} B_{k}^{*} R_{\bbeta,k+1}(zA)^{*}A^* \Gr^{(k+1)}_{\bbeta,C,A} 
B_{k} \\
    & \qquad 
     + \zeta B_{k}^{*} \Gr^{(k+1)}_{\bbeta,C,A} A R_{\bbeta,k+1}(\zeta A) B_{k}  
\\
     & \qquad 
    - \overline{z} \zeta \beta_k B_{k}^{*} R_{\bbeta,k+1}(zA)^{*}
    \left(\Gr^{(k)}_{\bbeta,C,A} - A^{*} \Gr^{(k+1)}_{\bbeta,C,A}A \right) 
R_{\bbeta,k+1}(\zeta 
A) 
B_{k}
    \end{align*}
    where we made use of  \eqref{7.3}, \eqref{jul13} and \eqref{jul14a} in the 
last step.
  By making use next of the relation 
$$
 R_{\bbeta,k}(zA)=\beta_k^{-1}I+zA R_{\bbeta,k+1}(zA)
$$
we can continue the
  computation as
  \begin{align*}
   & \beta_k^{-1}I_{\cU_{k}} - \Theta_{k}(z)^{*} \Theta_{k}(\zeta)  \\
   &= \beta_k^{-1} B_{k}^{*} \Gr^{(k+1)}_{\bbeta,C,A} B_{k} + B_{k}^{*}
   \left( R_{\bbeta,k}(zA)^{*} - \beta_k^{-1}I \right)
   \Gr^{(k+1)}_{\bbeta,C,A} B_{k}  \\
   & \quad + B_{k}^{*} \Gr^{(k+1)}_{\bbeta,C,A} \left(
   R_{\bbeta,k}(\zeta A) - \beta_k^{-1}I \right) B_{k} \\
   & \quad + \beta_{k} B_{k}^{*} \left( R_{\bbeta,k}(zA)^{*} -
   \beta_k^{-1}I \right) \Gr^{(k+1)}_{\bbeta,C,A} \left( R_{\bbeta,k}(\zeta
   A) - \beta_k^{-1}I \right) B_{k}  \\
   & \quad - \beta_{k} \overline{z} \zeta B_{k}^{*} R_{\bbeta,k+1}(zA)^{*}
   \Gr^{(k)}_{\bbeta,C,A} R_{\bbeta,k+1}(\zeta A) B_{k} \\
   & = \beta_{k} B_{k}^{*} R_{\bbeta,k}(zA)^{*} \Gr^{(k+1)}_{\bbeta,C,A} 
R_{\bbeta,k}(\zeta
   A) B_{k} \\
   &\quad- \overline{z} \zeta\beta_k B_{k}^{*}
   R_{\bbeta,k+1}(zA)^{*} \Gr^{(k)}_{\bbeta,C,A} R_{\bbeta,k+1}(\zeta A) B_{k}
   \end{align*}
   verifying formula \eqref{jul18}.\qed
     
\smallskip        
        
The following result is an immediate consequence of Lemma \ref{L:5.1}.

\begin{corollary}   \label{C:5.3}
Let us assume that the pair $(C,A)$ is $\bbeta$-output stable and that relations 
\eqref{jul13}, 
\eqref{jul14}  hold for all $k\ge 0$. Then the representation \eqref{jul12} of 
the
function $\widehat{y}$ is orthogonal in the metric of $H^2_\bbeta(\cY)$ and
\begin{equation}
\|\widehat y\|^2_{H^2_\bbeta(\cY)}=\|\cO_{\bbeta,C,A}x\|^2+\sum_{k=0}^\infty
\|\Theta_{k}u_k\|^2\le
\|\cG_{\bbeta,C,A}^{\frac{1}{2}}x\|^2_{\cX}+\sum_{k=0}^\infty \|u_k\|^2_{\cU_k}.
\label{jul17}
\end{equation}
If relations \eqref{jul14} hold with equalities for all $k\ge 0$, then equality 
holds in \eqref{jul17}.
\end{corollary}

Observe that in case the pair $(C,A)$ is exactly $\bbeta$-observable 
(so $\Gr^{(k)}_{\bbeta,C,A}$ is strictly positive definite for all $k \ge 
0$ by Proposition \ref{P:exactbetaobs}),
the inequality \eqref{contr} can equivalently be expressed as 
$\| \Xi \| \le 1$ where $\Xi: \; \sbm{\cX \\ \cU_{k}}\to \sbm{\cX \\ \cY}$ 
is the operator given by
\begin{equation}   \label{pr7}
 \Xi: = \begin{bmatrix} \left(\Gr^{(k+1)}_{\bbeta,C,A}\right)^{1/2} & 0 \\ 0
 & \beta_k^{-\frac{1}{2}}I_{\cY} \end{bmatrix}
 \begin{bmatrix} A & B_{k} \\ C & D_{k} \end{bmatrix}
     \begin{bmatrix} \left(\Gr^{(k)}_{\bbeta,C,A}\right)^{-1/2} & 0 \\ 0 &
         I_{\cU_{k}} \end{bmatrix}.
 \end{equation}
An equivalent condition is that $\| \Xi^{*} \| \le 1$ which in
turn can be expressed as
$$
    \begin{bmatrix} A & B_{k} \\ C & D_{k} \end{bmatrix}
        \begin{bmatrix} \left(\Gr^{(k)}_{\bbeta,C,A}\right)^{-1} & 0 \\ 0 & 
I_{\cU_k} 
\end{bmatrix}
        \begin{bmatrix} A^{*} & C^{*} \\ B_{k}^{*} & D_{k}^{*}
        \end{bmatrix}\le  \begin{bmatrix} 
\left(\Gr^{(k+1)}_{\bbeta,C,A}\right)^{-1} & 0 \\ 0 &
        \beta_{k} I_{\cY} \end{bmatrix}.
$$
 Note that equality \eqref{isom} means that the operator $\Xi$ is isometric.
Of  particular interest  is the case where $\Xi$ is coisometric, i.e., where
the colligation operator $U_k= \left[ \begin{smallmatrix} A & B_k \\ C & 
D_k \end{smallmatrix} 
\right]$ is coisometric with respect to the weights indicated below:
\begin{equation}   \label{wghtcoisom}
    \begin{bmatrix} A & B_{k} \\ C & D_{k} \end{bmatrix}
        \begin{bmatrix} \left(\Gr^{(k)}_{\bbeta,C,A}\right)^{-1} & 0 \\ 0 & 
I_{\cU_k} 
\end{bmatrix}
        \begin{bmatrix} A^{*} & C^{*} \\ B_{k}^{*} & D_{k}^{*}
        \end{bmatrix} = \begin{bmatrix} 
\left(\Gr^{(k+1)}_{\bbeta,C,A}\right)^{-1} & 0 \\ 0 &
        \beta_{k} I_{\cY} \end{bmatrix}.
\end{equation}

\begin{lemma}   \label{L:frakcoisom}
Let $(C,A)$ be an exactly $\bbeta$-observable $\bbeta$-output stable pair and let 
$\Theta_{k}$ 
be defined as in \eqref{2.10}  for some operators $B_k\in \cL(\cU_k, \cX)$ and 
$D_k\in 
\cL(\cU_k)$ subject to equality \eqref{wghtcoisom}. Then
\begin{align}
&  \beta_k^{-1}I_{\cY}- \Theta_{k}(z)\Theta_{k}(\zeta)^{*} = 
CR_{\bbeta,k}(zA)\left(\Gr^{(k)}_{\bbeta,C,A}\right)^{-1}R_{\bbeta,k}(\zeta 
A)^*C^*  \notag  \\
& \quad  - z \overline{\zeta} \cdot 
CR_{\bbeta,k+1}(zA)\left(\Gr^{(k+1)}_{\bbeta,C,A}\right)^{-1}R_{\bbeta,k+1}(\zeta 
A)^*C^*.
\label{frakkerid}
\end{align}
\end{lemma}
 
\begin{proof}The proof parallels the verification of the identity
    \eqref{jul18} done above.  The weighted-coisometry condition 
\eqref{wghtcoisom}
    gives us the set of equations
 \begin{align}
A \left(\Gr^{(k)}_{\bbeta,C,A}\right)^{-1} A^{*} + B_{k} B_{k}^{*} & = 
\left(\Gr^{(k+1)}_{\bbeta,C,A}\right)^{-1}, \notag\\
C\left(\Gr^{(k)}_{\bbeta,C,A}\right)^{-1} A^{*} + D_{k} B_{k}^{*} & = 0,  
\label{relations1} \\
C \left(\Gr^{(k)}_{\bbeta,C,A}\right)^{-1} C^{*} + D_{k} D_{k}^{*} &= \beta_{k} 
I_{\cY}.\notag
 \end{align}
We then  compute:
 \begin{align*}
    & \beta_k^{-1}I_{\cY} - \Theta_{k}(z)\Theta_{k}(\zeta)^{*} \\
     & =   \beta_k^{-1}I_{\cY} - \left[ \beta_k^{-1}D_{k} + z
     CR_{\bbeta,k+1}(zA) B_{k} \right] \left[ \beta_k^{-1}D_{k}^{*} +
     \bar{\zeta} B_{k}^{*} R_{\bbeta,k+1}(\zeta A)^*C^*\right] \\
     & = \beta_k^{-1}I_{\cY} - \beta_k^{-2}D_{k}D_{k}^{*} - z CR_{\bbeta,k+1}(zA)
     \beta_k^{-1}B_{k}D_{k}^{*} - \bar{\zeta}\beta_k^{-1}
     D_{k}B_{k}^{*} R_{\bbeta,k+1}(\zeta A)^*C^*  \\
     & \quad -  z \bar{\zeta} \cdot CR_{\bbeta,k+1}(zA)B_{k}B_{k}^{*}
     R_{\bbeta,k+1}(\zeta A)^*C^* \text{ (by \eqref{relations1})}      \\
     & =  \beta_k^{-2}C\left(\Gr^{(k)}_{\bbeta,C,A}\right)^{-1}C^{*} + 
z\beta_k^{-1} 
CR_{\bbeta,k+1}(zA) A \left(\Gr^{(k)}_{\bbeta,C,A}\right)^{-1}C^{*} \\
 & \quad + \bar{\zeta}\beta_k^{-1}C
\left(\Gr^{(k)}_{\bbeta,C,A}\right)^{-1}A^{*} R_{\bbeta,k+1}(\zeta A)^*C^* \\
     & \quad        - z \bar{\zeta} \cdot CR_{\bbeta,k+1}(zA)\left[ 
\left(\Gr^{(k+1)}_{\bbeta,C,A}\right)^{-1}
- A\left(\Gr^{(k)}_{\bbeta,C,A}\right)^{-1}A^{*}\right] R_{\bbeta,k+1}(\zeta 
A)^*C^*\\
&=C\left(\beta_k^{-1}I_{\cX}+zR_{\bbeta,k+1}(zA)A\right)\left(\Gr^{(k)}_{\bbeta,C,A}\right)^{-1}
\left(\beta_k^{-1}I_{\cX}+\overline{\zeta}A^*R_{\bbeta,k+1}(\zeta A)^*\right)C^* 
\\
&\quad- z \bar{\zeta} \cdot 
CR_{\bbeta,k+1}(zA)\left(\Gr^{(k+1)}_{\bbeta,C,A}\right)^{-1}
R_{\bbeta,k+1}(\zeta A)^*C^*\\
&=CR_{\bbeta,k}(zA)\left(\Gr^{(k)}_{\bbeta,C,A}\right)^{-1}
R_{\bbeta,k}(\zeta A)^*C^*\notag \\
&\quad - z \overline{\zeta} \cdot 
CR_{\bbeta,k+1}(zA)\left(\Gr^{(k+1)}_{\bbeta,C,A}\right)^{-1}
R_{\bbeta,k+1}(\zeta A)^*C^*.
     \end{align*}
\end{proof}

\begin{remark}   \label{R:defect}   
    {\em More generally, if $\Theta_{k}(z)$ is given by \eqref{2.10}
    and if we do not assume the weighted coisometry condition
    \eqref{wghtcoisom}, then the decomposition \eqref{frakkerid}
    holds in the more general form
    \begin{align*}
     &  \beta_k^{-1}I_{\cY} - \Theta_{k}(z)\Theta_{k}(\zeta)^{*} =
CR_{\bbeta,k}(zA)\left(\Gr^{(k)}_{\bbeta,C,A}\right)^{-1}R_{\bbeta,k}(\zeta 
A)^*C^* \\
     & \quad - z  \overline{\zeta} \cdot
CR_{\bbeta,k+1}(zA)\left(\Gr^{(k+1)}_{\bbeta,C,A}\right)^{-1}
R_{\bbeta,k+1}(\zeta A)^*C^*+ \Xi_{k}(z, \zeta)
    \end{align*}
    where the defect kernel $\Xi_{k}(z, \zeta)$ is given by
  \begin{align*}
       \Xi_{k}(z, \zeta) = &\begin{bmatrix} z CR_{n,k}(zA) &
 \beta_k^{-1}I_{\cY} \end{bmatrix}
\left( \begin{bmatrix} \left(\Gr^{(k+1)}_{\bbeta,C,A}\right)^{-1} & 0 \\ 0 &
      \beta_{k} I_{\cY} \end{bmatrix}\right. \\
      & \left.
    - \begin{bmatrix} A & B_{k} \\ C & D_{k} \end{bmatrix}
  \begin{bmatrix} \left(\Gr^{(k)}_{\bbeta,C,A}\right)^{-1}
& 0 \\ 0 &  I_{\cY} \end{bmatrix}
      \begin{bmatrix} A^{*} & C^{*} \\ B_{k}^{*} & D_{k}^{*}
      \end{bmatrix} \right) \cdot  \begin{bmatrix} \overline{\zeta}
   R_{\bbeta,k}(\zeta A)^*C^* \\ \beta_k^{-1}I_{\cY}
  \end{bmatrix}.
  \end{align*}
     }\end{remark}
     
Since equality \eqref{wghtcoisom} implies inequality \eqref{contr}, it follows 
that under 
assumption of Lemma \ref{L:frakcoisom}, all the conclusions of parts (1), (2) and 
(4) in Lemma 
\ref{L:5.1} are true. To have
all conclusions true, we need the operator \eqref{pr7} to be unitary.
\begin{lemma}   \label{L:5.6}
    Suppose that we are given an integer $k \ge 0$ and an exactly 
$\bbeta$-observable 
$\bbeta$-output-stable pair
    $(C,A)\in\cL(\cX,\cY)\times\cL(\cX)$. Then there exist operators  $B_k\in 
\cL(\cU_k, 
\cX)$
and $D_k\in \cL(\cU_k, \cY)$ such that equalities \eqref{wghtcoisom} and 
\eqref{isom} hold.
 Explicitly, such  $B_{k}$ and $C_{k}$ are essentially uniquely
 determined by solving the Cholesky factorization problem:
  \begin{equation}  \label{pr6}
\begin{bmatrix}B_k \\ D_k\end{bmatrix}\begin{bmatrix}B_k^* & D_k^*\end{bmatrix}=
\begin{bmatrix}\left(\Gr^{(k+1)}_{\bbeta,C,A}\right)^{-1}
 & 0 \\ 0 & \beta_{k} I_{\cY}\end{bmatrix}-
\begin{bmatrix}A \\ C\end{bmatrix}\left(\Gr^{(k)}_{\bbeta,C,A}\right)^{-1}
\begin{bmatrix}A^* & C^*\end{bmatrix}
\end{equation}
subject to the additional constraint that the coefficient space
$\cU_{k}$ be chosen so that $\left[ \begin{smallmatrix} B_{k} \\
D_{k} \end{smallmatrix} \right] \colon \cU_{k} \to \cX \oplus \cY$ is
injective.
\end{lemma}
\begin{proof}
By Proposition \ref{P:3.0}, the weighted Stein
identity \eqref{7.3} holds for each $k \ge 1$.
Since $(C,A)$ is exactly observable, Proposition \ref{P:exactbetaobs} 
assures us that the gramian $\Gr^{(k)}_{\bbeta,C,A}$ is strictly
positive definite. It then follows from \eqref{7.3} that the operator
$$
\begin{bmatrix}\left(\Gr^{(k+1)}_{\bbeta,C,A}\right)^{\frac{1}{2}}A
\left(\Gr^{(k)}_{\bbeta,C,A}\right)^{-\frac{1}{2}}\\
\beta_{k}^{-\frac{1}{2}}C\left(\Gr^{(k)}_{\bbeta,C,A}\right)^{-\frac{1}{2}}
\end{bmatrix} \colon  \cX\to \cX\oplus \cY
$$
is an isometry.  By extending this operator to a coisometric  operator 
\eqref{pr7}
we arrive at $B_k$ and $D_k$ solving \eqref{pr6}.  Further, extension
of this operator to a unitary amounts to the additional restriction
that $\left[ \begin{smallmatrix} B_{k} \\ D_{k} \end{smallmatrix}
\right]$ be injective.
\end{proof}

\begin{remark}
    {\em For the classical Hardy-space setting the general 
    principle behind Lemma \ref{L:5.6} is as follows:  {\em given a 
    kernel on ${\mathbb D}$ with realization of the form $K(z, \zeta) 
    = C(I - zA)^{-1} (I - \overline{\zeta} A^{*})^{-1} C^{*}$ where 
    the pair $(C,A)$ is isometric in the sense that $A^{*}A + C^{*} C 
    = I$, one can produce a function $\Theta(z)$ with associated de 
    Branges-Rovnyak kernel $K_{\Theta}(z, \zeta)$ equal to $K$:
    $$
    \frac{ I - \Theta(z) \Theta(\zeta)^{*}}{1 -z \overline{\zeta}} = 
    C (I - zA)^{-1} (I - \overline{\zeta} A^{*})^{-1} C^{*}.
    $$
    Moreover, one can take $\Theta(z)$ to have the form $\Theta(z) = 
    D + z C (I - zA)^{-1} B$ where $\left[ \begin{smallmatrix} B \\ C 
    \end{smallmatrix} \right]$ is constructed as an injective  
    solution of the Cholesky factorization problem}
    $$
    \begin{bmatrix} B \\ D \end{bmatrix} \begin{bmatrix} B^{*} & 
	D^{*} \end{bmatrix} = \begin{bmatrix} I & 0 \\ 0 & I 
    \end{bmatrix} - \begin{bmatrix} A \\ C \end{bmatrix}
    \begin{bmatrix} A^{*} & C^{*} \end{bmatrix}.
    $$
    This principle appears explicitly in \cite[Section 3.1]{BF-Helton} 
    for the indefinite metric setting and in \cite[Theorem 
    1.3]{BB-IEOT2008} for the Drury-Arveson-space multivariable 
    setting.
    }\end{remark}

The results of this section suggest that the following definition 
will be useful.

\begin{definition}  \label{D:conserv-col}  Suppose that
    $$
    U_{k} = \begin{bmatrix} A & B_{k} \\ C & D_{k} \end{bmatrix} 
    \colon \begin{bmatrix} \cX \\ \cU_{k} \end{bmatrix} \to 
    \begin{bmatrix} \cX \\ \cY \end{bmatrix}
   $$
 is a colligation family with $A$ $\bbeta$-hypercontractive and 
 $(C,A)$ exactly $\bbeta$-observable.  We  then say that 
 $\{U_{k}\}_{k \ge 0}$ is
 \begin{enumerate}
     \item a {\em  $\bbeta$-isometric colligation family} if $U_{k}$ 
     satisfies \eqref{isom} for each $k$,
     \item a {\em $\bbeta$-coisometric colligation family} if 
     $U_{k}$ satisfies  \eqref{wghtcoisom} for each $k$, and
     \item a {\em $\bbeta$-unitary colligation family} if $U_{k}$ 
     satisfies both \eqref{isom} and \eqref{wghtcoisom} for each $k$.
  \end{enumerate}
  \end{definition}
  
 The following corollary is an immediate consequence of Lemma \ref{L:5.1} 
 parts (1) and (3) together with  formula \eqref{ZtransIO}. In its formulation 
we use the notation $H^{2}(\{\cU_{k}\}_{k \ge 0})$ for the ``time-varying Hardy
         space'' $\bigoplus_{k=0}^{\infty} z^{k} \cU_{k}$ and we let 
$\ell^{2}_{\bbeta}(\cY)$ to denote
     the space of $\cY$-valued sequences $\{ y(k) \}_{k \ge 0}$ with
     norm given by $\|\{ y(k) \}_{k \ge 0}\|^{2} =
     \sum_{k=0}^{\infty} \bbeta_{k} \| y\|^{2}_{\cY}$.
 
 \begin{corollary} \label{C:bbetaisom}
     Suppose that  $\left\{ U_{k} = \left[\begin{smallmatrix} A & B_{k} \\ C & 
     D_{k} \end{smallmatrix} \right] \right\}$ is a 
     $\bbeta$-isometric family and let $\Sigma_{\bbeta}$ be the 
     associated time-varying linear system as in \eqref{2.1}, and let 
     $\{\Theta_{k}\}_{k \ge 0}$ be the associated transfer-function 
     family as in \eqref{2.10}.  Then:
     \begin{enumerate}
	 \item The operator
	 $$
	 M_{\Theta} = \begin{bmatrix} M_{\Theta_{0}} & M_{\Theta_{1}} 
	 & M_{\Theta_{2}} & \cdot \end{bmatrix} \colon 
	 \bigoplus_{k=0}^{\infty} z^{k}u(k) \mapsto 
	 \sum_{k=0}^{\infty} \Theta_{k}(z) z^{k} u(k)
	 $$
	 is an isometry from $H^{2}(\{ \cU_{k}\}_{k \ge 0})$ into 
	 $H^{2}_{\bbeta}(\cY)$. 
	 \item The input-output map $T_{\Sigma_{\bbeta}}$ \eqref{IOmap} acts as 
     an isometry from $\bigoplus_{k=0}^{\infty} \cU_{k}$ into 
     $\ell^{2}_{\bbeta}(\cY)$.  
     \end{enumerate}
  \end{corollary}

\begin{remark}  \label{R:unitarycol}
   {\em  To handle the case where $(C,A)$ is not necessarily exactly 
    $\bbeta$-observable, one can proceed as follows.  We let 
    $\cX_{k}$ be the space $\cX$ but with a new inner product
    $$
    \langle x, y \rangle_{\cX_{k}} = \langle \Gr^{(k)}_{\bbeta, C,A} 
    x, y \rangle.
    $$
    If $\Gr^{(k)}_{\bbeta, C,A}$ is not injective, we identify elements of 
    self inner-product equal to $0$ with the zero element of the space.
   We then complete $\cX_{k}$ if necessary to arrive at a Hilbert 
   space, still denoted as $\cX_{k}$.  Similarly, we let $\cY_{k}$ be 
   the space $\cY$ but with new inner product
   $$
   \langle y, y' \rangle_{\cY_{k}} = \beta_{k}^{-1} \cdot \langle y, y' 
   \rangle_{\cY}.
   $$
   We let $A_{k}$ denote the operator $A$, but viewed as an operator 
   from $\cX_{k}$ to $\cX_{k+1}$.  Similarly we let $C_{k}$ denote 
   the operator $C$ but viewed as an operator from $\cX_{k}$ into 
   $\cY_{k}$.  Then the identity \eqref{7.3} tells us that the 
   operator
   $$ \begin{bmatrix} A_{k} \\ C_{k} \end{bmatrix} \colon  \cX_{k} \to 
   \begin{bmatrix} \cX_{k+1} \\ \cY_{k} \end{bmatrix},
   $$
   defined initially only on the image of $\cX \oplus \cY$ in $\cX_{k} 
   \oplus \cY_{k}$, extends uniquely to a well-defined isometry.  
   Thus $\cN_{k}: = \operatorname{Ran} \left[ \begin{smallmatrix} A_{k} \\ C_{k} 
   \end{smallmatrix} \right]$ is a closed subspace of $\cX_{k} \oplus 
   \cY_{k}$.  We choose as coefficient space $\cU_{k}$ a copy of the 
   orthogonal complement 
   $$
     \cU_{k} = \begin{bmatrix} \cX_{k} \\ \cY_{k} \end{bmatrix} 
     \ominus \cN_{k}
   $$
   and let $\left[ \begin{smallmatrix} B_{k} \\ D_{k} 
\end{smallmatrix} \right] \colon \cU_{k} \to \left[ \begin{smallmatrix} 
\cX_{k} \\ \cY_{k} \end{smallmatrix}  \right]\ominus \cN_{k}
$ be any convenient unitary identification map.  In this way we 
arrive at a unitary colligation matrix
$$
  \bU_{k} = \begin{bmatrix} A_{k} & B_{k} \\ C_{k} & D_{k} 
\end{bmatrix} \colon \begin{bmatrix} \cX_{k} \\ \cU_{k} \end{bmatrix} 
\to \begin{bmatrix} \cX_{k+1} \\ \cY_{k} \end{bmatrix}.
$$
In case $B_{k}$ has range inside the image of $\cX$ in $\cX_{k}$, one 
can interpret the unitary property of $\bU_{k}$ in terms of the 
original $\cX$-inner product to 
arrive back at the relations \eqref{isom}.  Even in the general 
case, the relations \eqref{isom} still hold with proper 
interpretation.  While this procedure is more general than that taken 
in Lemma \ref{L:5.6}, the construction via this procedure is less 
explicit.
}\end{remark}

\section{Beurling-Lax theorem for $H^2_\bbeta(\cY)$}

The classical Beurling-Lax theorem states that every shift-invariant closed 
subspace $\cM$ of $H^2(\cY)$ can be represented in the form $\cM=\Theta\cdot 
H^2(\cU)$ for an auxiliary coefficient Hilbert space $\cU$ and an 
inner $\cL(\cU,\cY)$-valued 
function $\Theta$. In this section we present 
three analogues of the Beurling-Lax representation theorem for the 
weighted Hardy-space setting.

\subsection{Shift-invariant subspaces
contractively included in $H^2_\bbeta(\cY)$}  \label{S:c-i-sub}

Let us say that the Hilbert space  $\cM$ is contractively included in
the Hilbert space $\cH$ if $\cM \subset \cH$ as sets and moreover
$ \|m\|_{\cM} \ge \| m\|_{\cH}$ for all $m \in \cM$.
We also say that an $\cL(\cU,\cY)$-valued function $\Theta$ is a {\em contractive 
multiplier}
if the operator $M_{\Theta}: \, f(z) \mapsto \Theta(z) \cdot f(z)$ of 
multiplication
by $\Theta$ defines a contractive operator from $H^2_\bbeta(\cU)$ to 
$H^2_\bbeta(\cY)$.

\begin{theorem} \label{T:4.1} 
A Hilbert space $\cM$ is such that
\begin{enumerate}
\item $\cM$ is contractively included in $H^2_\bbeta(\cY)$,
\item $\cM$ is $S_\bbeta$-invariant: $S_\bbeta\cM\subset\cM$,
\item the operator $A=(S_\bbeta\vert_{\cM})^*$ is a $\bbeta$-strongly stable 
$\bbeta$-hypercontraction,
\end{enumerate}
if and only if there is a coefficient Hilbert space $\cU$ and a
contractive multiplier $\Theta$ so that
\begin{equation}  \label{4.0} 
\cM = \Theta\cdot H^2_\bbeta({\cU})
\end{equation}
with lifted norm
\begin{equation}  \label{4.1}
\| \Theta\cdot f \|_{\cM} = \| Q f \|_{H^2_\bbeta({\cU})}
\end{equation}
where $Q$ is the orthogonal projection onto $(\operatorname{Ker} \, 
M_{\Theta})^\perp$.  In this case $\cM$ is itself a reproducing 
kernel Hilbert space with reproducing kernel given by
\begin{equation}   \label{kerM}
    K_{\cM}(z, \zeta) = \Theta(z) (K_{\bbeta}(z, \zeta) I_{\cU}) 
    \Theta(\zeta)^{*}, \quad (z, \zeta) \in {\mathbb D}^{2}.
\end{equation}
        \end{theorem}
  
\begin{proof}  We first verify sufficiency. Suppose that $\cM$ has the form 
\eqref{4.0} for a
contractive multiplier $\Theta$ with $\cM$-norm given by \eqref{4.1}.  Since
$\|M_{\Theta}\| \le 1$, it follows that
$$
\|\Theta f\|_{H^2_\bbeta(\cY)} =  \|M_\Theta Qf\|_{H^2_\bbeta(\cY)}\le 
\|Qf\|_{H^2_\bbeta(\cU)}=\|\Theta f\|_{\cM}
$$
i.e., (1) holds. Property (2) follows from the intertwining equality $S_\bbeta 
M_{\Theta} =
M_{\Theta}S_\bbeta$. The latter  intertwining equality also implies $M_\Theta
S_\bbeta\vert_{\operatorname{Ker}M_\Theta}=0$ which can be written equivalently 
in terms of 
the orthogonal projection $Q$ onto $(\operatorname{Ker}
M_{\Theta})^{\perp}\subset H^2_\bbeta({\cU})$ as $QS_\bbeta(I-Q)=0$. Thus, we 
have
\begin{equation}  \label{4.2}
Q S_\bbeta = Q S_\bbeta Q\quad \text{and}\quad S_\bbeta^{*} Q= QS_\bbeta^{*}Q.
\end{equation}
Furthermore, for every $f, \, g\in H^2_\bbeta(\cU)$, we have
\begin{align*}
\langle \Theta g, \, A\Theta f\rangle_{\cM}=&\langle S_\bbeta\Theta g, \, \Theta 
f\rangle_{\cM}=\langle \Theta S_\bbeta g, \, \Theta f\rangle_{\cM}=\langle Q 
S_\bbeta g, \, 
f\rangle_{H^2_\bbeta(\cU)}\\
=&\langle Q S_\bbeta Qg, \, f\rangle_{H^2_\bbeta(\cU)}=\langle Q g, \, 
S_\bbeta^*Qf\rangle_{H^2_\bbeta(\cU)}=
\langle \Theta g, \, \Theta S_\bbeta^* Qf\rangle_{\cM},
\end{align*}
which implies that $A: \, \Theta f\to \Theta S_\bbeta^* Qf$. Iterating the latter 
formula 
gives
\begin{equation}
A^j: \, \Theta f\to \Theta S_\bbeta^{*j} Qf\quad\mbox{for}\quad j\ge 0.
\label{4.4}
\end{equation}

We have  from \eqref{4.2}, \eqref{4.4} and \eqref{4.1},
\begin{align*}
\left\langle (\Gamma^{(k)}_{\bbeta,A}[I_{\cM}])\Theta f, \, \Theta
f\right\rangle_{\cM}
& =  \sum_{j=0}^\infty \left( - \sum_{\ell=1}^{k}
\frac{c_{j+\ell}}{\beta_{k-\ell}} \right)  \|A^j\Theta f\|^2_{\cM}\\
&= \sum_{j=0}^\infty\left( - \sum_{\ell=1}^{k} \frac{ c_{j+\ell}}{
\beta_{k-\ell}} \right) \|S_\bbeta^{*j}Qf\|^2_{H^2_\bbeta(\cU)}\\
&=\left\langle \left(\Gamma^{(k)}_{\bbeta,S_\bbeta^*}[I_{H^2_\bbeta(\cU)}]\right)
Q f, \, Q f\right\rangle_{H^2_\bbeta(\cU)}
\end{align*}
for $k=0,1,2,\dots$, and also 
\begin{align*}
\langle A_j^{*k}\Gamma^{(k)}_{\bbeta,A}[I_{\cM}]A^k\Theta f, \, \Theta f\rangle_{\cM}
&=\langle\Gamma^{(k)}_{\bbeta,A}[I_{\cM}]A^k\Theta f, \, A^k\Theta f\rangle_{\cM}\\
&=\langle\Gamma^{(k)}_{\bbeta,A}[I_{\cM}]\Theta S_\bbeta^{*k}Qf, \, \Theta S_\bbeta^{*k}
Qf\rangle_{\cM}\\
&=\left\langle
\left(\Gamma^{(k)}_{\bbeta,S_\bbeta^*}[I_{H^2_\bbeta(\cU)}]\right)QS_\bbeta^{*k}Qf, \,
QS_\bbeta^{*k}Qf\right\rangle_{H^2_\bbeta(\cY)}\\
&=\left\langle
\left(\Gamma^{(k)}_{\bbeta,S_\bbeta^*}[I_{H^2_\bbeta(\cU)}]\right)S_\bbeta^{*k}Qf, \,
S_\bbeta^{*k}Qf\right\rangle_{H^2_\bbeta(\cY)}.
\end{align*}
Since $S_\bbeta^*$ is a $\bbeta$-strongly stable $\bbeta$-hypercontraction on $H^2_\bbeta(\cY)$,
we conclude from the latter computations that 
$$
\left\langle (\Gamma^{(k)}_{\bbeta,A}[I_{\cM}])\Theta f, \, \Theta
f\right\rangle_{\cM}=\left\langle \left(\Gamma^{(k)}_{\bbeta,S_\bbeta^*}[I_{H^2_\bbeta(\cU)}]\right)
Q f, \, Q f\right\rangle_{H^2_\bbeta(\cU)}\ge 0 
$$
for $k\ge 0$, and that
$$
\lim_{k\to \infty}\langle A_j^{*k}\Gamma^{(k)}_{\bbeta,A}[I_{\cM}]A^k\Theta f, \, \Theta f\rangle_{\cM}
=0.
$$
We conclude that $A$ is a $\bbeta$-strongly stable $\bbeta$-hypercontraction 
on $\cM$ and thereby complete the proof of sufficiency. 

\smallskip

Suppose now that the Hilbert space $\cM$ satisfies conditions (1), (2), (3) in 
the
statement of the theorem. Using hypothesis (2) we can define the operator 
$A:=(S_\bbeta\vert_{\cM})^*$
on $\cM$ and since it is $\bbeta$-hypercontractive by hypothesis (3), the 
operator 
$\Gamma_{\bbeta,A}[I_{\cM}]$ is positive semidefinite. Choose the coefficient 
Hilbert space 
$\cU$ so that
$$
\operatorname{dim} \cU = \operatorname{rank} \Gamma_{\bbeta,A}[I_{\cM}]
$$
and then choose the operator $C \colon \cM \to \cU$ so that
$C^{*}C =\Gamma_{\bbeta,A}[I_{\cM}]$. Then $(C, A)$ is a $\bbeta$-isometric pair 
and, since $A$ 
is $\bbeta$-strongly stable by hypothesis (3),
it follows that the observability operator
$\cO_{\bbeta,C, A} \colon f \mapsto CR_{\bbeta}(zA) f$
is an isometry from $\cM$ into $H^2_\bbeta(\cY)$. By part (1) of Theorem 
\ref{T:1.2}, we have
the intertwining equality \eqref{3.3}. Taking adjoints in \eqref{3.3} then gives
$$
\cO_{\bbeta, C,A}^*S_\bbeta=A^*\cO_{\bbeta, C,A}^*.
$$
The inclusion map $\iota \colon \cM \to H^2_\bbeta(\cY)$ is a contraction by 
hypothesis (1).
Moreover, $\iota \circ A^*=S_\bbeta \circ\iota: \, \cM\to  H^2_\bbeta(\cY)$. 
Therefore the 
operator
$$
R = \iota \circ \cO_{\bbeta,C, A}^{*} \colon  H^2_\bbeta(\cU)\to  H^2_\bbeta(\cY)
$$
is a contraction and
$$
RS_\bbeta=\iota \circ \cO_{\bbeta,C, A}^{*}S_\bbeta=\iota \circ A^*\cO_{\bbeta,C, A}^{*}=
S_\bbeta \circ\iota \circ\cO_{\bbeta,C, A}^{*}=S_\bbeta R.
$$
Therefore (see \cite{olieot}) $R$ is a multiplication operator, i.e., there is a 
contractive
multiplier $\Theta$ so that $R = M_{\Theta}$. Since
$\cO_{\bbeta,C, A} \colon \cM \to H^2_\bbeta(\cY)$ is an isometry, it 
follows that
$\operatorname{Ran}\cO_{\bbeta,C,A}^{*} = \cM$ and also that $\cM = \Theta \cdot
H^2_\bbeta(\cU)$ with $\cM$-norm given by \eqref{4.1}.  

\smallskip

Finally, if $\cM$ is given by \eqref{4.0} and \eqref{4.1} and if $f = 
\Theta \cdot g$ (with $g$ assumed to be in 
$(\operatorname{Ker} M_{\Theta})^{\perp}$) is a generic element of 
$\cM$, then we see from the lifted-norm property transferred to inner 
products that
\begin{align*}
    \langle f, \Theta k_{\bbeta}(\cdot, \zeta)  \Theta(\zeta)^{*} y 
    \langle_{\cM} & = \langle \Theta \cdot g, \Theta \cdot 
    k_{\bbeta}(\zeta, \zeta) u \rangle_{\cM} \\
    & = \langle g, k_{\bbeta}(\cdot, \zeta) \Theta(\zeta)^{*} y 
    \rangle_{H^{2}(\cU)} \\
    & = g(\zeta), \Theta(\zeta)^{*} y \rangle_{\cU}\\
    & = \langle 
    \Theta(\zeta) g(\zeta), y \rangle_{\cY} = \langle f(\zeta), y 
    \rangle_{\cY}
\end{align*}
and it follows that $\cM = \cH(K_\cM)$ with $K_\cM$ as in \eqref{kerM} as 
asserted.
\end{proof}

\begin{remark} \label{R:multipliers} {\em
    Theorem \ref{T:4.1} suggests the question as to how to 
    characterize the contractive multipliers from 
    $H^{2}_{\bbeta}(\cU)$ into $H^{2}_{\bbeta}(\cY)$ in general.  If 
the multiplication operator     $M_{\Theta} \colon f \mapsto \Theta f$ is contractive from 
    $H^{2}_{\bbeta}(\cU)$ into $H^{2}_{\bbeta}(\cY)$, a standard 
    reproducing-kernel-space computation shows that 
    $$ M_{\Theta}^{*} \colon K_{\bbeta}(\cdot, \zeta) y \mapsto 
    K_{\bbeta}(\cdot, \zeta) \Theta(\zeta)^{*} y
    $$
    for each $\zeta \in {\mathbb D}$ and $y \in \cY$ from which it 
    follows that the associated kernel
    \begin{equation}   \label{deBRker}
    L(z, \zeta): = (I - \Theta(z) \Theta(\zeta)^{*}) K_{\bbeta}(z, \zeta)
    \end{equation}
    is a positive kernel on ${\mathbb D} \times {\mathbb D}$.
    By looking at the diagonal entries of this 
    kernel we see that $\Theta(z) \Theta(z)^{*} \le I$ for all $z \in {\mathbb 
    D}$.  Furthermore, by noting the action of $M_{\Theta}$ on constant 
    vectors $u \in H^{2}_{\bbeta}(\cU)$, we see that $\Theta$ is analytic.
    Conversely, it can be shown that the converse holds:  {\em any 
    contractive analytic operator-valued function $\Theta$ on ${\mathbb 
    D}$ induces a contractive multiplier from $H^{2}_{\bbeta}(\cU)$ 
    into $H^{2}_{\bbeta}(\cY)$.}  The following more general 
    formulation was suggested to us by the referee.
    
    \begin{theorem} \label{T:conmult}
	Let $\bbeta = \{ \beta_{k}\}_{k \ge 0}$ be any positive 
	non-increasing weight sequence such that $\beta_{k}^{1/k}\to 1$ as 
        $k \to \infty$.  Let $\cU$ and $\cY$ be auxiliary 
	Hilbert spaces and let $\Theta$ be an $\cL(\cU, \cY)$-valued 
	function on the open unit disk ${\mathbb D}$. 
	Then $\Theta$ is a contractive multiplier from 
	$H^{2}_{\bbeta}(\cU)$ into $H^{2}_{\bbeta}(\cY)$ if and only 
	if $\Theta$ is analytic with $\|\Theta(z) \| \le 1$ for $z \in {\mathbb D}$.
 \end{theorem}
 
 \begin{proof}
     The necessity direction follows by the same argument as sketched 
     before the statement of the theorem.
     
\smallskip

     Conversely assume that $\Theta$ is analytic with 
     contractive values on ${\mathbb D}$.  By the argument sketched 
     above, we see that $\Theta$ is a contractive multiplier if and 
     only if the associated kernel $L$ \eqref{deBRker} is a positive 
     kernel on ${\mathbb D}$.  The kernel $L$ factors in the form
     $$
     L(z, \zeta) = \frac{I_{\cY} - \Theta(z) \Theta(\zeta)^{*}}{1 - 
     z \overline{\zeta}}  \cdot (1 - z \overline{\zeta}) 
     K_{\bbeta}(z, \zeta).
     $$
     The first factor
     $\; \frac{I_{\cY} - \Theta(z) \Theta(\zeta)^{*}}{1 - z 
	\overline{\zeta}} \; $
    is a positive kernel on ${\mathbb D}$ since $\Theta$ is a 
    contractive multiplier from $H^{2}(\cU)$ into $H^{2}(\cY)$.  The 
    second factor $\; (1 - z \overline{\zeta}) K_{\bbeta}(z, \zeta)\; $
    is a positive kernel on ${\mathbb D}$ since the shift operator 
    $S_{\bbeta}$ is a contraction on $H^{2}(\bbeta)$ (here we use 
    that the weight sequence $\bbeta$ is non-increasing).  By the 
    Schur theorem about Schur products of positive semidefinite 
    matrices, it follows that the associated kernel $L$ is a positive 
    kernel since it is the product of two positive kernels (see e.g. 
    \cite[Section 1.8]{Aron}).  Note that the application of the 
    Schur theorem is fine as long as one of the kernels is 
    scalar-valued.  It now follows that indeed $\Theta$ is a 
    contractive multiplier from $H^{2}_{\bbeta}(\cU)$ into 
    $H^{2}_{\bbeta}(\cY)$.
     \end{proof}
    
 We note that this result sharpens the result of Giselsson-Olofsson 
\cite[Proposition 4.1]{GO}.
} \end{remark}

\begin{remark} \label{R:c-i-ss}
{\em It is of interest to consider Theorem \ref{T:4.1}
    for the case condition (1) is strengthened to}
    ($1^{\prime}$) $\; \cM$ is isometrically contained in
        $H^2_\bbeta(\cY)$.
{\em  Unlike the classical Hardy space case ($\beta_{j} = 1$ for $j = 
0,1,2, \dots$), condition (3) in Theorem \ref{T:4.1} is not automatic 
and it may not be the case that $\cM = \Theta \cdot 
H^{2}_{\bbeta}(\cU)$ with $\Theta$ equal to a partially isometric 
multiplier; see Theorem 4.1 in \cite{GO} for a related result.  For 
more detailed discussion of the case where 
$\beta_{j} = \frac{ j!(n-1)!}{(n+j-1)!}$ for a positive integer $n$, see Remark 7.4 in 
\cite{BBberg}.
}\end{remark}

\subsection{Isometric representations of $S_{\bbeta}$-invariant 
subspaces via inner function families}

In this section we obtain a finer representation for 
$S_{\bbeta}$-invariant subspaces using a generalization of inner 
functions which we call an {\em inner family} (see Definition 
\ref{D:infam} below). 

\smallskip

We start with a general observation. If the subspace $\cM \subset  H^2_\bbeta(\cY)$ is 
$S_\bbeta$-invariant, 
then $\cM^{\perp}$ is $S_\bbeta^*$-invariant and, by part (3) of 
Theorem \ref{T:1.2}, we may find a $\bbeta$-output stable exactly 
$\bbeta$-observable  pair $(C,A)$ so that $\cM^{\perp} =
    \operatorname{Ran} \cO_{\bbeta, C,A}$; in fact, we may take $(C,A)$ to be the
    model output pair $(C,A) = (E|_{\cM^{\perp}},
    S_{\bbeta}^{*}|_{\cM^{\perp}})$, and $\cO_{\bbeta,C,A}$ amounts to
    the inclusion map of $\cM^{\perp}$ into $H^2_\bbeta(\cY)$. 

\smallskip

Since 
$\|{\cO}_{\bbeta,C,A}x\|_{H^2_\bbeta(\cY)}^2=
\langle \cG_{\bbeta,C,A}x, \, x\rangle_{\cX}$ for every $x\in\cX$, it follows 
that 
$\cM^{\perp}$ is a reproducing kernel Hilbert space  with reproducing kernel 
$$
K_{\cM^\perp}(z,\zeta) = CR_\bbeta(zA) \cG_{\bbeta,C,A}^{-1}R_\bbeta(\zeta 
A)^*C^{*}.
$$
It then follows that $\cM = \left( \cM^{\perp} \right)^{\perp}$ has reproducing 
kernel
\begin{equation}   
\label{kM}
K_{\cM}(z,\zeta) = K_\bbeta(z,\zeta)\cdot I_{\cY} -
CR_\bbeta(zA) \cG_{\bbeta,C,A}^{-1}R_\bbeta(\zeta A)^*C^{*}.
    \end{equation}
Since 
$$
\bigcap_{k \ge 0} S_{\bbeta}^{k} \cM \subset \bigcap_{k \ge 0} S_{\bbeta}^{k}
H^2_\bbeta(\cY) = \{0\},
$$
we can  decompose $\cM$ into the orthogonal sum
\begin{equation}
\cM=\bigoplus_{k=0}^\infty \left(S^{k}_\bbeta\cM\ominus S^{k+1}_\bbeta\cM\right).
\label{cmdec} 
\end{equation}
To compute  the reproducing kernel for the subspace $S^{k}_\bbeta\cM\ominus 
S^{k+1}_\bbeta\cM$, we first characterize the space  $(S^k_{\bbeta} \cM)^{\perp}$
in terms of the shifted observability operator ${\Ob}^{(k)}_{\bbeta,C,A}$ defined 
in 
\eqref{defR}.
\begin{proposition}   \label{P:SMperp}
    The space $(S^k_{\bbeta} \cM)^{\perp}$ is characterized as
 \begin{equation}   \label{SMperp}
     \left( S^k_{\bbeta} \cM\right)^{\perp} = 
\left(\bigoplus_{j=0}^{k-1}S_\bbeta^j\cY\right)\bigoplus
    S_\bbeta^k \operatorname{Ran} {\Ob}^{(k)}_{\bbeta,C,A}
 \end{equation}
 where we identify the first term with the subspace of polynomials of degree at 
most $k-1$  
in $H^2_\bbeta(\cY)$.
 \end{proposition}

\begin{proof}
We wish to characterize all functions $f(z)={\displaystyle
\sum_{j=0}^{\infty} f_{j} z^{j}}$ which are orthogonal to $S_{\bbeta}^k \cM$ 
in $H^2_\bbeta(\cY)$. We may write 
$f(z)  = p(z) + z^k \widetilde f(z)$ where $p(z)={\displaystyle\sum_{j=0}^{k-1} 
f_{j} z^{j}}$.
Clearly polynomials of degree at most $k-1$ are orthogonal to
$S^k_{\bbeta} \cM$, so it suffices to characterize which functions of the   
form $z^k \widetilde f(z)$ are orthogonal to $S^k_{\bbeta} \cM$. To this end, 
observe that 
$S_\bbeta^k\widetilde{f}$ is orthogonal to $S^k_{\bbeta} \cM$ if and only if the 
function 
$(S^k_{\bbeta})^*S^k_{\bbeta}\widetilde{f}$ belongs to 
$\cM^\perp= \operatorname{Ran}\cO_{\bbeta,C,A}$.
It follows from the formula \eqref{3.1a} that 
$$
(S^k_{\bbeta})^*S^k_{\bbeta} \colon  \sum_{j=0}^\infty 
\widetilde{f}_jz^j\mapsto \sum_{j=0}^\infty 
\frac{\beta_{j+k}}{\beta_j} \, \widetilde{f}_jz^j.
$$
We thus conclude that $S_\bbeta^k\widetilde{f}$ is orthogonal to $S^k_{\bbeta} 
\cM$ if and only 
if there exists a vector  $x\in\cX$ such that 
$$
\sum_{j=0}^\infty \frac{\beta_{j+k}}{\beta_j}
\, \widetilde{f}_jz^j=CR_\bbeta(zA)x=\sum_{j=0}^\infty \left(\beta_j^{-1}\cdot 
CA^jx\right)z^j.
$$ 
Equating the corresponding Taylor coefficients gives
$$
\widetilde{f}_j=\beta_{j+k}^{-1}\cdot CA^jx\quad\mbox{for all}\quad j\ge 0
$$
and therefore,
$$
\widetilde{f}(z)=\sum_{j=0}^\infty \widetilde{f}_jz^j=\sum_{j=0}^\infty 
\left(\beta_{j+k}^{-1}\cdot CA^jx\right)z^j={\Ob}^{(k)}_{\bbeta,C,A}x.
$$
Thus, $\widetilde{f}\in\operatorname{Ran} {\Ob}^{(k)}_{\bbeta,C,A}$.
As the analysis is necessary and sufficient, the result follows.
    \end{proof}

    With this result in hand, it is straightforward to derive the   
    kernel function for the space $S_{\bbeta}^k \cM$ with respect to the   
    metric inherited from $H^2_\bbeta(\cY)$.

   \begin{proposition}  \label{P:6.2}  
Let $\cM$ be a closed shift-invariant subspace of $H^2_\bbeta(\cY)$ with 
reproducing kernel given by \eqref{kM}. Then the reproducing kernel functions for 
$S^k_\bbeta\cM$
and $S^{k}_\bbeta\cM\ominus S^{k+1}_\bbeta\cM$ 
are given  by
\begin{align}
& K_{S^k_\bbeta\cM}(z,\zeta)=  z^k\overline\zeta^k 
\left(R_{\bbeta,k}(z\overline{\zeta})I_{\cY}
-CR_{\bbeta,k}(zA)\left(\Gr^{(k)}_{\bbeta,C,A}\right)^{-1}R_{\bbeta,k}(\zeta 
A)^*C^*\right),\label{kscm}\\
& K_{S^{k}_\bbeta\cM\ominus S^{k+1}_\bbeta\cM}(z,\zeta)=  z^{k}\overline\zeta^{k}
\left(\beta_{k}^{-1}I_\cY-CR_{\bbeta,k}(zA)
\left(\Gr^{(k)}_{\bbeta,C,A}\right)^{-1}R_{\bbeta,k}(\zeta A)^*C^*\right.\notag\\
& \hspace{45mm}  \left.
+z\overline\zeta CR_{\bbeta,k+1}(zA)
\left(\Gr^{(k+1)}_{\bbeta,C,A}\right)^{-1}R_{\bbeta,k+1}(\zeta A)^*C^*
\right).
\label{kdif}
\end{align}
\end{proposition}

\begin{proof}  We first derive the kernel $\boldsymbol{\mathfrak{K}}$
    for the space $S^k_{\bbeta} \operatorname{Ran} {\Ob}^{(k)}_{\bbeta,C,A}$ 
(with inner product induced by $H^2_\bbeta(\cY)$).  By \eqref{defRa},
$$
\|S^k_\bbeta {\Ob}^{(k)}_{\bbeta,C,A}x\|^2_{H^2_\bbeta(\cY)}
=\left\langle \Gr^{(k)}_{\bbeta,C,A} x, \, x\right\rangle_{\cX},
$$
and thus by the general principle from \cite{NFRKHS}, it follows that
the reproducing kernel for $S^k_{\bbeta} \operatorname{Ran} 
{\Ob}^{(k)}_{\bbeta,C,A}$ 
is given by
\begin{equation}
\boldsymbol{\mathfrak{K}}_k(z,\zeta)=  z^k\overline\zeta^k
CR_{\bbeta,k}(zA)\left({\Gr}^{(k)}_{\bbeta,C,A}\right)^{-1}R_{\bbeta,k}(\zeta 
A)^*C^*.
\label{deffrakk}  
\end{equation}
From the formula \eqref{SMperp} for $(S_{\bbeta} \cM)^{\perp}$, we deduce that
\begin{align}
    S^k_{\bbeta} \cM & = \left(\bigoplus_{j=0}^{k-1}S_\bbeta^j\cY\right)^{\perp} 
\bigcap 
\left(S^k_{\bbeta} \operatorname{Ran}{\Ob}^{(k)}_{\bbeta,C,A} 
\right)^{\perp}\notag  \\
    & = S^k_{\bbeta}H^2_\bbeta(\cY) \ominus  S^k_{\bbeta} 
\operatorname{Ran}{\Ob}^{(k)}_{\bbeta,C,A}.
    \label{SM}
\end{align}
Since the reproducing kernel of the subspace $S_\bbeta^j\cY$ of $H^2_\bbeta(\cY)$ 
is 
$z^j\overline{\zeta}^j \beta_j^{-1} I_{\cY}$, we deduce that 
$S^k_{\bbeta} H^2_\bbeta(\cY)=H^2_\bbeta(\cY)\ominus 
{\displaystyle\left(\bigoplus_{j=0}^{k-1}S_\bbeta^j\cY\right)}$
has reproducing kernel
\begin{align*}
K_{S^k_{\bbeta} H^2_\bbeta(\cY)}(z, \zeta)&=
(K_\bbeta(z, \zeta) - 
\sum_{j=0}^{k-1}\beta_j^{-1}z^j\overline{\zeta}^j)I_{\cY} \notag \\
&=(R_\bbeta(z\overline{\zeta})-\sum_{j=0}^{k-1}\beta_j^{-1}z^j\overline{\zeta}^j 
)I_{\cY}= z^{k} \overline{\zeta}^{k} R_{\bbeta,k}(z\overline{\zeta})I_\cY. 
\end{align*}
Hence from \eqref{SM} we deduce that
$$
K_{S^k_{\bbeta} \cM} (z,\zeta) = K_{S^k_{\bbeta} H^2_\bbeta(\cY)}(z, 
\zeta) -\boldsymbol{\mathfrak K}_k(z,\zeta) 
$$
and formula \eqref{kscm} follows from \eqref{deffrakk}. Finally, since
$$
K_{S^{k}_{\bbeta} 
H^2_\bbeta(\cY)}(z,\zeta)-K_{S^{k+1}_{\bbeta}H^2_\bbeta(\cY)}(z,
\zeta)=z^{k}\overline\zeta^{k}\beta_{k}^{-1} I_\cY,
$$
we have from the two latter equalities
\begin{align*}
K_{S^{k}_\bbeta\cM\ominus S^{k+1}_\bbeta\cM}(z,\zeta)&=K_{S^{k}_{\bbeta} \cM} (z, 
\zeta)-K_{S^{k+1}_{\bbeta} \cM} (z,\zeta)\\
&=K_{S^{k}_{\bbeta} H^2_{\bbeta}(\cY)}(z, 
\zeta)-\boldsymbol{\mathfrak K}_{k}(z, \zeta)\notag\\
&\qquad-K_{S^{k+1}_{\bbeta}H^2_\bbeta(\cY)}(z,\zeta)+\boldsymbol{\mathfrak 
K}_{k+1}(z, \zeta)\\
&=z^{k}\overline\zeta^{k}\beta_{k}^{-1} I_\cY-\boldsymbol{\mathfrak K}_{k}(z, 
\zeta)+
\boldsymbol{\mathfrak K}_{k+1}(z, \zeta)
\end{align*}
which implies \eqref{kdif} due to \eqref{deffrakk}.
\end{proof}
\begin{lemma}  \label{L:6.8}
Given an integer $k\ge 1$ and an exactly $\bbeta$-observable $\bbeta$-output 
stable pair
$(C,A)$, construct operators $B_k \in \cL(\cU_k, \cX)$ and $D_k \in \cL(\cU_k, 
\cY)$
as in Lemma \ref{L:5.6} and let $\Theta_{k}$ be the associated
function given by \eqref{2.10}. Then the kernel \eqref{kdif} can be factored as
\begin{equation}
K_{S^{k}_\bbeta\cM\ominus 
S^{k+1}_\bbeta\cM}(z,\zeta)=z^{k}\overline{\zeta}^{k}\Theta_{k}(z)\Theta_{k}(\zeta)^*.
\label{id6}
\end{equation}
\end{lemma}
\begin{proof}  By Lemma \ref{L:frakcoisom}, identity \eqref{frakkerid} holds. 
Multiplying both 
parts
of \eqref{frakkerid} by $z^k\bar\zeta^k$ and combining the obtained equality with
\eqref{kdif} we get \eqref{id6}.
\end{proof}

\begin{definition}
In what follows, we will use the symbol $\bigvee$ for the closed linear span.
\label{D:infam}
A family of operator-valued functions $\{
\Theta_{k} \colon {\mathbb D} \to \cL(\cU_{k}, \cY)\}_{k=0}^{\infty}$ will be 
said to be a 
{\em $\bbeta$-inner function family} if, for each $k\ge 0$, we have:
\begin{enumerate}
    \item $M_{\Theta_{k}} \colon S_{\bf 1}^{k} \cU_{k} \to H^2_\bbeta(\cY)$ is  
isometric,
    \item $M_{\Theta_{k}} \left( S_{\bf 1}^{k} \cU_{k} \right)$ is
    orthogonal (in $H^2_\bbeta(\cY)$) to ${\displaystyle\bigvee_{\ell>k}
    M_{\Theta_{\ell}} S_{\bf 1}^{\ell} \cU_{\ell}}$,
    \item $S_{\bbeta}^{k+1} M_{\Theta_{k}} \cU_{k} \subset
    {\displaystyle\bigvee_{\ell>k} S_{\bbeta}^{\ell}
    M_{\Theta_{\ell}} \cU_{\ell}}$.
  \end{enumerate}
  An equivalent operator-theoretic characterization of the {\em 
  $\bbeta$-inner function family} property is:
  \begin{enumerate}
      \item[(1$^{\prime}$)] The multiplication operator
      $$
M_{\Theta} = \begin{bmatrix} M_{\Theta_{0}} & M_{\Theta_{1}} &
M_{\Theta_{2}} & \cdots \end{bmatrix}  \colon   
\bigoplus_{k=0}^{\infty} S_{1}^{k} \cU_{k} \to H^2_\bbeta(\cY)
$$
maps the {\em time-varying Hardy space}
$H^{2}(\{\cU_{k}\}_{k \ge 0}): = {\displaystyle\bigoplus_{k=0}^{\infty} S_{\bf 
1}^{k}
\cU_{k}}$ (where ${\bf u}={\displaystyle\bigoplus_{k=0}^{\infty} z^{k} u_{k}} \in
H^{2}(\{\cU_{k}\})$ is assigned the Hardy-space norm $\| {\bf u}\|^{2} = 
{\displaystyle\sum_{k=0}^{\infty}
\| u_{k} \|^{2}}$) isometrically into $H^2_{\bbeta}(\cY)$, and

\item[(2$^{\prime}$)] there is a strictly lower triangular matrix $L = 
[L_{ij}]_{i,j = 0,1,2,\dots}$ (so $L_{ij} = 0$ for $i \le j$) with 
entries $ L_{ij} \in \cL(\cU_{j}, \cU_{i})$
so that 
$$
S_{\bbeta} \begin{bmatrix} M_{\Theta_{0}}  & M_{\Theta_{1}} S_{1} & 
M_{\Theta_{2}} S_{1}^{2} & \cdots \end{bmatrix}  = 
 \begin{bmatrix} M_{\Theta_{0}}  & M_{\Theta_{1}} S_{1} & 
M_{\Theta_{2}} S_{1}^{2} & \cdots \end{bmatrix} L.
$$
\end{enumerate}
\end{definition}

\smallskip

We conclude that
if $\{ \Theta_{k}\}_{k \ge 0}$ is a $\bbeta$-inner function family and if we set
$$  
\cM = \bigoplus_{k=0}^{\infty} \Theta_{k} S_{\bf 1}^{k} \cU_{k} \subset 
H^2_\bbeta(\cY),
$$
it then follows that $\cM$ is $S_{\bbeta}$-invariant and that the multiplication 
operator
\begin{equation}  \label{tvmultop}
M_{\Theta} = \begin{bmatrix} M_{\Theta_{0}} & M_{\Theta_{1}} &
M_{\Theta_{2}} & \cdots \end{bmatrix}  \colon   
\bigoplus_{k=0}^{\infty} S_{1}^{k} \cU_{k} \to H^2_\bbeta(\cY)
\end{equation}
maps the {\em time-varying Hardy space}
$H^{2}(\{\cU_{k}\}_{k \ge 0}): = {\displaystyle\bigoplus_{k=0}^{\infty} S_{\bf 
1}^{k}
\cU_{k}}$ unitarily onto the  $S_{\bbeta}$-invariant subspace $\cM 
\subset 
H^2_\bbeta(\cY)$.
Putting all the pieces together, we arrive at the following converse
of all these observations which amounts to our second analogue of the 
Beurling-Lax
theorem for the weighted Hardy space setting.  To follow the 
statement the reader should refer back to Definitions 
\ref{D:bbeta-real} and \ref{D:conserv-col}.

\begin{theorem}  \label{T:BL3}
Let $\cM$ be a closed $S_{\bbeta}$-invariant subspace of
$H^2_\bbeta(\cY)$.  Then there is a $\bbeta$-inner function family
$\{\Theta_{k}\}_{k \ge 0}$ so that $\cM=M_{\Theta} H^{2}(\{\cU_k\}_{k\ge 0})$ 
(with $M_{\Theta}$ as in \eqref{tvmultop}).

    Furthermore, given the $S_{\bbeta}$-invariant closed subspace $\cM \subset
    H^2_\bbeta(\cY)$, a $\bbeta$-unitary colligation family 
    realization $\left\{U_{k} = \left[ \begin{smallmatrix} A & B_{k} 
    \\ C & D_{k} \end{smallmatrix} \right] \right\}_{k \ge 0}$ for the
   $\bbeta$-inner function family representer
    $\{\Theta_{k}\}_{k\ge 0}$ for $\cM$ can be constructed according to the
    following algorithm:
    \begin{enumerate}
        \item Set $\cX = \cM^{\perp}$ and define $A \in \cL(\cX)$ and
        $C \in \cL(\cX, \cY)$ by
$$
  A = S_{\bbeta}^{*}|_{\cM^{\perp}},  \quad Cf = f(0) \quad\text{for}\quad f \in
  \cM^{\perp}.
$$
\item Construct $\left[ \begin{smallmatrix} B_{k} \\ D_{k}
\end{smallmatrix} \right]$ by solving the Cholesky factorization
problem \eqref{pr6} in Lemma \ref{L:5.6}.

\item Set $\Theta_{k}(z) = \beta_k^{-1}D_{k} + z C R_{\bbeta,k+1}(zA) B_{k}$.
\end{enumerate}
Then $\{\Theta_{k}\}_{k\ge 0}$ is a $\bbeta$-inner function family and any 
$\bbeta$-inner function family
  arises in this way.
 \end{theorem}
 
 \begin{remark} \label{R:sysinter}   {\em
     If $\{U_{k}\}_{k \ge 0}$ is the $\bbeta$-unitary colligation 
     family constructed as in Theorem \ref{T:BL3}, then the 
     associated input-output map $\bT_{\bbeta}$ \eqref{IOmap} is 
     isometric from $\bigoplus_{k=0}^{\infty} \cU_{k}$ into 
     $\ell^{2}_{\bbeta}(\cY)$; this fact follows directly from 
     Corollary \ref{C:bbetaisom} and requires only the 
     $\bbeta$-isometric property of $\{U_{k}\}_{k \ge 0}$.  From the 
     fact that $\cM = M_{\Theta} H^{2}(\{\cU_{k}\}_{k \ge 0})$ is 
     $S_{\bbeta}$-invariant, it follows that the image of $\cM$ under 
     the inverse $Z$-transform
     $$
     \cM^{\vee} = \{ \{y(k)\}_{k \ge 0} \colon \sum_{k=0}^{\infty} 
     y(k) z^{k} \in \cM \} = \bT_{\bbeta}\left( 
     \bigoplus_{k=0}^{\infty} \cU_{k}\right) \subset 
     \ell^{2}_{\bbeta}(\cY)
     $$
     is invariant under  the inverse $Z$-transform version  
     $S_{\bbeta}^{\vee}$ of $S_{\bbeta}$ 
     acting on $\ell^{2}_{\bbeta}(\cY)$:   
     $$
       S^{\vee}_{\bbeta} \colon \{y(0), y(1), y(2), \dots) \mapsto ( 
       0, y(0), y(1), \dots).
     $$
     From the arguments above one can see that the 
     $S^{\vee}_{\bbeta}$-invariance property of $\cM^{\vee}$ is a consequence 
     of the fact that $\{U_{k}\}_{k \ge 0}$ satisfies the additional 
     weighted coisometry property \eqref{wghtcoisom}.
 } \end{remark}
    
\subsection{Wandering-subspace $\bbeta$-inner functions and Beurling-Lax 
representations}

If the subspace $\cM \subset H^2_\bbeta(\cY)$ is $S_{\bbeta}$-invariant, then the
subspace $\cE: = \cM \ominus S_\bbeta\cM$ has the property that $\cE \subset
\cM$ and $\cE \perp S_{\bbeta} \cM$.  If it is the case that $\cE$
generates $\cM$ in the sense that 
$
\cM={\displaystyle\bigvee_{k\ge 0} S_\bbeta^k\cE}, \;$
then one says  that $\cM$ has the {\em wandering subspace
property} (with wandering subspace equal to $\cE$). 

\smallskip

Letting $k=0$ in formula \eqref{kdif} and recalling formulas \eqref{st2},
we conclude that the reproducing kernel for the subspace $\cE$  equals
\begin{align}
K_\cE(z,\zeta)=& \,
I_{\cY}-CR_{\bbeta}(zA)\cG_{\bbeta,C,A}^{-1}R_{\bbeta}(\zeta A)^*C^*\notag\\
&+z\overline{\zeta}CR_{\bbeta,1}(zA)\left(\Gr^{(1)}_{\bbeta,C,A}\right)^{-1}R_{\bbeta,1}(\zeta 
A)^*C^*.\notag
\end{align}

Following \cite{oljfa, olaa} we say that the function $\Theta$ is a
{\em wandering-subspace $\bbeta$-inner function}
whenever
 \begin{enumerate}
        \item $M_{\Theta} \colon \cU \to H^2_\bbeta(\cY)$ is
        isometric, and
        \item $M_\Theta \cU$ is orthogonal to $S_{\bbeta}^{\ell}
        M_\Theta \cU$ for $\ell \ge 1$.
\end{enumerate}  
In case $M_{\Theta} \cU  = \cE$ is the wandering subspace for the
$S_{\bbeta}$-invariant subspace $\cM$, then we have the Beurling-Lax-type
representation of $\cM$ as the closure of $\Theta \cdot \cU[z]$, 
where $\cU[z]$ denotes the collection of 
polynomials with coefficients in the Hilbert space $\cU$.  Construction of a
wandering-subspace Bergman-inner function (or Bergman-inner function 
for short) for the wandering subspace $\cE$ of $\cM$  
amounts to focusing on the first element $\Theta_{0}$ in the
$\bbeta$-inner function family associated with $\cM$.
Specifying Lemma \ref{L:6.8} for the case $k=0$ then leads to the
following.  The special case $\beta_{j} = \frac{j! 
(n-1)!}{(n+j-1)!}$ is treated in \cite{BBberg}; closely 
related results for this special case were obtained earlier by 
Olofsson (see \cite{olaa}, especially Theorems 4.2 and 6.1 there).

\begin{theorem}  \label{P:6.4u}
    Given a $\bbeta$-strongly stable $\bbeta$-output-pair $(C,A)$,   there exist 
operators
    $B \in \cL(\cU, \cX)$ and $D \in \cL(\cU, \cY)$ which solve the
    Cholesky factorization problem
$$
\begin{bmatrix}B \\ D\end{bmatrix}\begin{bmatrix}B^* & D^*\end{bmatrix}=
\begin{bmatrix}\left(\Gr^{(1)}_{\bbeta,C,A}\right)^{-1} & 0 \\ 0 & 
I_{\cY}\end{bmatrix}-
\begin{bmatrix}A \\ C\end{bmatrix}\cG_{\bbeta,C,A}^{-1}\begin{bmatrix}A^* & 
C^*\end{bmatrix}.
$$
Moreover, if $\Theta$ is defined by
$$
\Theta(z)=D+zCR_{\bbeta,1}(zA)B
$$
where $\left[ \begin{smallmatrix} B \\ D \end{smallmatrix} \right]$
solves \eqref{pr6} (with $k=0$), then:
\begin{enumerate}
\item The factorization $K_\cE(z,\zeta)=\Theta(z)\Theta(\zeta)^*$ holds and 
therefore,
the multiplication operator $M_{\Theta}$ maps $\cU$ onto $\cE$ unitarily.

\item The subspace $\cE$ is orthogonal to $S_\bbeta^{k}M_{\Theta}\cU$ for every 
$k\ge 1$.

\item $M_{\Theta}$ is a contractive multiplier from the Hardy space
$H^{2}(\cU)$ into $H^2_\bbeta(\cY)$.  Hence, if
$\cM$ has the wandering subspace property, we conclude that $\cM$ has
the Beurling-Lax-type representation
$$
 \cM = \text{$H^2_\bbeta(\cY)$-closure of } M_{\Theta} H^{2}(\cU).
 $$
\end{enumerate}
\end{theorem}

We note that the last statement in the theorem is a consequence of 
part (4) of Lemma \ref{L:5.1}. 

\begin{remark} \label{R:BLcompare} {\em
    We note that our companion paper \cite{BBberg} also discusses 
    several versions of Beurling-Lax-type representations for 
    weighted Hardy spaces, but for the special case where $\bbeta_{k} 
    = \frac{ k! (n-1)!}{(k+n-1)!}$ ($\bbeta = \bbeta_{\alpha}$ as in 
    \eqref{1.5} with $\alpha = n$ a positive integer).  In 
    particular, Section 6.1 above is the generalization of the 
    ``second approach'' from \cite{BBberg}, and Section 6.2 above is the 
    generalization of the ``third approach'' from \cite{BBberg}. The 
    ``first approach'' from \cite{BBberg} uses the special form 
    $\bbeta = \bbeta_{n}$ in a fundamental way and hence we have no 
    analogue of the ``first approach'' for the general weight setting 
    considered here.
    } \end{remark}

\section{Weighted Hardy spaces, characteristic functions, and operator model theory}

Theorem \ref{T:BL3} explains how closed $S_{\bbeta}$-invariant 
subspaces $\cM \subset H^{2}_{\bbeta}(\cY)$ correspond to 
$\bbeta$-inner function families $\{\Theta_{k}\}_{k \ge 0}$, 
including a transfer-function-like realization $\Theta_{k}(z) = 
\bbeta_{k}^{-1} D_{k} + z C R_{\bbeta, k+1}(zA) B_{k}$ having a nice 
system-theoretic interpretation.   It will be convenient to introduce 
the following terminology.

\begin{definition} \label{D:betaCdot0}
    We say that the Hilbert space operator $T \in \cL(\cX)$ is a 
    $*$-$\bbeta$-hypercontraction if its adjoint $A = T^{*}$ is a 
    $\bbeta$-hypercontraction (see Definition \ref{D:3.5}). 
    We say that $T$ is a $\bbeta$-$C_{\cdot 0}$ $*$-$\bbeta$ 
    hypercontraction if also $A = T^*$ is 
    $\bbeta$-strongly stable (see Definition \ref{D:4}).
    \end{definition}
    
From Lemma \ref{L:3.1} it is easily seen that any operator $T$ 
of the form 
\begin{equation}  \label{Tmodel}
T = P_{\cM^{\perp}} S_{\bbeta}|_{\cM^{\perp}}
\end{equation}
for an $S_{\bbeta}$-invariant subspace $\cM \subset 
H^{2}_{\bbeta}(\cY)$ is a $\bbeta$-$C_{\cdot 0}$ 
$*$-$\bbeta$-hypercontraction. 
To complete our operator-model theory for the class of $\bbeta$-$C_{\cdot 0}$ 
$*$-$\bbeta$-hypercon\-traction operators $T$, it remains to define
a $\bbeta$-inner characteristic function family $\{\Theta_{T,k}\}_{k 
\ge 0}$ for any $\bbeta$-$C_{\cdot 0}$ $*$-$\bbeta$-hypercon\-traction operator $T$ with the 
property that we recover $T$ up to unitary equivalence via
the formula \eqref{Tmodel} with $\cM = M_{\Theta_{T}} H^{2}(\{ 
\cU_{k}\}_{k \ge 0})$.

\smallskip

Let us suppose first only that $T \in \cL(\cX)$ is a 
$*$-$\bbeta$-hypercontraction, i.e., $A:= T^{*}$ is a 
 $\bbeta$-hypercontraction.  In 
particular, we then have that $\Gamma_{\bbeta,A}[I_{\cX}] \ge 0$ and 
hence $\Gamma_{\bbeta,A}[I_{\cX}]$ has a positive semidefinite 
square root, denoted as $D_{\bbeta,A}: = \left( \Gamma_{\bbeta, 
A}[I_{\cX}] \right)^{1/2}$. Since $A$ is  $\bbeta$-hypercontractive, a 
consequence 
of the sufficiency side of part (1) of Theorem \ref{T:1.1}  (with $H = I_{\cX}$) 
is that 
$(D_{\bbeta, A}, A)$ is a $\bbeta$-output stable pair which is also 
$\bbeta$-isometric (i.e., the inequality \eqref{3.16} holds with 
equality) by construction.  We then have all the shifted gramians 
$\Gr^{(k)}_{\bbeta, D_{\bbeta, A},A}$ given by \eqref{defR} or 
\eqref{defRa} defined and positive semidefinite.  Furthermore, by 
Proposition \ref{P:3.0}  we know that the weighted Stein identities 
\eqref{7.3} hold for $k=0,1,2, \dots$.

\smallskip

As in Remark \ref{R:unitarycol}, for 
$k=0,1,2, \dots$ we let $\cX_{k}$ be the completion of $\cX$ in the 
$\cX_{k}$-metric given by
$$
   \langle x, x' \rangle_{\cX_{k}}: = \langle 
   \Gr^{(k)}_{\bbeta,D_{\bbeta,A},A} x, x' \rangle_{\cX},
$$
where any elements of $\cX$ having zero $\cX_{k}$-norm are identified 
with zero. Similarly, we let $\cY_{k}$ denote the space 
$\cD_{\bbeta, A}: = \overline{\operatorname{Ran}} D_{\bbeta, A}$ but 
with inner product given by
$$
   \langle D_{\bbeta, A}x, D_{\bbeta, A} x' \rangle_{\cY_{k}} = 
   \beta_{k}^{-1}\cdot \langle D_{\bbeta, A}x, D_{\bbeta, A} x' 
   \rangle_{\cX}.
$$
We next let $\bA_{k} \in \cL(\cX_{k}, \cX_{k+1})$ denote 
the operator $A \in \cL(\cX)$, but viewed as acting from $\cX_{k}$ 
into $\cX_{k+1}$.  Similarly, we let $\bC_{k}$ denote the operator 
$D_{\bbeta, A}$ but viewed as an operator acting from $\cX_{k}$ into 
$\cY_{k}$.   Due to the validity of the weighted Stein identities 
\eqref{7.3}, we see that the operator $\left[ \begin{smallmatrix} 
\bA_{k} \\ \bC_{k} \end{smallmatrix} \right]$ extends uniquely to a 
well-defined isometry acting from $\cX_{k}$ into $\left[ 
\begin{smallmatrix} \cX_{k+1} \\ \cY_{k} \end{smallmatrix} \right]$.
 In particular, we see that there is a uniquely determined unitary 
 transformation $\bomega_{k} \colon \cD_{\bA_{k}} \to \cY_{k}$
 (here $\cD_{\bA_{k}} = \overline{\operatorname{Ran}} D_{\bA_{k}}$ where 
 we use the notation $D_{\bA_{k}} = (I - \bA_{k}^{*} \bA_{k})^{1/2}$) so 
 that $\bC_{k} = \bomega_{k} D_{A_{k}}$ (note that for $k=0$ we have 
 $\Gr^{(0)}_{\bbeta, D_{\bbeta,A},A} = I_{\cX}$ and $\bomega_{0} = 
 I_{\cD_{\bbeta,A}}$).  We next introduce the other defect operator
 $D^{(k)}_{T}: = D_{\bA_{k}^{*}} = \left(I - \bA_{k} \bA_{k}^{*} 
 \right)^{1/2}$ and the coefficient space $\cU_{k} = 
 \overline{\operatorname{Ran}}( D_{\bA_{k}^{*}})^{1/2}$.
 We then define a unitary colligation matrix $\bU_{T,k}$
    by
$$
 \bU_{T,k} = \begin{bmatrix} \bA_{k} & D_{\bA_{k}^{*}} \\ \bC_{k} & 
 -\bomega_{k} \bA_{k}^{*} \end{bmatrix}   \colon 
 \begin{bmatrix} \cX_{k} \\ \cU_{k} \end{bmatrix}  \to
\begin{bmatrix} \cX_{k+1} \\ \cY_{k} \end{bmatrix}.
$$
  Finally, we define the {\em characteristic function family} for the 
  $*$-$\bbeta$-hypercon\-traction $T = A^{*}$ to be the family of 
  functions 
  $\{\bTheta_{T,k}\}_{k=0,1,2, \dots}$ where
$$
\bTheta_{T,k}(z) =  \left. \left(  -\beta_{k}^{-1} \bomega_{k} \bA_{k}^{*} +
z \bC_{k}  R_{\bbeta, k+1}(z \bA_{k}) D_{\bA_{k}^{*}} \right) \right|_{\cU_{k}} 
\colon  \cU_{k} \to \cY_{k}.
$$

We now impose the additional hypothesis that {\em $T$ is 
$\bbeta$-$C_{\cdot 0}$}, i.e., {\em $A = T^{*}$ is $\bbeta$-strongly stable}.
Then Lemma \ref{L:3.12} tells us that 
 $(D_{\bbeta, A}, A)$ is exactly $\bbeta$-observable, the 
 zero-shifted gramian satisfies 
 $\Gr^{(0)}_{\bbeta, D_{\bbeta, A},A} =  \cG_{\bbeta, D_{\bbeta, 
 A},A} = I_{\cX}$, and as a consequence of Proposition 
 \ref{P:exactbetaobs} we have that all the shifted gramians 
 $\Gr^{(k)}_{\bbeta, D_{\bbeta, A}, A}$ are strictly positive 
 definite.  Then $\cX_{k}$ and $\cX$ are the same as sets, i.e.,
 the inclusion maps $i^{(k)} \colon \cX \to \cX_{k}$ are all 
 invertible for $k=0,1,2, \dots$.  If we introduce the adjusted 
 colligation 
 \begin{align}   
 U_{T,k}  & = \begin{bmatrix} (i^{(k+1)})^{-1} & 0 \\ 0 & \beta_{k}^{1/2}I_{\mathcal Y} 
\end{bmatrix} \bU_{T,k} \begin{bmatrix} i^{(k)} & 0 \\ 0 & 
I_{\cU_{k}} \end{bmatrix}  \notag\\ 
& = \begin{bmatrix} (i^{(k+1)})^{-1} \bA_{k} i^{(k)} & (i^{(k+1)})^{-1} 
\bB_{k} \\ \beta_{k}^{1/2} \bC_{k} & \beta_{k}^{1/2} \bD_{k} 
\end{bmatrix},
\label{bUTk}
\end{align}
 then $U_{T,k}$ has the form
 $$
  U_{T,k} = \begin{bmatrix} A & B_{k} \\ C & D_{k} \end{bmatrix}
  \colon \begin{bmatrix} \cX \\ \cU_{k} \end{bmatrix} \to
\begin{bmatrix} \cX \\ \cD_{\bbeta, C,A} \end{bmatrix}
 $$
 and the unitary property of $\bU_{T,k}$ translates to $U_{T,k}$ 
 satisfying the relations \eqref{isom} as well as \eqref{wghtcoisom}.
 We then define the {\em adjusted characteristic function family to be
 $\{\Theta_{T,k}\}$ where we set}
 \begin{equation}   \label{charfuncfam}
     \Theta_{T,k}(z) = \beta_{k}^{-1} D_{k} + z C R_{\bbeta,k+1}(zA) B_{k} \colon 
    \cU_{k} \to \cD_{\bbeta,C,A} 
 \end{equation}
 where $A,C,B_{k},D_{k}$ are determined ultimately from the $\bbeta$-$C_{\cdot 
 0}$ $*$-$\bbeta$-hypercon\-trac\-tion as above.

\smallskip
 
 Alternatively, given the $\bbeta$-$C_{\cdot 0}$ 
 $*$-$\bbeta$-hypercontraction $T$, the characteristic 
 function family $\{\Theta_{T,k}\}_{k \ge 0}$ can be defined more 
 directly as follows.  Set 
$$
A = T^{*}\quad \mbox{and}\quad C = D_{\bbeta,A}.
$$  
Then  $(C,A)$ is an exactly $\bbeta$-observable, $\bbeta$-output stable 
 (also $\bbeta$-isometric) output pair.  Construct operators $B_{k} 
 \in \cL(\cU_{,}, \cX)$ and $D_{k} \in \cL(\cU_{,}, \cY)$ with 
 $\left[ \begin{smallmatrix} B_{k} \\ D_{k} \end{smallmatrix} 
 \right]$ injective by solving the Cholesky factorization problem 
 \eqref{pr6} as in Lemma \ref{L:5.6}, and then set
 $\Theta_{T,k}(z) = \beta_{k}^{-1} D_{k} + z C R_{\bbeta, k+1}(zA) 
 B_{k}$.  Then $\Theta_{T,k}$ is uniquely determined by $T$ up to  a 
 unitary change-of-basis transformation $\sigma_{k}$ on the input 
 space $\cU_{k}$ for $k=0,1,2, \dots$.  We shall also call any such 
 choice $\{\Theta_{T,k}\}_{k \ge 0}$ a {\em characteristic function family}
 for the $\bbeta$-$C_{\cdot 0}$-$*$-$\bbeta$-hypercontraction $T$.
 
\smallskip

 We note that the first element $\Theta_{T,0}$ in the characteristic 
 function family in the first form given above using the defect 
 operators $D_{\bA_{0}^{*}}$ for the special case where $\beta_{j} = 
 \frac{j! (n-1)!}{(j+n-1)!}$ amounts to the characteristic function 
 for the $C_{0 \cdot}$ $n$-hypercontraction  $A = T^*$ introduced and studied
 by  Olofsson \cite{oljfa, olaa}.
 
 \begin{theorem}   \label{T:funcmodel}
     Suppose that $T \in \cL(\cX)$ is a $\bbeta$-$C_{\cdot 0}$ 
     $*$-$\bbeta$-hypercontraction as above.  Let 
     $\{\Theta_{T,k}\}_{k \ge 0}$ be the adjusted characteristic 
     function family for $T$ as given  by \eqref{charfuncfam} and 
     \eqref{bUTk}.  Then $\{\Theta_{T,k}\}_{k \ge 0}$ is an inner 
     function family and $T$ is unitarily equivalent to the operator 
     $P_{\cM^{\perp}} S_{\bbeta}|_{\cM^{\perp}}$, where
     $$
     \cM = \begin{bmatrix} M_{\Theta_{T,0}} &  M_{\Theta_{T,1}} & 
      M_{\Theta_{T,2}} & \cdots \end{bmatrix} 
     H^{2}({\{\cU_{k}\}_{k \ge 0}})
     $$
     is the $S_{\bbeta}$-invariant subspace associated with the inner 
     function family $\{\Theta_{T,k}\}$.

\smallskip
     
     Furthermore, if $T'$ is another $\bbeta$-$C_{\cdot 0}$ $*$-$\bbeta$-hypercontraction 
     on a Hilbert space $\cX'$ with 
     characteristic function family $\{ \Theta_{T',k}\}$, then $T$ 
     and $T'$ are unitarily equivalent if and only if the 
     adjusted characteristic function families $\{\Theta_{T,k}\}_{k \ge 0}$ 
     and $\{\Theta_{T',k}\}_{k \ge 0}$ {\em coincide} in the 
     following sense:  for each $k=0,1, \dots$ there are unitary 
     operators $\sigma_{k} \colon \cU_{k} \to \cU'_{k}$ and $\tau \colon 
     \cD_{\bbeta,A} \to \cD_{\bbeta, A'}$ so that
$$
     \tau \Theta_{T,k}(z) =  \Theta_{T',k}(z) \sigma_{k}\quad \text{for 
     each}\quad z \in {\mathbb D}.
$$
  \end{theorem}
  
  \begin{proof}  Suppose that $T \in \cL(\cX)$ is a $\bbeta$-$C_{\cdot 0}$ 
      $*$-$\bbeta$-hypercontraction and set $A = T^{*}$.  As was 
      remarked in the introductory remarks to this section, it 
      follows that $(C,A): = (D_{\bbeta},A)$ is an exactly 
      $\bbeta$-observable and $\bbeta$-isometric output pair with 
      associated gramian $\cG_{\bbeta, C,A} = \Gr^{(0)}_{\bbeta,C,A} 
      = I_{\cX}$.  Theorem \ref{T:beta-stablemodel}  then tells us that $A$ is 
      unitarily equivalent to $S_{\bbeta}^{*}|_{\operatorname{Ran} 
      \cO_{\bbeta,C,A}}$.   Then the subspace
$$\cM: = \left( \operatorname{Ran} 
      \cO_{\bbeta,C,A} \right)^{\perp} \subset 
      H^{2}_{\bbeta}(\cD_{\bbeta, A})
$$ 
is $S_{\bbeta}$-invariant, and 
      hence by Theorem \ref{T:BL3} has a representation as $\cM = 
      M_{\Theta} H^{2}(\{\cU_{k}\}_{k \ge 0})$ for $\Theta = 
      \begin{bmatrix} \Theta_{0} & \Theta_{1} & \Theta_{2} & \cdots 
	  \end{bmatrix}$ equal to the multiplier associated with an 
	  inner function family $\{\Theta_{k}\}_{k \ge 0}$.  Furthermore, the 
	  formulas \eqref{charfuncfam} and \eqref{bUTk} for 
	  $\Theta_{k}$ amount to one possible way to construct the 
	 associated $\bbeta$-inner function family  $\{ \Theta_{k}\}_{k \ge 
	 0}$ according to the prescriptions of Theorem \ref{T:BL3}.
      It remains now only to verify the uniqueness statement.
      
\smallskip
      
      Suppose first that $T$ and $T'$ are unitarily equivalent 
      $\bbeta$-$C_{\cdot 0}$ $*$-$\bbeta$-hypercon\-tractions on Hilbert spaces 
      $\cX$ and $\cX'$ respectively.  Thus there is a unitary 
      operator $\omega \colon \cX \to \cX'$ so that $\omega T = T' 
      \omega$, and hence also $\omega A = A' \omega$ where $A=T^{*}$ 
      and $A' = T^{\prime *}$.  We next set $\tau = 
      \omega|_{\cD_{\bbeta, A}}$.  Then it is easily verified that 
      $\tau$ is unitary from $\cD_{\bbeta, A}$ onto $\cD_{\bbeta, 
      A'}$.  The fact that the Cholesky factorization problem 
      \eqref{pr6} has a unique injective solution up to a unitary 
      transformation $\sigma_{k}: \,  \cU \to \cU'_{k}$ implies that the 
      colligation matrices $U_{T,k}$ and $U_{T',k}$ are related by
      $$
      \begin{bmatrix} \omega & 0 \\ 0 & \tau \end{bmatrix} U_{T,k} =
	  U_{T',k} \begin{bmatrix} \omega & 0 \\ 0 & \sigma_{k} 
      \end{bmatrix}.
      $$
      From this relation it follows that $\{ \Theta_{T,k}\}$ and $\{ \Theta_{T',k}\}$ 
      coincide in the sense given in the statement of the theorem.
      
\smallskip

      Conversely, suppose that $\{\Theta_{T,k}\}$ and $\{ 
      \Theta_{T',k}\}$ coincide via unitary operators $\tau \colon 
      \cD_{\bbeta, A} \to \cD_{\bbeta,A'}$ and $\sigma_{k} \colon 
      \cU_{k} \to \cU'_{k}$.  Set 
      $$
      \cM : = M_{\Theta} H^{2}_{\bbeta}(\{\cU_{T,k}\}_{k \ge 0}), \quad
       \cM' : = M_{\Theta} H^{2}_{\bbeta}(\{\cU_{T',k}\}_{k \ge 0}).
      $$
      We know that the operator $A = T^{*}$ is unitarily equivalent to 
      $S_{\bbeta}^{*}|_{\cM^{\perp}}$ while $A' = T^{\prime *}$ is 
      unitarily equivalent to $S_{\bbeta}^{*}|_{\cM^{\prime \perp}}$.
      From Lemma \ref{L:6.8}   we see that 
       $K_{S_{n}^{k} \cM \ominus S_{\bbeta}^{k+1}}(z, \zeta) = z^{k}
      \overline{\zeta}^{k} \Theta_{k}(z) \Theta_{k}(\zeta)^{*}$. 
      Combining this with the decomposition \eqref{cmdec} then gives
      $$
      K_{\cM}(z, \zeta) = \sum_{k=0}^{\infty} z^{k}
      \overline{\zeta}^{k} \Theta_{k}(z) \Theta_{k}(\zeta)^{*},
      $$
      from which we get
      $$
      K_{\cM^{\perp}}(z, \overline{\zeta}) = R_{\bbeta}(z \overline{\zeta}) 
      I_{\cD_{\bbeta, A}} - \sum_{k=0}^{\infty} z^{k}
      \overline{\zeta}^{k} \Theta_{T,k}(z) \Theta_{T,k}(\zeta)^{*}.
      $$
      A similar analysis gives the reproducing kernel for the 
      subspace $\cM^{\prime \perp}$:
      $$
      K_{\cM^{\prime \perp}}(z, \overline{\zeta}) = R_{\bbeta}(z \overline{\zeta}) 
      I_{\cD_{\bbeta, A'}} - \sum_{k=0}^{\infty} z^{k}
      \overline{\zeta}^{k} \Theta_{T',k}(z) \Theta_{T',k}(\zeta)^{*}.
      $$
      The fact that $\{ \Theta_{T, k}\}$ and $\{ \Theta_{T',k}\}$ 
      coincide then tells us that
      $$
      \tau K_{\cM^{ \perp}}(z, \overline{\zeta}) =
       K_{\cM^{\prime \perp}}(z, \overline{\zeta}) \tau.
      $$
      It is then easily seen that the map 
      $$ X \colon f(z) \mapsto \tau f(z)
      $$
      is unitary from $\cM^{\perp}$ onto $\cM^{\prime \perp}$ and 
      satisfies the intertwining relation 
      $$
      X \left(S_{\bbeta}|_{\cM^{\perp}}\right)^{*} = 
      \left(S_{\bbeta}|_{\cM^{\prime \perp}}\right)^{*} X
      $$
      and hence
      $ \left(S_{\bbeta}|_{\cM^{\perp}}\right)^{*}$ and 
      $\left(S_{\bbeta}|_{\cM^{\prime \perp}}\right)^{*}$ are 
      unitarily equivalent.  As it has already been observed that $A$ 
      is unitarily equivalent to $ \left(S_{\bbeta}|_{\cM^{\perp}}\right)^{*}$ 
      and $A'$ is unitarily equivalent to 
      $\left(S_{\bbeta}|_{\cM^{\prime \perp}}\right)^{*}$,
      it follows that $A$ and $A'$ (and hence also $T$ and $T'$) are 
      unitarily equivalent to each other.
    \end{proof}
    
   Given a $\bbeta$-inner function family $\{\Theta_{k}\}_{k \ge 0}$, 
 then $\cM : = M_{\Theta} H^{2}_\bbeta(\{\cU_{k}\}_{k \ge 0})$ 
 is a closed $S_{\bbeta}$-invariant subspace of $H^{2}_{\bbeta}(\cY)$.  We 
 may then construct colligation matrices $U_{k} = \left[ 
 \begin{smallmatrix} A & B_{k} \\ C & D_{k} \end{smallmatrix}\right]$ 
     satisfying the metric constraints \eqref{isom} and 
     \eqref{wghtcoisom} so that we also have $\cM = M_{\Theta(\cM)} 
     H^{2}(\{ \cU_{k}\})$, where $\Theta(\cM)_{k} = \beta_{k}^{-1} 
     D_{k} + z C R_{\bbeta, k+1}(zA) B_{k}$.  A consequence of the uniqueness in 
     Theorem \ref{T:BL3} is that then $\Theta_{k}$ and 
     $\Theta(\cM)_{k}$ are the same up to constant right unitary 
     factor.  Hence the realization for $\Theta(\cM)_{k}$ leads to a 
     realization for $\Theta_{k}$.  We next note that one can get a 
     more direct route to this result and as a bonus get a canonical 
     functional-model formulas for the operators $A,B_{k}, C, D_{k}$ 
     for which $\Theta_{k}(z) = \beta_{k}^{-1} D_{k} + z R_{\bbeta,k+1}(zA) 
     B_{k}$ as follows.
     We note that the specialization of this result to $k=0$ 
 is closely related to Theorem 4.2 of Olofsson \cite{olaa}.
 
 \begin{theorem}  \label{T:canfuncmod}
     Suppose that we are given an inner function family $\{\Theta_{k}\}_{k \ge 
     0}$ (say $\Theta_{k}(z) \colon \cU_{k} \to \cY$ for 
     $k = 0,1, \dots$) generating the $S_{\bbeta}$-invariant subspace
     $$
     \cM = M_{\Theta} H^{2}(\{ \cU_{k}\}_{k \ge 0}) \subset 
     H^{2}_{\bbeta}(\cY).
     $$
     Then a $\bbeta$-unitary colligation family realization  
     $$
     \left\{U_{k} = \begin{bmatrix} A & B_{k} \\ C & D_{k} \end{bmatrix} 
     \colon \begin{bmatrix} \cX \\ \cU_{k} \end{bmatrix} \to 
     \begin{bmatrix} \cX \\ \cY \end{bmatrix} \right\}_{k \ge 0}
$$
for $\{\Theta_{k}\}_{k\ge 0}$ can be constructed as follows.
Take $\cX = \cM^{\perp} \subset H^{2}_{\bbeta}(\cY)$ and 
\begin{equation}   \label{funcmodcol}
U_{k} = \begin{bmatrix} \left. S_{\bbeta}^{*} \right|_{\cM^{\perp}} &
\left. \left( S_{\bbeta}^{k+1} \right)^{*} S_{\bbeta}^{k} 
M_{\Theta_{k}}\right|_{\cU_{k}} 
\\  E|_{\cM^{\perp}} & \beta_{k} \Theta_{k}(0) \end{bmatrix} \colon 
\begin{bmatrix}  \cM^{\perp} \\ \cU_{k} \end{bmatrix} \to 
    \begin{bmatrix} \cM^{\perp} \\ \cY \end{bmatrix}.
\end{equation}
  \end{theorem}
  
  \begin{proof}
      Let us set 
\begin{equation}   \label{funcmodcol'}
      \begin{bmatrix} A & B_{k} \\ C & D_{k} \end{bmatrix} = 
	  \begin{bmatrix} \left. S_{\bbeta}^{*} \right|_{\cM^{\perp}} &
\left. \left( S_{\bbeta}^{k+1} \right)^{*} S_{\bbeta}^{k} 
M_{\Theta_{k}}\right|_{\cU_{k}} 
\\  E|_{\cM^{\perp}} & \beta_{k} \Theta_{k}(0) \end{bmatrix}.
\end{equation}
Before commencing the proof, let us point out some key identities.  Let us 
write out the Taylor series for $\Theta_{k}$ as
$$
   \Theta_{k}(z) = \sum_{j=0}^{\infty} \Theta_{k,j} z^{j}.
$$
Then simple applications of formula \eqref{3.1a} for the action of 
$S_{\bbeta}^{*k}$ gives us the formulas
\begin{align}
  &  \left( (S_{\bbeta}^{k})^{*} S_{\bbeta}^{k} M_{\Theta_{k}} u 
    \right) (z) = \sum_{j=0}^{\infty} \frac{\beta_{j+k}}{\beta_{j}} 
    (\Theta_{k,j} u) z^{j}, \label{verify1} \\
    & \left( B_{k} u\right)(z) = \left( (S_{\bbeta}^{k+1})^{*} 
    S_{\bbeta}^{k} M_{\Theta_{k}} u \right)(z) = \sum_{j=0}^{\infty} 
    \frac{ \beta_{j+k+1}}{\beta_{j}} (\Theta_{k,j+1} u) z^{j}. \label{Bk}
\end{align}

As a first step toward the proof of Theorem \ref{T:canfuncmod},
we verify that $B_{k}$ maps $\cU_{k}$ into $\cM^{\perp}$.  
Indeed, if $f \in \cM$ and $u \in \cU_{k}$, then
$$
\left\langle f, \left( S_{\bbeta}^{k+1}\right)^{*} S_{\bbeta}^{k} 
M_{\Theta_{k}} u \right\rangle_{H^{2}_{\bbeta}(\cY)} = 
\left\langle S_{\bbeta}^{k+1} f, \, S_{\bbeta}^{k} M_{\Theta_{k}} u 
\right\rangle_{H^{2}_{\bbeta}(\cY)} = 0
$$
since $S_{\bbeta}^{k} M_{\Theta_{k}} \cU_{k} \perp S_{\bbeta}^{k+1} 
\cM$ by one of the defining properties of $\{\Theta_{k}\}_{k \ge 0}$ 
being an inner function family.

\smallskip

To verify the weighted isometry property \eqref{isom} of $U_{k}$, it 
suffices to verify the three pieces
\begin{align}
  &  A^{*} \Gr^{(k+1)}_{\bbeta,C,A} A + \beta_{k}^{-1} C^{*}C = 
    \Gr^{(k)}_{\bbeta,C,A}, \label{check1} \\
  & A^{*} \Gr^{(k+1)}_{\bbeta,C,A} B_{k} + \beta_{k}^{-1} C^{*} D_{k} 
  = 0, \label{check2} \\
  & B_{k}^{*} \Gr^{(k+1)}_{\bbeta,C,A} B_{k} + \beta_{k}^{-1} 
  D_{k}^{*} D_{k} = I_{\cU_{k}}.  \label{check3}
\end{align}

To check \eqref{check1}, we first note the identity
\begin{equation}   \label{id}
    S_{\bbeta} \Gr^{(k+1)}_{\bbeta,E,S_{\bbeta}^{*}} S_{\bbeta}^{*} + 
    \beta_{k}^{-1} E^{*} E = \Gr^{(k)}_{\bbeta, E, S_{\bbeta}^{*}}
\end{equation}
which is a consequence of Lemma \ref{L:3.1} (2) and identity 
\eqref{7.3} in Proposition \ref{P:3.0}.
    Due to the $S_{\bbeta}^{*}$-invariance of the subspace 
    $\cM^{\perp}$, we see that simple compression of the identity 
    \eqref{id} to the subspace $\cM^{\perp}$ gives us the identity 
    \eqref{check1}.
 We next note that the identity \eqref{check2} is equivalent to the 
validity of 
\begin{equation}   \label{check2'}
    \left\langle \begin{bmatrix}  \Gr^{(k+1)}_{\bbeta, E, 
    S_{\bbeta}^{*}} S_{\bbeta}^{*}f \\ f(0) \end{bmatrix},  \, \begin{bmatrix} 
    B_{k} \\ \beta_{k}^{-1} D_{k} \end{bmatrix} u 
    \right\rangle_{H^{2}_{\bbeta}(\cY) \oplus \cY} = 0
 \end{equation}
for all $f \in \cM^{\perp}$ and $u \in \cU_{k}$.  Let us rewrite the 
left-hand side of \eqref{check2'} as
\begin{align}
   & \left\langle \Gr^{(k+1)}_{\bbeta, E, S_{\bbeta}^{*}}  
   S_{\bbeta}^{*}f, \, ( S_{\bbeta}^{k+1})^{*} 
   S_{\bbeta}^{k} M_{\Theta_{k}} u \right\rangle_{H^{2}_{\bbeta}(\cY)} + \left\langle E f, 
   \Theta_{k}(0) u \right\rangle_{\cY}\notag \\
   & \quad = \left\langle S_{\bbeta} \Gr^{(k+1)}_{\bbeta, E, 
   S_{\bbeta}^{*}}  S_{\bbeta}^{*}f, \, \left( S_{\bbeta}^{k}\right)^{*} S^k_{\bbeta} 
   M_{\Theta_{k}} u \right\rangle_{H^{2}_{\bbeta}(\cY)} + \left\langle f(0), 
   \Theta_{k}(0) u \right\rangle_{\cY}\notag\\
& \quad =\left\langle \Gr^{(k)}_{\bbeta, E, S_{\bbeta}^{*}}
    f, (S_{\bbeta}^{k})^{*} S_{\bbeta}^{k} M_{\Theta_{k}} u
    \right\rangle_{H^{2}_{\bbeta}(\cY)}
     - \beta_{k}^{-1} \left\langle
    Ef, E (S_{\bbeta}^{k})^{*} S_{\bbeta}^{k} M_{\Theta_{k}} u
    \right\rangle_{\cY}  \notag \\
   & \quad\qquad  + \langle f(0), \Theta_{k}(0) u \rangle_{\cY},
\label{check2''}
\end{align}
where we used identity \eqref{id} for the second step.
 From \eqref{verify1} we see that 
 $$
   E (S_{\beta}^{k})^{*} S_{\bbeta}^{k} M_{\Theta_{k}} u = \beta_{k} 
   \Theta_{k,0} u.
 $$
 Hence the last two terms in \eqref{check2''} cancel and it remains 
 to show that the first term is zero, i.e., that 
 \begin{equation}   \label{check2'''}
     \left\langle \Gr^{(k)}_{\bbeta, E, S_{\bbeta}^{*}} 
    f, (S_{\bbeta}^{k})^{*} S_{\bbeta}^{k} M_{\Theta_{k}} u 
    \right\rangle_{H^{2}_{\bbeta}(\cY)} = 0.
 \end{equation}
 
 Toward this end, we use the factorization \eqref{defRa} to see that
 \begin{align}
    & \left\langle \Gr^{(k)}_{\bbeta, E, S_{\bbeta}^{*}} 
    f, (S_{\bbeta}^{k})^{*} S_{\bbeta}^{k} M_{\Theta_{k}} u 
    \right\rangle_{H^{2}_{\bbeta}(\cY)}   \notag \\
    & \quad = 
    \left\langle S_{\bbeta}^{k} \Ob^{(k)}_{\bbeta, E, S_{\bbeta}^{*}} 
    f, \, S_{\bbeta}^{k} \Ob^{(k)}_{\bbeta,E, S_{\bbeta}^{*}} 
    (S_{\bbeta}^{k})^{*} S_{\bbeta}^{k} M_{\Theta_{k}} u 
    \right\rangle_{H^{2}_{\bbeta}(\cY)}.
    \label{verify1'}
 \end{align}
 We now note that for $f \in \cM^{\perp}$,
 \begin{equation}   \label{verify2}
     S_{\bbeta}^{k} \Ob^{(k)}_{\bbeta, E, S_{\bbeta}^{*}} 
    f = S_{\bbeta}^{k} \Ob^{(k)}_{\bbeta, C, A} f.
 \end{equation}
 From \eqref{verify1} and \eqref{Obk-model} we see that
 \begin{align}
   S_{\bbeta}^{k} 
   \Ob^{(k)}_{\bbeta,E,S_{\bbeta}^{*}}(S_{\bbeta}^{k})^{*} 
   S_{\bbeta}^{k} M_{\Theta_{k}}u &= S_{\bbeta}^{k} 
   \Ob^{(k)}_{\bbeta,E, S_{\bbeta}^{*}} \left( \sum_{j=0}^{\infty} 
\frac{\beta_{j+k}}{\beta_{j}} \Theta_{k,j} u z^{j} \right) \notag  \\
& = S_{\bbeta}^{k} \sum_{j=0}^{\infty} \frac{ 
\beta_{j}}{\beta_{j+k}} \cdot \frac{ \beta_{j+k}}{\beta_{j}} 
\Theta_{k,j} u z^{j}\notag\\
& = \sum_{j=0}^{\infty} \Theta_{k,j} u z^{j+k}
= S_{\bbeta}^{k} M_{\Theta_{k}} u. \label{verify3}
\end{align}
Combining \eqref{verify1}, \eqref{verify2}, and \eqref{verify3} gives 
us that
\begin{equation}   \label{verify4}
    \left\langle \Gr^{(k)}_{\bbeta, E, S_{\bbeta}^{*}}f, 
    (S_{\bbeta}^{k})^{*} S_{\bbeta}^{k} M_{\Theta_{k}}u 
    \right\rangle_{H^{2}_{\bbeta}(\cY)} = \left\langle S_{\bbeta}^{k} 
    \Ob^{(k)}_{\bbeta,C,A}f, S_{\bbeta}^{k} M_{\Theta_{k}} u 
    \right\rangle_{H^{2}_{\bbeta}(\cY)}.
\end{equation}
 We now note that  $S_{\bbeta}^{k} \Ob^{(k)}_{\bbeta, C, A} f$ is 
 orthogonal to $S_{\bbeta}^{k} \cM$ as a consequence of the 
 decomposition \eqref{SMperp}.  As $S_{\bbeta}^{k} M_{\Theta_{k}} u 
 \in S_{\bbeta}^{k} \cM$ by definition of $\cM$ and the fact that 
 $\Theta_{k}$ is an inner function family, we conclude that the right-hand 
 side of \eqref{verify4} is zero, and \eqref{check2'''} follows as 
 needed.
 
\smallskip
 
It remains to verify \eqref{check3} which is equivalent to the validity of
    \begin{equation}   \label{check3'}
	\langle \Gr^{(k+1)}_{\bbeta,C,A} B_{k} u, B_ku 
	\rangle_{H^{2}_{\bbeta}(\cY)} + \beta_{k}^{-1} \| D_{k} u 
	\|^{2} = \| u\|^{2}
	\end{equation}
for all $u \in \cU_{k}$.  From \eqref{Bk} and \eqref{Grk-quadform} in the first 
term on the left-hand side of \eqref{check3'} we have
 \begin{align*}
     &	\langle \Gr^{(k+1)}_{\bbeta,C,A} B_{k} u, B_ku 
	\rangle_{H^{2}_{\bbeta}(\cY)}  \\
	& \quad = 
	\left\langle \Gr^{(k+1)}_{\bbeta, E, S_{\bbeta}^{*}} 
	\left( \sum_{j=0}^{\infty} \frac{ \beta_{j+k+1}}{\beta_{j}} 
	\Theta_{k,j+1} u z^{j} \right), \,  \sum_{j=0}^{\infty} \frac{ 
\beta_{j+k+1}}{\beta_{j}} \Theta_{k,j+1} u z^{j} \right\rangle_{H^{2}_{\bbeta}(\cY)}  \\
	& \quad = \sum_{j=0}^{\infty} \frac{ 
	\beta_{j}^{2}}{\beta_{j+k+1}} \left( 
	\frac{\beta_{j+k+1}}{\beta_{j}} \right)^{2} \left\langle 
	\Theta_{k,j+1} u, \, \Theta_{k,j+1} u \right\rangle_{\cY}\\
	&\quad  = \sum_{j=0}^{\infty} \beta_{j+k+1} \|\Theta_{k,j+1} 
	u\|^2_{\cY} = \sum_{j=1}^{\infty} 
	\beta_{j+k}\|\Theta_{k,j} u\|^2_{\cY}.
\end{align*}
On the other hand, from the definitions it is easily verified that
$$
  \| S_{\bbeta}^{k} M_{\Theta_{k}} u \|^{2}_{H^{2}_{\bbeta}(\cY)} = \sum_{j=0}^{\infty} 
  \beta_{j+k} \|\Theta_{k,j} u \|_{\mathcal Y}^{2}.
$$
Combining this with the result of the preceding calculation gives
$$
\langle \Gr^{(k+1)}_{\bbeta,C,A} B_{k} u, B_k u 
	\rangle_{H^{2}_{\bbeta}(\cY)} + \beta_{k} \|\Theta_{k,0} u 
	\|^{2}_{\cY} = \| S_{\bbeta}^{k} M_{\Theta_{k}}  
	u\|^{2}_{H^{2}_{\bbeta}(\cY)}.
$$
On the other hand we know that $\| S_{\beta}^{k} M_{\Theta_{k}} u 
\|_{H^{2}_{\bbeta}(\cY)}^{2} = \| u \|^{2}_{\cU}$
by another one of the defining properties of $\{ \Theta_{k} \}_{k \ge 0}$ being 
an 
inner function family. From the identity $D_{k} = \beta_{k} \Theta_{k}(0) = 
\beta_{k} \Theta_{k,0}$, we see that $\beta_{k}^{-1} \| D_{k} u 
\|^{2}_{\cY} = \beta_{k} \|\Theta_{k,0} u \|^{2}_{\cY}$.  Combining all these 
observations now gives us \eqref{check3} as wanted.  This completes 
the proof of the weighted isometry property \eqref{isom} for $U_{k}$.

\smallskip

To verify the weighted coisometry property \eqref{wghtcoisom}, we 
note that, in view of the validity of the weighted isometry property 
\eqref{isom} already checked, it suffices to show that the 
colligation matrix $U_{k}$  \eqref{funcmodcol} maps 
$\cM^{\perp} \oplus \cU_{k}$ onto $\cM^{\perp} \oplus \cY$.  Toward 
this end we first consider the special case where $k=0$.

\smallskip

Let us therefore suppose that $\left[ \begin{smallmatrix} g \\ y 
\end{smallmatrix} \right]$ is an element of $\left[ 
\begin{smallmatrix} \cM^{\perp} \\ \cY \end{smallmatrix} \right]$ 
    which is orthogonal to $U_{0} \left[ \begin{smallmatrix} 
    \cM^{\perp} \\ \cU_0 \end{smallmatrix} \right]$ in the $\left[ 
    \begin{smallmatrix} \Gr^{(1)}_{\bbeta, E, S_{\bbeta}} & 0 \\ 0 
	& I_{\cY} \end{smallmatrix} \right]$-metric on 
	$\left[ \begin{smallmatrix} \cM^{\perp} \\ \cY 
    \end{smallmatrix} \right]$, i.e., we suppose that 
\begin{equation}   \label{hyp1}
    \left\langle \begin{bmatrix} \Gr^{(1)}_{\bbeta, E, S_{\bbeta}} & 
    0 \\ 0 & I_{\cY} \end{bmatrix} \begin{bmatrix} 
    S_{\bbeta}^{*} f + S_{\bbeta}^{*} M_{\Theta_{0}}u_{0} \\ f(0) + 
    \Theta_{0}(0) u_{0} \end{bmatrix},\, \begin{bmatrix} g \\ y 
\end{bmatrix} \right\rangle_{\cM^{\perp} \oplus \cY} = 0
\end{equation}
for all $f \in \cM^{\perp}$ and $u_{0} \in \cU_{0}$.  Using the 
factorization \eqref{defRa}, we may rewrite \eqref{hyp1} in the form
\begin{equation}   \label{hyp2}
    \langle S_{\bbeta} \Ob^{(1)}_{\bbeta, E, S_{\bbeta}^{*}} 
    S_{\bbeta}^{*}(f  + M_{\Theta_{0}} u_{0}), S_{\bbeta} 
    \Ob^{(1)}_{\bbeta, E, S_{\bbeta}^{*}} g \rangle + \langle f(0) + 
    \Theta_0(0) u_{0}, y \rangle = 0.
\end{equation}
From the formulas \eqref{Obk-model} for $\Ob^{(1)}_{\bbeta, E, 
S_{\bbeta}^{*}}$ and \eqref{3.1a} for $S_{\bbeta}^{*}$, it is easy to 
verify the identity
$$
  \Ob^{(1)}_{\bbeta, E, S_{\bbeta}^{*}} S_{\bbeta}^{*} = 
  (S_{\bbeta}^{*} S_{\bbeta})^{-1} S_{\bbeta}^{*}.
$$
Note that the latter operator is well defined since $S_{\bbeta}$ is 
left invertible as a consequence of the last condition in 
\eqref{1.6}; in fact this latter operator is the Moore-Penrose left 
inverse for $S_{\bbeta}$ with action the same as the Hardy-space 
adjoint $S_{1}^{*}$ of the Hardy-space shift operator $S_{1}$:
$$
S_{1}^{*} \colon f(z) = \sum_{j=0}^{\infty} f_{j} z^{j} \mapsto 
\sum_{j=0}^{\infty} f_{j+1} z^{j}.
$$
For convenience we shall simply write $S_{1}^{*}$ for 
$(S_{\bbeta}^{*} S_{\bbeta})^{-1} S_{\bbeta}^{*}$ even when acting on 
elements of $H^{2}_{\bbeta}(\cY)$.
  We may therefore rewrite \eqref{hyp2} as
\begin{equation}  \label{hyp3}
    \langle S_{\bbeta} S_{1}^{*} (f + M_{\Theta_{0}} u_{0}), 
    S_{\bbeta} \Ob^{(1)}_{\bbeta, E, S_{\bbeta}^{*}}g 
    \rangle_{H^{2}_{\bbeta}(\cY)} + 
    \langle f(0) + \Theta_{0}(0) u_{0}, y \rangle_{\cY} = 0.
\end{equation}
Note next that 
$$
f = f(0) + S_{\bbeta} S_{1}^{*} f, \quad M_{\Theta_{0}} u_{0} = 
\Theta_{0}(0) u_{0} + S_{\bbeta} S_{1}^{*} \Theta_0 u_0
$$
and that constant functions are orthogonal to $\operatorname{Ran} 
S_{\bbeta}$ in $H^{2}_{\bbeta}(\cY)$.  We may therefore rewrite 
\eqref{hyp3} as
\begin{equation}   \label{hyp4}
 \langle f + \Theta_{0} u_{0}, y + S_{\bbeta} \Ob^{(1)}_{\bbeta, E, 
 S_{\bbeta}^{*}}g \rangle_{H^{2}_{\bbeta}(\cY)} = 0.
 \end{equation}
 We now observe that the decomposition \eqref{kM} (with $k=1$) and the fact 
 that the inner family $\{\Theta_{k}\}_{k \ge 0}$ generates $\cM= 
 M_{\Theta} H^{2}(\{\cU_{k}\}_{k \ge 0})$ gives rise to the following 
 two orthogonal decompositions for the space 
 $(S_{\bbeta}\cM)^{\perp}$:
 \begin{align}
     (S_{\bbeta} \cM)^{\perp} & = \cM^{\perp} \oplus M_{\Theta_{0}} 
     \cU_{0} \notag \\
     & = \cY \oplus \operatorname{Ran} S_{\bbeta} \cdot 
     \Ob^{(1)}_{\bbeta, E|_{\cM^{\perp}}, S_{\bbeta}^{*}|_{\cM^{\perp}}}.
     \label{decom}
 \end{align}
 From the first decomposition in \eqref{decom} we see that $f + \Theta_{0}u_{0}$ 
 (with arbitrary $f \in \cM^{\perp}$ and $u_{0} \in \cU_{0}$) is a 
 generic element of $(S_{\bbeta} \cM)^{\perp}$.  From the second 
 decomposition in \eqref{decom} we see that $y + S_{\bbeta} \Ob^{(1)}_{\bbeta, E, 
 S_{\bbeta}^{*}} g$ is an element of $(S_{\bbeta} \cM)^{\perp}$.  The 
 condition \eqref{hyp4} holding for all $f \in \cM^{\perp}$ and 
 $u_{0} \in \cU_{0}$ thus forces $y + S_{\bbeta} \Ob^{(1)}_{\bbeta, 
 E, S_{\bbeta}^{*}} g = 0$.  As this decomposition is orthogonal in 
 $H^{2}_{\bbeta}(\cY)$, we get $y=0$ and $S_{\bbeta} 
 \Ob^{(1)}_{\bbeta. E, S_{\bbeta}^{*}} g = 0$ individually.  As 
 $S_{\bbeta}$ and $\Ob^{(1)}_{\bbeta, E, S_{\bbeta}^{*}}$ are 
 injective (see Proposition \ref{P:exactbetaobs}), we conclude that 
 $g=0$ as well.  This completes the verification that $U_{k}$ 
 \eqref{funcmodcol} is onto for the special case $k=0$.
 
\smallskip

 The next step is to show that the colligation matrix $U_{k}$ as in 
 \eqref{funcmodcol} is onto for a general $k > 0$.  To this end, we 
 first view the shifted family $\{S_{\bbeta}^{k}\Theta_{k+j}\}_{j \ge 0}$ as the 
 inner family representing the shifted shift-invariant subspace 
 $\widetilde \cM^{(k)}: = S_{\bbeta}^{k} \cM$ and apply the previous 
 $k=0$ (now labeled as $j=0$) analysis to this adjusted setting.  By the 
 result of the previous paragraph we know that the adjusted $0$-level 
 colligation matrix
 $$ \widetilde U^{(k)}_{0}: =
  \begin{bmatrix} A & B_{k} \\ C & D_{k} \end{bmatrix} = 
	  \begin{bmatrix} \left. S_{\bbeta}^{*} \right|_{\widetilde 
	      \cM^{(k) \perp}} &
\left.  S_{\bbeta} ^{*} S_{\bbeta}^{k} 
M_{\Theta_{k}}\right|_{\cU_{k}} 
\\  E|_{ \widetilde \cM^{(k) \perp}} & 0 \end{bmatrix} \colon 
\begin{bmatrix} \widetilde \cM^{(k) \perp} \\ \cU_{k} \end{bmatrix} 
    \to \begin{bmatrix} \widetilde \cM^{(k) \perp} \\ \cY \end{bmatrix}
$$
is onto.  From the decomposition 
$$
 \cM = \left( \oplus_{j=0}^{k-1} S_{\bbeta}^{j} \cU_{j} \right) 
 \bigoplus S_{\bbeta}^{k} \cM,
$$
we read off the following decomposition for 
$\widetilde \cM^{(k)\perp}$:
$$
\widetilde \cM^{(k)\perp} = \cM^{\perp} \bigoplus \left( 
\oplus_{j=0}^{k-1} S_{\bbeta}^{j} \Theta_{j}  \cU_{j} \right).
$$
Note that the operator $ \left. S_{\bbeta}^{j} M_{\Theta_{j}}\right|_{\cU_{j}}$ 
is an isometric embedding of $\cU_{j}$ onto 
$S_{\bbeta}^{j} \Theta_{j} \cU_{j} \subset H^{2}_{\bbeta}(\cY)$ for 
each $j$.  This observation suggests that, in place of $\widetilde 
U^{(k)}$, we may analyze instead the operator $\widetilde U^{(k) 
\prime}$ given by
$$
 \widetilde U^{(k) \prime} = \widetilde U^{(k)} \begin{bmatrix} 
 I_{\cM^{\perp}} &  & & & \\ & M_{\Theta_{0}} & & & \\
 & & \ddots & & \\ & & & S_{\bbeta}^{k-1} M_{\Theta_{k-1}} & \\
 & & & & S_{\bbeta}^{k} M_{\Theta_{k}} \end{bmatrix} 
$$
mapping   $\cM^{\perp} \bigoplus \left(\oplus_{j=0}^{k}  
\cU_{j}\right)$ into $\cM^{\perp} \bigoplus \left(\oplus_{j=0}^{k-1} 
S^{j} \Theta_{j}  \cU_{j} \right) \bigoplus \cY$.  The fact noted 
above that $\widetilde U^{(k)}$ is onto tells us that $\widetilde 
U^{(k) \prime}$ is onto as well.
 Before proceeding further, we need to verify the following lemma.

\begin{lemma}   \label{L:tildeU0k}
    For $j$ any nonnegative integer,
    $$
    S_{\bbeta}^{*} S_{\bbeta}^{j+1} \Theta_{j+1} \cU_{j+1} \subset 
   \cM^{\perp} \oplus  S_{\bbeta}^{j} \Theta_{j} \cU_{j}.
$$
\end{lemma}

\begin{proof}[Proof of Lemma \ref{L:tildeU0k}]
    
    In view of the decomposition
$$
 H^{2}_{\bbeta}(\cY) = \cM^{\perp} \bigoplus \left( 
 \oplus_{\ell=0}^{j} S_{\bbeta}^{\ell} \Theta_{\ell} \cU_{\ell} 
 \right) \bigoplus S_{\bbeta}^{j+1} \cM
$$
it suffices to check, for any $u_{j+1} \in \cU_{j+1}$, that
\begin{align}
    & \langle S_{\bbeta}^{*} S_{\bbeta}^{j+1} \Theta_{j+1} u_{j+1}, \,
    S_{\bbeta}^{j+1}h \rangle = 0 \text{ for all } h \in \cM, 
    \label{verify1''} \\
& \langle S_{\bbeta}^{*} S_{\bbeta}^{j+1} \Theta_{j+1} u_{j+1}, \,  
S_{\bbeta}^{\ell} \Theta_{\ell} u_{\ell} \rangle = 0 \text{ for all } 
0 \le \ell < j \text{ and } u_{\ell} \in \cU_{\ell}.  \label{verify2''}
\end{align}
As for \eqref{verify1''}, compute
$$
\langle S_{\bbeta}^{*} S_{\bbeta}^{j+1} \Theta_{j+1} u_{j+1}, 
S_{\bbeta}^{j+1} h \rangle = \langle S_{\bbeta}^{j+1} \Theta_{j+1} 
u_{j+1}, S_{\bbeta}^{j+2} h \rangle = 0,
$$
since $S_{\bbeta}^{j+1} \Theta_{j+1} \cU_{j+1} \perp S_{\bbeta}^{j+2} 
\cM$.  As for \eqref{verify2''}, compute
$$
\langle S_{\bbeta}^{*} S_{\bbeta}^{j+1} \Theta_{j+1} u_{j+1}, 
S_{\bbeta}^{\ell} \Theta_{\ell} u_{\ell} \rangle = 0,
$$
since $S_{\bbeta}^{\ell + 1} \Theta_{\ell} u_{\ell} \perp 
S_{\bbeta}^{j+1} \cM \supset S_{\bbeta}^{j+1} \Theta_{j+1} \cU_{j+1}$ 
for $\ell < j$.
\end{proof}

With the aid of the result of Lemma \ref{L:tildeU0k}, we see that the 
colligation matrix $\widetilde U_{0}^{(k) \prime}$, when expanded out as a 
block matrix with respect to the decompositions $\cM^{\perp} 
\bigoplus \left( \bigoplus_{j=0}^{k} \cU_{j} \right)$ on the domain side and 
$\cM^{\perp} \bigoplus \left( \bigoplus_{j=0}^{k-1} S_{\bbeta}^{j} 
\Theta_{j} \cU_{j} \right) \bigoplus \cY$ on the range side, is given 
by  $\widetilde U^{(k) \prime} = $
\begin{equation}   \label{tildeU0k}
  \begin{bmatrix}  \left. S_{\bbeta}^{*} \right|_{\cM^{\perp}} & 
      \left. S_{\bbeta}^{*} M_{\Theta_{0}} \right|_{\cU_{0}} & 
     \left. P_{\cM^{\perp}} S_{\bbeta}^{*} S_{\bbeta}M_{\Theta_{1}} 
     \right|_{\cU_{1}} & \dots & \left.  P_{\cM^{\perp}} S_{\bbeta}^{*} 
     S_{\bbeta}^{k} M_{\Theta_{k}} \right|_{\cU_{k}} \\
     & & \left. P_{\Theta_{0} \cU_{0}}S_{\bbeta}^{*} S_{\bbeta} M_{\Theta_{1}}
     \right|_{\cU_{1}} & \ &   \\
     &  &  & \ddots &  \\
     & & &  & \left.  P_{S_{\bbeta}^{k-1}
     \Theta_{k-1} \cU_{k-1}}S_{\bbeta}^{*} S_{\bbeta}^{k} M_{\Theta_{k}}  
     \right|_{\cU_{k}}   \\
     \left. E\right|_{\cM^{\perp}} & \left. E M_{\Theta_{0}} 
     \right|_{\cU_{0}} &  &  &  \end{bmatrix}
 \end{equation}
with all unspecified entries 
equal to $0$.  As noted above, we know that $\widetilde U_{0}^{(k) 
\prime}$ is onto.  This fact combined with  
the matrix structure of $\widetilde U_{0}^{(k) \prime }$ exhibited in 
\eqref{tildeU0k} enables us to see that $P_{S_{\bbeta}^{j-1} \Theta_{j-1} 
\cU_{j-1}} S_{\bbeta}^{*}S_{\bbeta}^{j} M_{\Theta_{j}}$ maps 
$\cU_{j}$ onto $S_{\bbeta}^{j-1}M_{\Theta_{j-1}} \cU_{j-1}$ for each 
$j=1,2, \dots, k$.  A simple induction argument then shows that 
$P_{\Theta_{0} \cU_{0}}S_{\bbeta}^{*k} S_{\bbeta}^{k} M_{\Theta_{k}}$ maps $\cU_{k}$ onto 
$\Theta_{0} \cU_{0}$.  We conclude that the block matrix
$$
    \Xi^{(k)}: = \begin{bmatrix} I_{\cM^{\perp}} & 
     \left. P_{\cM^{\perp}} S_{\bbeta}^{*k} S_{\bbeta}^{k} 
     M_{\Theta_{k}}\right|_{\cU_{k}}  \\ 0 &
    \left. P_{\Theta_{0} \cU_{0}} S_{\bbeta}^{*k} S_{\bbeta}^{k} M_{\Theta_{k}}\right|_{\cU_{k}} 
    \end{bmatrix} \colon \begin{bmatrix}  \cM^{\perp} \\ \cU_{k} 
\end{bmatrix} \to \begin{bmatrix}  \cM^{\perp} \\ \Theta_{0} \cU_{0} \end{bmatrix}
$$
is onto. Let us use the identification map $M_{\Theta_{0}} \colon 
\cU_{0} \to \Theta_{0} \cU_{0}$ to introduce the adjusted $0$-level 
colligation matrix
$$
  U_{0}' = \begin{bmatrix}\left.  S_{\bbeta}^{*}\right|_{\cM^{\perp}} 
  &\left.  S_{\bbeta}^{*}\right|_{\Theta_{0} \cU_{0}} \\ \left. 
  E\right|_{\cM^{\perp}} & \left. E \right|_{\Theta_{0} \cU_{0}} 
\end{bmatrix} \colon \begin{bmatrix}  \cM^{\perp} \\ \Theta_{0} \cU_{0} \end{bmatrix} 
\to \begin{bmatrix} \cM^{\perp} \\ \cY \end{bmatrix}
$$
so that we have the factorization
$$
  U_{0} = U_{0}^{\prime} \begin{bmatrix}  I_{\cM^{\perp}} & 0 \\ 0 & 
  \left. M_{\Theta_{0}}\right|_{\cU_{0}} \end{bmatrix}
$$
(with $U_{0}$ as in \eqref{funcmodcol} with $k=0$).
By the $k=0$ case already handled, we know that $U_{0}$ 
is onto; it follows that $U_{0}^{\prime}$ is onto as well.

\smallskip

Next, we use the easily verified identity
$$
 \left. E S_{\bbeta}^{*k} S_{\bbeta} M_{\Theta_{k}} \right|_{\cU_{k}} 
 = \beta_{k} \Theta_{k}(0)
$$
together with the definitions to verify 
 the factorization formula for $U_{k}$: 
$$ U_{k} = U_{0}' \, \Xi^{(k)} \colon \begin{bmatrix} \cM^{\perp} \\ 
\cU_{k} \end{bmatrix} \to \begin{bmatrix} \cM^{\perp} \\ \cY 
\end{bmatrix}.
$$
We now have exhibited $U_{k}$ as the composition of two onto linear 
operators and it finally follows that $U_{k}$ is also onto as wanted.

\smallskip

It remains now only to check that we recover $\Theta_{k}$ via the 
realization formula \eqref{2.10} or equivalently \eqref{jul15}.
We first note that, for $u \in \cU_{k}$
\begin{align*}
    C A^{j} B_{k} u & = E S_{\bbeta}^{*j} B_{k} u  \text{ (by 
    \eqref{funcmodcol'})} \\
    & = E S_{\bbeta}^{*j} \left( \sum_{\ell=0}^{\infty} 
    \frac{\beta_{\ell+k+1}}{\beta_{\ell}} \Theta_{k, \ell+1} u 
    z^{\ell} \right)  \text{ (using \eqref{Bk})} \\
    & = E \left( \sum_{\ell = 0}^{\infty} \frac{ 
    \beta_{\ell+j}}{\beta_{\ell}} \cdot \frac{ 
    \beta_{\ell+j+k+1}}{\beta_{\ell+j}} \Theta_{k, \ell +j +1} u 
    z^{\ell} \right) \text{ (using \eqref{3.1a})}  \\
    & = \beta_{j+k+1} \Theta_{k,j+1} u.
    \end{align*}
Using this together with the formula $D_{k} = \beta_{k} \Theta_{k}(0) = \beta_{k} 
\Theta_{k,0}$ then gives us, for $\lambda \in {\mathbb D}$,  
\begin{align*}
    \beta_{k}^{-1} D_{k} u + \sum_{j=0}^{\infty} \beta_{j+k+1}^{-1} C 
    A^{j} B u \lambda^{j+1} & = 
    \Theta_{k,0} u + \sum_{j=0}^{\infty} \beta_{j+k+1}^{-1} 
    \beta_{j+k+1} \Theta_{k,j+1} \lambda^{j+1} u \\
    & = \Theta_{k}(\lambda) u
\end{align*}
and \eqref{jul15} follows as wanted.
 \end{proof}

\end{document}